\documentclass[11pt,reqno]{amsart}
\usepackage{amssymb,amsxtra,amsmath,amsfonts}
\usepackage{amsthm,euscript}
\usepackage{color}
\usepackage[linktocpage=true,colorlinks=true, linkcolor=blue, citecolor=red, urlcolor=green]{hyperref} 
\usepackage{booktabs}  
\usepackage{tikz}
\usetikzlibrary{automata}
\usetikzlibrary{arrows}
\usetikzlibrary{arrows.meta}
\usetikzlibrary{decorations.shapes}
\usepackage{dsfont}
\usepackage{subcaption}
\usepackage{mathtools}
\usepackage{textcomp}  
\usepackage{xcolor}
\usepackage{lipsum}    
\usepackage{longtable}
\usepackage{esvect}

\newtheorem{theorem}{Theorem}
\newtheorem{lemma}[theorem]{Lemma}
\newtheorem{proposition}[theorem]{Proposition}

\theoremstyle{definition}
\newtheorem{definition}[theorem]{Definition}
\newtheorem{example}[theorem]{Example}

\theoremstyle{remark}
\newtheorem{remark}[theorem]{Remark}

\numberwithin{equation}{section}
\numberwithin{theorem}{section}

\title{Poisson Pseudoalgebras}

\author{Bojko Bakalov}
\address{Department of Mathematics,
North Carolina State University,
Raleigh, NC 27695, United States}
\email{bojko\_bakalov@ncsu.edu}

\author{Ju Wang}
\address{Department of Mathematics,
University of Lynchburg,
Lynchburg, VA 24501, United States}
\email{wang\_j@lynchburg.edu}

\date{July 30, 2023}



\thanks{The first author was supported in part by a Simons Foundation grant 584741.}

\begin{document}

\begin{abstract}
For any cocommutative Hopf algebra $H$ and a left $H$-module $V$, we construct an operad $\mathcal{P}^{cl}_H(V)$, which in the special case when $H$ is the algebra of polynomials in one variable reduces to the classical operad $\mathcal{P}^{cl}(V)$ of \cite{BDHK19}. Morphisms from the Lie operad to $\mathcal{P}^{cl}(V)$ correspond to Poisson vertex algebra structures on $V$. Likewise, our operad $\mathcal{P}^{cl}_H(V)$ gives rise to the notion of a Poisson pseudoalgebra; thus extending the notion of a Lie pseudoalgebra from \cite{BDK01}. As a byproduct of our construction, we introduce two cohomology theories for Poisson pseudoalgebras, generalizing the variational and classical cohomology of Poisson vertex algebras.
\end{abstract}

\maketitle

\tableofcontents

\section{Introduction}

The notion of a Lie conformal algebra, introduced by Victor Kac \cite{K98}, provides an axiomatic description of the operator product expansion of chiral fields in conformal field theory,
and is closely related to infinite-dimensional Lie algebras such as affine Kac--Moody algebras and the Virasoro algebra.
Recall that a \emph{Lie conformal} (super)\emph{algebra} is a vector superspace $V$,
endowed with an even endomorphism $\partial\in\mathrm{End}(V)$
and a bilinear (over the ground field $\mathbb F$) $\lambda$-bracket
$[\cdot\,_\lambda\,\cdot] \colon V\times V\to V[\lambda]$
satisfying the following three axioms for all $a,b,c\in V$:
\begin{itemize}
\item
\textbf{Sesquilinearity:} 
\begin{equation}\label{sesquil}
[\partial a_\lambda b]=-\lambda[a_\lambda b] \,,
\quad
[a_\lambda \partial b]=(\lambda+\partial)[a_\lambda b]
\,,
\end{equation}
\item
\textbf{Skewsymmetry:} 
\begin{equation}\label{20170612:eq4}
[a_\lambda b]=-(-1)^{p(a)p(b)}[b_{-\lambda-\partial}a]
\,,
\end{equation}
\item
\textbf{Jacobi identity:} 
\begin{equation}\label{20170612:eq5}
[a_{\lambda}[b_\mu c]]-(-1)^{p(a)p(b)}[b_\mu [a_\lambda c]]
=[[a_\lambda b]_{\lambda+\mu}c]
\,,
\end{equation}
\end{itemize}
where $p(a)$ denotes the parity of $a$.

The theory of conformal algebras, their representations and cohomology has been developed in a series of papers; see e.g.\ \cite{DK98,CK97, CKW98, BKV99}.
The definition of conformal algebra can be generalized naturally to the case when $\lambda$ is a
vector of dimension $N$ (see \cite{BKV99}): one only needs to replace the single indeterminate
$\lambda$ with the vector $\vec\lambda = (\lambda_1,\dots,\lambda_N)$ and $\partial$ with
$\vec\partial = (\partial_1,\dots,\partial_N)$. This replacement is straightforward in \eqref{20170612:eq4} and \eqref{20170612:eq5}, while  sesquilinearity becomes coordinate-wise: 
\begin{align}\label{sesq-comp}
[\partial_i a_{\vec{\lambda}} b]&=-\lambda_i[a_{\vec{\lambda}} b]
\,,\quad
[a_{\vec{\lambda}} \partial_i b]=(\lambda_i+\partial_i)[a_{\vec{\lambda}} b]
\,, \quad 1\le i\le N \,. \\
[a_{\vec{\lambda}} b] &= -(-1)^{p(a)p(b)}[b_{-\vec{\lambda}-\vec{\partial}}a], \label{skewsymN} \\
[a_{\vec{\lambda}}[b_{\vec{\mu}}c]] &= [[a_{\vec{\lambda}}b]_{\vec{\lambda}+\vec{\mu}}c]+(-1)^{p(a)p(b)}[b_{\vec{\mu}}[a_{\vec{\lambda}}c]]. \label{jacN}
\end{align}


Taking this generalization a step further, one arrives at the notion of a \emph{Lie} (super)\emph{pseudo\-algebra} over a cocommutative Hopf algebra $H$, which was introduced in \cite{BDK01}. To be more concise, from now on, we will omit the prefix ``super" where its meaning is clear from the context. Recall that a Lie $H$-pseudoalgebra $L$ is a left $H$-module endowed with an even linear map 
\begin{equation}\label{eq:introbeta}
\beta \colon L \otimes L \to (H \otimes H) \otimes_H L \,,
\quad \beta(a\otimes b) = [a*b] \,,
\end{equation}
which is called a \emph{pseudobracket} and is subject to the following three axioms. For $a,b,c \in L$, $f,g \in H$, and $\sigma = (12) \in S_2$, we have:
\begin{itemize}
    \item \textbf{$H$-bilinearity:} \begin{align}[fa*gb] = ((f \otimes g) \otimes _H 1)[a*b],\end{align} 
    \item \textbf{Skewsymmetry:} \begin{align}[b*a] = -(-1)^{p(a)p(b)}(\sigma \otimes _H \text{id})[a*b],\label{skewsym1}\end{align}
    \item \textbf{Jacobi identity:} \begin{align}[a*[b*c]] - (-1)^{p(a)p(b)}((\sigma \otimes \text{id}) \otimes_H \text{id})[b*[a*c]] = [[a*b]*c].\label{Jacobi1}\end{align}
\end{itemize}

The details on how to compose pseudobrackets in (\ref{Jacobi1}) are given in Remark \ref{bracketcomp} below.
Unless otherwise specified, we will work over a field $\mathbb F$ of characteristic $0$ and the Hopf algebra $H$ will be cocommutative. With the above definition, a Lie conformal algebra is the same as a Lie pseudoalgebra with $H = \mathbb{F}[\partial]$, while a Lie conformal algebra in dimension $N$ is just a Lie pseudoalgebra with $H = \mathbb F[\partial_1,\dots,\partial_N]$ (see \cite{BDK01}).

The simple Lie pseudoalgebras that are finitely generated as $H$-modules have been classified in \cite{BDK01}. They turn out to be closely related to the Lie--Cartan algebras of vector fields $W_N$, $S_N$, $H_N$, and $K_N$. The representation theory of simple Lie pseudoalgebras was developed in \cite{BDK06,BDK13,BDK21,D23}, while their cohomology is related to the Gelfand--Fuchs cohomology of the Lie--Cartan algebras of vector fields \cite{F86}.

In the discussion above, Lie conformal algebras are generalized to Lie pseudoalgebras. On the other hand, one can consider the Poisson type algebras of Lie conformal algebras, which are called Poisson vertex algebras \cite{FB04, L04}.
Recall that a \emph{Poisson vertex} (super)\emph{algebra} is a (super)commutative differential algebra 
endowed with a Lie conformal (super)algebra 
$\lambda$-bracket satisfying the 
\begin{itemize}
    \item \textbf{Leibniz rule:} 
\begin{equation}\label{eq:intro11}
[a_\lambda bc]
=
[a_\lambda b]c
+
(-1)^{p(b)p(c)}
[a_\lambda c]b
\,.
\end{equation}
\end{itemize}
Poisson vertex algebras have been studied extensively in recent years; see, e.g., \cite[Sect.\ 16]{BDK01} and \cite{BDK09, DK13, K15} for applications of Poisson vertex algebras to the integrability of Hamiltonian partial differential equations, and \cite{DK13, BDK20, BDK20v, BDHKV21} for the cohomology theory of Poisson vertex algebras and vertex algebras.

In this paper, we introduce the notion of a \emph{Poisson pseudoalgebra} as a Lie pseudoalgebra $V$ equipped with a (super)commutative associative product $V\otimes V\to V$, which
is a homomorphism of $H$-modules and satisfies the following generalization of \eqref{eq:intro11}:
\begin{itemize}
    \item \textbf{Leibniz rule:} 
\begin{align}\label{leibniz1}
    [a*bc] = [a*b]c + (-1)^{p(b)p(c)}
    [a*c]b,
\end{align}
\end{itemize}
where the product $[a*b]c$ is defined in \eqref{interaction} below.

We remark that the skewsymmetry \eqref{skewsym1} in a Lie pseudoalgebra looks more symmetric than the skewsymmetry \eqref{20170612:eq4} or \eqref{skewsymN}
in a Lie conformal algebra. Likewise, in a Poisson vertex algebra, there is a right version of the Leibniz rule, which looks more complicated than the left one
(cf.\ \cite[(1.26)]{BDK09}). For Poisson pseudoalgebras, we derive a \emph{right Leibniz rule} very similar to \eqref{leibniz1}:
\begin{align}\label{leibnizr1}
    [ab*c] = a[b*c]+(-1)^{p(a)p(b)}b[a*c],
\end{align}
where the product $a[b*c]$ is defined in \eqref{right1} below. 
This allowed us to find an \emph{iterated Leibniz rule} \eqref{iterleib} for the pseudobracket of any two products (cf.\ \cite[(1.34)]{BDK09}).
Using \eqref{iterleib}, we show that the symmetric algebra $S(L)$ of any Lie pseudoalgebra $L$ has a canonical structure of a Poisson pseudoalgebra,
just as in the usual case of Lie superalgebras (corresponding to $H=\mathbb F$). Then, in Sect.\ \ref{Poisson pseudoalgebra4examples}, we present several examples of Poisson pseudoalgebras, which generalize examples from \cite{BDK20} and \cite{BDK01}.

The paper \cite{BDHK19} developed a unified approach to Lie superalgebras, Lie conformal algebras, Poisson vertex algebras, vertex algebras, and their cohomology theories. 
The main idea, due to Beilinson and Drinfeld \cite{BD04}, is to view all these different algebras as Lie algebras in certain pseudotensor categories. A \emph{pseudotensor category} is equipped with notions of $n$-linear maps for all $n\ge1$ that can be composed and have actions of the symmetric groups $S_n$. This allows one to define the notions of a Lie algebra,
module, and cohomology; see \cite{BDK01} where these ideas were developed for Lie pseudoalgebras.

For any fixed object $V$ in a pseudotensor category, the spaces $\mathcal{P}(V)(n)$ of $n$-linear maps from $V$ to $V$ form
an \emph{operad}, a notion that originated in algebraic topology in the works of Boardman--Vogt and May in the early 1970's. Since then, operads have been used extensively in algebra and
mathematical physics; see \cite{MSS02,LV12} for modern reviews. In the language of operads, an object $V$ in a pseudotensor category has a Lie algebra structure if and only if
there is a morphism of operads from the so-called Lie operad $\mathcal{L}ie$ to the operad $\mathcal{P}(V)$.

Another equivalent formulation, which is more suitable for introducing cohomology, goes back to \cite{NR67} and was developed in \cite{DK13, BDHK19}.
As a first step, to any operad $\mathcal{P}$, one assigns a $\mathbb Z$-graded Lie superalgebra
\begin{align}
W(\mathcal{P}) = \bigoplus_{n \geq -1} W_n(\mathcal{P}) \,,
\end{align}
where $W_n(\mathcal{P})$ is the set of $S_{n+1}$-invariant elements in $\mathcal{P}(n+1)$ (see \cite{T02,BDHK19} and Sect.\ \ref{UniversalLie} below).
Next, we replace $\mathcal{P}(V)$ with the operad $\mathcal{P}(\Pi V)$, where $\Pi V$ is the same vector superspace as $V$ but with reversed parity.
Then an operad morphism $\mathcal{L}ie \to \mathcal{P}(V)$ is equivalent to an odd element $X \in W_1(\mathcal{P}(\Pi V))$ such that $[X,X] =0$.
Moreover, $[X,X] =0$ implies $\mathrm{ad}_X^2 = 0$; hence, $W(\mathcal{P}(\Pi V))$ becomes a cohomology complex with the differential $\mathrm{ad}_X$.

Let us fix a left module $V$ over a cocommutative Hopf algebra $H$.
The operad corresponding to Lie $H$-pseudoalgebra structures on $V$ via the above construction is given by \cite[Sect.\ 3]{BDK01}:
\begin{align}\label{p*hn}
\mathcal{P}_H^*(n) = \mathrm{Hom}_{H^{\otimes n}} (V^{\otimes n}, H^{\otimes n} \otimes_H V) \,.
\end{align}
In the case when $H = \mathbb{F}[\partial]$, this operad is also known as $\mathcal{C}hom$ or the \emph{conformal Hom} (see \cite{DK13,BDHK19}).
Again for $H = \mathbb{F}[\partial]$, the paper \cite{BDHK19} constructs extensions $\mathcal{P}^{cl}$ and $\mathcal{P}^{ch}$ of $\mathcal{C}hom$, called the \emph{classical}
and \emph{chiral operads}, which correspond to Poisson vertex algebras and vertex algebras, respectively.
The classical operad $\mathcal{P}^{cl}$ consists of certain linear maps labeled by acyclic graphs, so that for graphs with $n$ vertices and no edges they reduce to
maps from $\mathcal{C}hom(n)$. Note that, when $H = \mathbb{F}$, the operad $\mathcal{P}_H^*$ is the $\mathcal{H}om$ operad (also known as $\mathcal{E}nd$)
defined by $\mathcal{H}om(n) = \mathrm{Hom}(V^{\otimes n},V)$.
In this case, the operad $\mathcal{P}^{cl}$ has an analogue
$\mathcal{P}^{fn}$ called the \emph{finite classical operad} \cite[Sect.\ 10.5]{BDHK19}.
The operads $\mathcal{H}om$ and $\mathcal{P}^{fn}$ correspond to Lie superalgebras and Poisson superalgebras, respectively.

In this paper, we generalize the classical operad $\mathcal{P}^{cl}$ and its finite version $\mathcal{P}^{fn}$ to the case of an arbitrary cocommutative Hopf algebra $H$. At the same time, our \emph{generalized classical operad} $\mathcal{P}^{cl}_H$ contains $\mathcal{P}_H^*$ when restricted to graphs with no edges.
We prove that Poisson pseudoalgebra structures on an $H$-module $V$ are in bijection with odd elements $X \in W_1(\mathcal{P}^{cl}_H(\Pi V))$ such that $[X,X] =0$.
This result motivates our definition of a Poisson pseudoalgebra. As an application of the construction of the generalized classical operad, we obtain a cohomology complex,
called the \emph{classical cohomology} complex of $V$. Recall that, in the Poisson vertex algebras case ($H = \mathbb{F}[\partial]$), there is another cohomology theory called
the \emph{variational cohomology} \cite{DK13}, which is related to the classical one \cite{BDHK19,BDHKV21}. We also introduce the variational cohomology of Poisson pseudoalgebras,
but the main theorem of \cite{BDHKV21}, which asserts the isomorphism of the two cohomology theories under certain conditions, remains beyond the scope of the present paper.

The paper is organized as follows. In Sect.\ \ref{sec2}, we review some preliminaries including the definition of an operad, its reformulation in terms of $\circ$-products, the universal Lie superalgebra associated to an operad, and the cooperad of graphs. In Sect.\ \ref{sec3}, we introduce the generalized classical operad $\mathcal{P}^{cl}_H$ and prove that it is indeed an operad. In Sect.\ \ref{sec4}, we define the notion of a Poisson pseudoalgebra, establish its relation to the generalized classical operad, provide examples of Poisson pseudoalgebras, and introduce the classical and variational cohomology of Poisson pseudoalgebras.

\section{Preliminaries on Operads}\label{sec2}

In this section, we review the definition and some key properties of operads. In particular, we formulate the compositions in an operad in terms of $\circ$-products. 
We also recall the universal Lie superalgebra associated to an operad, introduced in \cite{T02} (see also \cite{BDHK19}).
For more detailed reviews on operads, we refer the readers to \cite{MSS02,LV12}.


\subsection{Axioms defining an operad} \label{defofoperad}

Recall from \cite[Sect.\ 3.1]{BDHK19}
that a (linear, unital, symmetric) (super)\emph{operad} consists of a sequence $\mathcal{P}(n)$ $(n \geq 0)$ of vector superspaces, with parity denoted by $p$,
equipped with the following operations and subject to the following axioms.

\begin{itemize}
    \item \textbf{Compositions.} For $n \geq 1$ and $m_1, \dots, m_n \geq 0$, we have parity preserving linear maps
\begin{align}
\mathcal{P}(n) \otimes \mathcal{P}(m_1) \otimes\dots \otimes \mathcal{P}(m_n) &\rightarrow \mathcal{P}(M_n),
\qquad M_n := m_1 + \dots+ m_n, \notag \\
f \otimes g_1 \otimes \dots \otimes g_n &\mapsto f(g_1\otimes \dots \otimes g_n) \,. \label{gencom}
\end{align}
    \item \textbf{Associativity axiom.} The compositions satisfy:
\begin{align*}
f((g_1 \otimes \dots \otimes g_n)(h_1 \otimes \dots \otimes h_{M_n})) = (f(g_1 \otimes \dots \otimes g_n))(h_1 \otimes \dots \otimes h_{M_n}),
\end{align*}
for $h_j \in \mathcal{P}(l_j)$, $j = 1, \dots, M_n$, where in the left-hand side

    \begin{align*}
        &(g_1 \otimes \dots \otimes g_n)(h_1 \otimes \dots \otimes h_{M_n})\notag \\
        = & \pm g_1(h_1 \otimes \dots \otimes h_{M_1}) \otimes \dots \otimes g_n(h_{M_{n-1}+1} \otimes \dots \otimes h_{M_n}),
    \end{align*}
    with the Koszul--Quillen sign given by
    \begin{align*}
        \pm = (-1)^{\sum_{i<j, M_{i-1}<k\leq M_{i}}p(g_j)p(h_k)} \,,
    \end{align*}
    and $M_0:=0$, $M_k:= m_1+\cdots+m_k$.
    
    \item \textbf{Unity axiom.} There exists a unit $1 \in \mathcal{P}(1)$ such that:
    \begin{align}\label{unity}
        f(1 \otimes \dots \otimes 1) = 1(f) =f, \quad \text{for any} \; f \in \mathcal{P}(n).
    \end{align}

    \item \textbf{Permutation actions.}
    For each $n \geq 1$, there is a right action of the symmetric group $S_n$ on $\mathcal{P}(n)$:
    \begin{align*}
        \mathcal{P}(n) \times S_n \rightarrow \mathcal{P}(n) \,, \quad
        (f, \sigma) \mapsto f^\sigma .
    \end{align*}
    
    \item \textbf{Equivariance axiom.}
    For any $\sigma\in S_n$ and $\tau_1 \in S_{m_1}, \dots, \tau_n \in S_{m_n}$, we have:
    \begin{align}\label{equivariance}
        f^\sigma(g_1^{\tau_1} \otimes \dots \otimes g_n^{\tau_n}) = \bigl(f(\sigma(g_1 \otimes \dots \otimes g_n))\bigr)^{\sigma(\tau_1, \dots, \tau_n)},
    \end{align}
    where $\sigma(\tau_1, \dots, \tau_n) \in S_{M_n}$ is defined by \cite[(2.12)]{BDHK19}:
    \begin{align}\label{permutation comp}
        \sigma(\tau_1, &\dots, \tau_n)(M_{k-1}+i) := \tau_k(i) + \sum_{j=1}^{\sigma(k)-1} m_{\sigma^{-1}(j)} \,,
    \end{align}
    for $1\le k\le n$, $1\le i \le m_k$, and
    \begin{align}\label{permuvec}
        \sigma(g_1 \otimes \dots \otimes g_n) := \epsilon_g(\sigma)(g_{\sigma^{-1}(1)} \otimes \dots \otimes g_{\sigma^{-1}(n)}),
    \end{align}
    with the sign factor $\epsilon_g(\sigma)$ given again by the Koszul--Quillen rule:
    \begin{align}\label{sign1}
        \epsilon_g(\sigma) := \prod_{i<j \,:\, \sigma(i)> \sigma(j)}(-1)^{p(g_i)p(g_j)}.
    \end{align}
\end{itemize}  


For simplicity, in the rest of the paper, we will use the term operad in place of superoperad.

\subsection{$\circ$-product formulation of the axioms}\label{circleprod}
For an operad $\mathcal{P}$, one can define the $\circ_i$-product as the insertion at the $i$-th position in the composition. 
More precisely, 
for $m\ge0$, $n\ge1$ and all $1\le i\le n$,
we define the linear maps
\begin{align}\label{circleprodgeneral}
    \circ_i\colon \mathcal{P}(n) \otimes \mathcal{P}(m) \rightarrow \mathcal{P}(n+m-1),\notag \\
    Y \circ_i X = Y(1 \otimes \dots \otimes 1 \otimes X \otimes 1 \otimes \dots \otimes 1),
\end{align}
 where $X$ is inserted at the $i$-th position in the tensor product above. 
 
Note that if all $\circ_i$-products are known, then one can recover the general compositions \eqref{gencom}.
Indeed, by the associativity axiom, we have:
\begin{align}\label{compcircle}
    Y(X_1 \otimes \dots \otimes X_n) = ( \cdots ((Y \circ_1 X_1)\circ_{M_1+1} X_2) \cdots )\circ_{M_{n-1}+1}X_n,
\end{align}
for $Y \in \mathcal{P}(n)$ and $X_k \in \mathcal{P}(m_k)$ $(1\le k \le n)$.
Hence, the axioms of an operad from the previous subsection can be formulated equivalently in terms of $\circ$-products
(see, e.g., \cite{MSS02,LV12}).

Explicitly, the associativity axiom is equivalent to the following identities:
\begin{align}\label{assoeq}
    (Z \circ_i Y) \circ_j X =
    \begin{cases}
    (-1)^{p(Y)p(X)}(Z \circ_j X) \circ_{i+m-1} Y \,, \quad &\text{if }\; 1 \leq j <i,\\
    Z \circ_i(Y \circ_{j-i+1} X) \,, \quad &\text{if }\; i \leq j <i+n,\\
    (-1)^{p(Y)p(X)}(Z \circ_{j-n+1} X) \circ_i Y \,, \quad &\text{if }\; i+n \leq j <n+l,
    \end{cases}
\end{align}
for $X \in \mathcal{P}(m)$, $Y \in \mathcal{P}(n)$ and $Z \in \mathcal{P}(l)$.
The unity axiom is equivalent to:
\begin{align}\label{unityeq}
    1 \circ_1 Y = Y \circ_i 1 = Y, \; \text{ for any } \; i = 1, \dots n,
\end{align}
and the equivariance axiom is equivalent to:
\begin{align}\label{equieq}
    Y^\sigma \circ_i X^\tau = (Y \circ_{\sigma(i)} X)^{\sigma \circ_i \tau}.
\end{align}
Here the $\circ_i$-products of permutations are defined similarly as above (cf. (\ref{permutation comp})):
\begin{align}\label{sigcircitau}
    \sigma \circ_i \tau = \sigma (1, \dots, 1,\tau,1, \dots, 1) \in S_{m+n-1},
\end{align}
for $\sigma \in S_n$ and $\tau \in S_m$,
where $\tau$ is inserted at the $i$-th position.

We remark that the third identity in \eqref{assoeq} is equivalent to the first one after flipping the equality.
Another observation is that the $\circ_1$-product is associative, which is obvious from \eqref{assoeq}.

\subsection{Universal Lie superalgebra associated to an operad}\label{UniversalLie}

With $\circ$-products from the last subsection, one can construct the universal Lie superalgebra associated to an operad \cite{T02,BDHK19}.

Let $\mathcal{P}$ be an operad. For $n \geq -1$, we define $W_n$ to be the set of all permutation invariant elements in $\mathcal{P}(n+1)$, that is
\begin{align*}
    W_n = \bigl\{f \in \mathcal{P}(n+1) \;\big|\; f^\sigma = f \;\; \forall\; \sigma \in S_{n+1} \bigr\},
\end{align*}
and form the $\mathbb{Z}$-graded vector superspace
\begin{align*}
W(\mathcal{P}) = \bigoplus_{n \geq -1} W_n \,.
\end{align*}
We define a product on $W(\mathcal{P})$ as follows:
\begin{align*}
    f \square g = \sum_{\sigma \in S_{m+1,n}}(f \circ_1 g)^{\sigma^{-1}} \in W_{n+m}, \qquad
    f \in W_n, \;\; g \in W_m,
\end{align*}
where $S_{m,n}$ is the subset of $S_{m+n}$ consisting of all $(m,n)$-shuffles:
\begin{align*}
    S_{m,n} = \bigl\{ \sigma \in S_{m+n} \; \big| \; \sigma(1)< \dots < \sigma(m), \; \sigma(m+1)< \dots < \sigma(m+n) \bigr\}.
\end{align*}

Here is a quick example, which will be useful later. For $f,g \in W_1$, their $\square$-product is:
\begin{align}\label{square1}
    f \square g & = \sum_{\sigma \in S_{2,1}}(f \circ_1 g)^{\sigma^{-1}} \notag \\
    & = f \circ_1 g + (f \circ_1 g)^{(23)^{-1}} + (f \circ_1 g)^{(123)^{-1}}\notag \\
    & = f \circ_1 g + (f \circ_1 g)^{(23)} + (f \circ_1 g)^{(132)}\notag \\
    & = f \circ_1 g + f \circ_2 g + (f \circ_2 g)^{(12)}.
\end{align}
The last equality follows from the equivariance axiom (\ref{equieq}) 
and the identities $(23) = (132)(12)$ and $(132) = (12) \circ_2 (1)$. 
Another example is when $f \in W_{-1}$; then $S_{m+1,-1} = \emptyset$ and hence $f \square g = 0$ is trivial. 

Now we define a bracket on $W(\mathcal{P})$ as the commutator bracket of the $\square$-product:
\begin{align}\label{bracket}
    [f,g] = f \square g -(-1)^{p(f)p(g)}g\square f.
\end{align}

\begin{theorem}[\cite{T02}, \cite{BDHK19}]
With the bracket given by \eqref{bracket}, $W(\mathcal{P})$ is a $\mathbb{Z}$-graded Lie superalgebra,  
called the \emph{universal Lie superalgebra} associated to the operad $\mathcal{P}$.
\end{theorem}


\subsection{The cooperad of $n$-graphs}\label{ngraphs}

In this subsection, we review some properties of graphs from \cite{BDHK19} that are necessary ingredients of the new operad we will introduce later. 

For $n \geq 1$, an \emph{$n$-graph} $\Gamma$ is a set of \emph{vertices} $V(\Gamma)$ labeled from $1$ to $n$ and a collection of oriented \emph{edges} $E(\Gamma)$ between them. Let $G(n)$ denote all graphs without tadpoles (edges that start and end at the same vertex), and $G_0(n)$ be the set of all acyclic $n$-graphs, which are graphs in $G(n)$ with no (unoriented) cycles, including multiple edges. By convention, we let $G_0(0) = G(0) = \{\emptyset\}$ be the set consisting of a single element, the empty graph $\emptyset$ with no vertices.

\begin{example}\label{exg012}
For $n=1$, the only graph in $G_0(1) = G(1) =\{\bullet\}$ is a single vertex with no edges. When $n = 2$, there are $3$ graphs $\Gamma$ in $G_0(2)$$:$
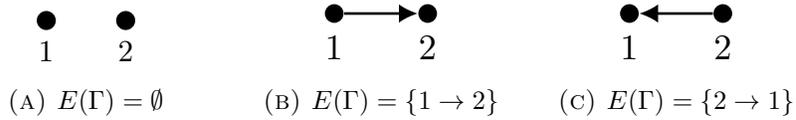
\begin{figure}[h]
     \begin{subfigure}[b]{0.3\textwidth}
          \centering
          \resizebox{1.8cm}{!}{
          \begin{tikzpicture}
          \tikzset{>={Latex[width=2mm,length=2mm]}}
          \tikzstyle{vertex} =[circle, fill = black,inner sep=0pt,minimum size=2.4mm,scale=1]
          \node[vertex,label=below:$1$](v1) at (1,0){};
          \node[vertex,label=below:$2$](v2) at (2,0){};
          \end{tikzpicture}
          }  
          \caption{$E(\Gamma)=\emptyset$}
          \label{fig:A2}
     \end{subfigure}
     \begin{subfigure}[b]{0.3\textwidth}
          \centering
          \resizebox{1.8cm}{!}{\begin{tikzpicture}
          \tikzset{>={Latex[width=2mm,length=2mm]}}
          \tikzstyle{vertex} =[circle, fill = black,inner sep=0pt,minimum size=2mm,scale=1]
          \node[vertex,label=below:$1$](v1) at (1,0){};
          \node[vertex,label=below:$2$](v2) at (2,0){};
          \draw [->,thick] (v1) to (v2);
          \end{tikzpicture}}  
          \caption{$E(\Gamma)=\{1 \rightarrow 2\}$}
          \label{fig:B2}
     \end{subfigure}
     \begin{subfigure}[b]{0.3\textwidth}
          \centering
          \resizebox{1.8cm}{!}{\begin{tikzpicture}
          \tikzset{>={Latex[width=2mm,length=2mm]}}
          \tikzstyle{vertex} =[circle, fill = black,inner sep=0pt,minimum size=2mm,scale=1]
          \node[vertex,label=below:$1$](v1) at (1,0){};
          \node[vertex,label=below:$2$](v2) at (2,0){};
          \draw [->,thick] (v2) to (v1);
          \end{tikzpicture}}  
          \caption{$E(\Gamma)=\{2 \rightarrow 1\}$}
          \label{fig:C2}
     \end{subfigure}
 \end{figure}

 \noindent
 For $n=3$, the graphs in $G_0(3)$ are the following, with any combination of directions on the edges$:$
 
 
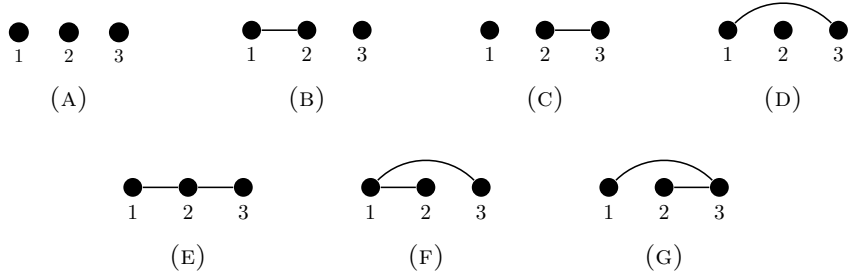
\begin{figure}[h]
     \begin{subfigure}[b]{0.24\textwidth}
          \centering
          \resizebox{1.8cm}{!}{
          \begin{tikzpicture}
          \tikzstyle{vertex} =[circle, fill = black,inner sep=0pt,minimum size=2.35mm,scale=1.7]
          \node[vertex,label=below:$1$](v1) at (1,0){};
          \node[vertex,label=below:$2$](v2) at (2,0){};
          \node[vertex,label=below:$3$](v3) at (3,0){};
          \end{tikzpicture}
          }  
          \caption{}
          \label{fig:A3}
     \end{subfigure}
     \begin{subfigure}[b]{0.24\textwidth}
          \centering
          \resizebox{1.8cm}{!}{\begin{tikzpicture}
          \tikzstyle{vertex} =[circle, fill = black,inner sep=0pt,minimum size=2mm,scale=1.7]
          \node[vertex,label=below:$1$](v1) at (1,0){};
          \node[vertex,label=below:$2$](v2) at (2,0){};
          \node[vertex,label=below:$3$](v3) at (3,0){};
          \draw [-,thick] (v1) to (v2);
          \end{tikzpicture}}  
          \caption{}
          \label{fig:B3}
     \end{subfigure}
     \begin{subfigure}[b]{0.24\textwidth}
          \centering
          \resizebox{1.8cm}{!}{\begin{tikzpicture}
          \tikzstyle{vertex} =[circle, fill = black,inner sep=0pt,minimum size=2mm,scale=1.7]
          \node[vertex,label=below:$1$](v1) at (1,0){};
          \node[vertex,label=below:$2$](v2) at (2,0){};
          \node[vertex,label=below:$3$](v3) at (3,0){};
          \draw [-,thick] (v2) to (v3);
          \end{tikzpicture}}  
          \caption{}
          \label{fig:C3}
     \end{subfigure}
     \begin{subfigure}[b]{0.24\textwidth}
          \centering
          \resizebox{1.8cm}{!}{\begin{tikzpicture}
          \tikzstyle{vertex} =[circle, fill = black,inner sep=0pt,minimum size=2mm,scale=1.7]
          \node[vertex,label=below:$1$](v1) at (1,0){};
          \node[vertex,label=below:$2$](v2) at (2,0){};
          \node[vertex,label=below:$3$](v3) at (3,0){};
          \draw [-,thick] (v1) to [out = 45, in = 135] (v3);
          \end{tikzpicture}}  
          \caption{}
          \label{fig:D3}
     \end{subfigure}\\~\\ 
     \begin{subfigure}[b]{0.24\textwidth}
          \centering
          \resizebox{1.8cm}{!}{\begin{tikzpicture}
          \tikzstyle{vertex} =[circle, fill = black,inner sep=0pt,minimum size=2mm,scale=1.7]
          \node[vertex,label=below:$1$](v1) at (1,0){};
          \node[vertex,label=below:$2$](v2) at (2,0){};
          \node[vertex,label=below:$3$](v3) at (3,0){};
          \draw [-,thick] (v1) to (v2);
          \draw [-,thick] (v2) to (v3);
          \end{tikzpicture}}  
          \caption{}
          \label{fig:E3}
     \end{subfigure}
     \begin{subfigure}[b]{0.24\textwidth}
          \centering
          \resizebox{1.8cm}{!}{\begin{tikzpicture}
          \tikzstyle{vertex} =[circle, fill = black,inner sep=0pt,minimum size=2mm,scale=1.7]
          \node[vertex,label=below:$1$](v1) at (1,0){};
          \node[vertex,label=below:$2$](v2) at (2,0){};
          \node[vertex,label=below:$3$](v3) at (3,0){};
          \draw [-,thick] (v1) to (v2);
          \draw [-,thick] (v1) to [out = 45, in = 135] (v3);
          \end{tikzpicture}}  
          \caption{}
          \label{fig:F3}
     \end{subfigure}
     \begin{subfigure}[b]{0.24\textwidth}
          \centering
          \resizebox{1.8cm}{!}{\begin{tikzpicture}
          \tikzstyle{vertex} =[circle, fill = black,inner sep=0pt,minimum size=2mm,scale=1.7]
          \node[vertex,label=below:$1$](v1) at (1,0){};
          \node[vertex,label=below:$2$](v2) at (2,0){};
          \node[vertex,label=below:$3$](v3) at (3,0){};
          \draw [-,thick] (v2) to (v3);
          \draw [-,thick] (v1) to [out = 45, in = 135] (v3);
          \end{tikzpicture}}  
          \caption{}
          \label{fig:G3}
     \end{subfigure}
 \end{figure}
\end{example}

For an $n$-tuple of positive integers $(m_1,\dots, m_n)$, and $M_i$ defined as in (\ref{gencom}), 
we have the following \emph{cocomposition} map \cite{BDHK19}: 
\begin{align}\label{cocomposition}
    \Delta^{m_1 \dots m_n} \colon G(M_n) &\rightarrow G(n) \times G(m_1) \times \dots \times G(m_n),\notag \\
    \Gamma &\mapsto \bigl( \Delta_0^{m_1 \dots m_n}(\Gamma),\Delta_1^{m_1 \dots m_n}(\Gamma),\dots, \Delta_n^{m_1 \dots m_n}(\Gamma) \bigr),
\end{align}
where:
\begin{enumerate}
    \item $\Delta_0^{m_1 \dots m_n}(\Gamma)$ is the graph obtained by clasping the $m_k$ vertices in group $k$ into a single vertex, for each $1 \leq k \leq n$;
    \item $\Delta_k^{m_1 \dots m_n}(\Gamma)$, for $1 \leq k \leq n$, is the subgraph of $\Gamma$ obtained by taking all $m_k$ vertices in group $k$ and all edges among them.
\end{enumerate}

\begin{example}\label{tengraph}
Consider the following graph $\Gamma \in G(10)$ and the partition of its vertices given by $(m_1,m_2,m_3,m_4) = (2,4,1,3)$$:$
\medskip

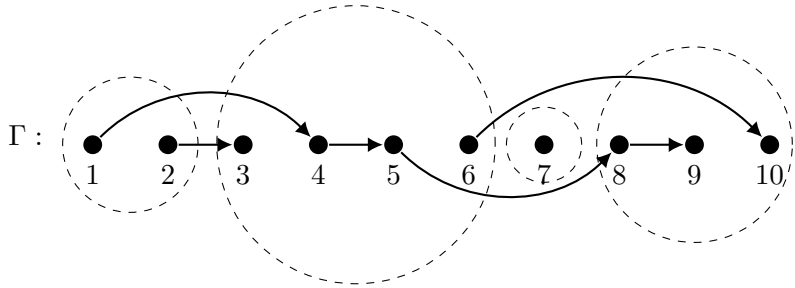
\begin{center}
\begin{tabular}{cc}
$\Gamma :$
\raisebox{-.5\height}{
\begin{tikzpicture}
\tikzset{>={Latex[width=2mm,length=2mm]}}
\tikzstyle{vertex} =[circle, fill = black,inner sep=0pt,minimum size=1mm,scale=2.5]

\node[vertex,label=below:$1$](v1) at (1,0){};
\node[vertex,label=below:$2$](v2) at (2,0){};
\node[vertex,label=below:$3$](v3) at (3,0){};
\node[vertex,label=below:$4$](v4) at (4,0){};
\node[vertex,label=below:$5$](v5) at (5,0){};
\node[vertex,label=below:$6$](v6) at (6,0){};
\node[vertex,label=below:$7$](v7) at (7,0){};
\node[vertex,label=below:$8$](v8) at (8,0){};
\node[vertex,label=below:$9$](v9) at (9,0){};
\node[vertex,label=below:$10$](v10) at (10,0){};

\draw [->,thick] (v1) to [out=45,in=135] (v4);
\draw [->,thick] (v2) to  (v3);
\draw [->,thick] (v4) to  (v5);
\draw [->,thick] (v5) to [out=-45,in=-135] (v8);
\draw [->,thick] (v6) to [out=45,in=135] (v10);
\draw [->,thick] (v8) to (v9);

\draw[dashed] (1.5,0) circle[radius=0.9cm];
\draw[dashed] (4.5,0) circle[radius=1.85cm];
\draw[dashed] (7,0) circle[radius=0.5cm];
\draw[dashed] (9,0) circle[radius=1.30cm];
\end{tikzpicture}
}
\end{tabular}

\end{center}
\medskip
where the dashed circles denote the groupings. Then, after relabeling the vertices so that they start from $1$, we have$:$

\medskip

\begin{center}
\begin{tabular}{cc}
$\Delta_0^{2413}(\Gamma) =$
\raisebox{-.45\height}{
\begin{tikzpicture}
\tikzset{>={Latex[width=2mm,length=2mm]}}
\tikzstyle{vertex} =[circle, fill = black,inner sep=0pt,minimum size=1mm,scale=2.5]

\node[vertex,label=below:$1$](v1) at (2,0){};
\node[vertex,label=below:$2$](v2) at (4,0){};
\node[vertex,label=below:$3$](v3) at (6,0){};
\node[vertex,label=below:$4$](v4) at (8,0){};

\draw [->,thick] (v1) to [out=45,in=135] (v2);
\draw [->,thick] (v1) to [out=-45,in=-135] (v2);
\draw [->,thick] (v2) to [out=45,in=135] (v4);
\draw [->,thick] (v2) to [out=-45,in=-135] (v4);
\end{tikzpicture}
}
\end{tabular}
\end{center}
\vspace{3mm}

\medskip
\begin{flushleft}
\begin{tabular}{cc}
\hspace{18mm}$\Delta_1^{2413}(\Gamma) =$
\raisebox{-.65\height}{
\begin{tikzpicture}
\tikzset{>={Latex[width=2mm,length=2mm]}}
\tikzstyle{vertex} =[circle, fill = black,inner sep=0pt,minimum size=1mm,scale=2.5]
\node[vertex,label=below:$1$](v1) at (2,0){};
\node[vertex,label=below:$2$](v2) at (4,0){};
\end{tikzpicture}
}
\end{tabular}
\end{flushleft}
\vspace{3mm}

\medskip
\medskip

\begin{center}
\begin{tabular}{cc}

$\Delta_2^{2413}(\Gamma) =$
\raisebox{-.68\height}{
\begin{tikzpicture}
\tikzset{>={Latex[width=2mm,length=2mm]}}
\tikzstyle{vertex} =[circle, fill = black,inner sep=0pt,minimum size=1mm,scale=2.5]

\node[vertex,label=below:$1$](v1) at (2,0){};
\node[vertex,label=below:$2$](v2) at (4,0){};
\node[vertex,label=below:$3$](v3) at (6,0){};
\node[vertex,label=below:$4$](v4) at (8,0){};
\draw [->,thick] (v2) to (v3);
\end{tikzpicture}
}
\end{tabular}
\end{center}
\vspace{3mm}

\medskip
\begin{flushleft}
\begin{tabular}{cc}
\hspace{18mm}$\Delta_3^{2413}(\Gamma) =$
\raisebox{-.68\height}{
\begin{tikzpicture}
\tikzset{>={Latex[width=2mm,length=2mm]}}
\tikzstyle{vertex} =[circle, fill = black,inner sep=0pt,minimum size=1mm,scale=2.5]

\node[vertex,label=below:$1$](v1) at (2,0){};
\end{tikzpicture}
}
\end{tabular}
\end{flushleft}
\vspace{3mm}

\medskip
\begin{flushleft}
\begin{tabular}{cc}
\hspace{18mm}$\Delta_4^{2413}(\Gamma) =$
\raisebox{-.68\height}{
\begin{tikzpicture}
\tikzset{>={Latex[width=2mm,length=2mm]}}
\tikzstyle{vertex} =[circle, fill = black,inner sep=0pt,minimum size=1mm,scale=2.5]

\node[vertex,label=below:$1$](v1) at (2,0){};
\node[vertex,label=below:$2$](v2) at (4,0){};
\node[vertex,label=below:$3$](v3) at (6,0){};
\draw [->,thick] (v1) to (v2);
\end{tikzpicture}
}
\end{tabular}
\end{flushleft}
\vspace{3mm}
\end{example}

Another notion we will need is the external connectedness defined below.

\begin{definition}[\cite{BDHK19}]\label{externalconn}
For an $n$-tuple of positive integers $(m_1,\dots,$ $m_n)$, let $M_i$ be defined as in \eqref{gencom}, and $\Gamma \in G(M_n)$. For a vertex $k \in \{1,\dots, M_n\}$ and a group $j \in \{1,\dots, n\}$, suppose that $k$ is in group $i$, i.e., $M_{i-1}+1 \leq k \leq M_i$. We say that $j$ is \emph{externally connected} to $k$ if there exists an unoriented path (without repeating edges) between $j$ and $i$ in $\Delta_0^{m_1 \dots m_n}(\Gamma)$ such that the edge connecting $i$ is the image under $\Delta$ of an edge in $\Gamma$ that starts or ends at $k$. 
\end{definition}

The set of all $j$'s that are externally connected to $k$ is denoted by $\mathcal{E}(k)$, a subset of $\{1,\dots,n\}$. Given a set of variables $x_1,\dots,x_n$, we define:
\begin{align}\label{excon}
    X(k) = \sum_{j \in \mathcal{E}(k)}x_j \,.
\end{align}

\begin{example}
For the graph $\Gamma \in G(10)$ in {\normalfont{Example \ref{tengraph}}}, we have:
\begin{align*}
    X(7)=X(9)=0, \quad X(k)=x_1 + x_2 + x_4 \;\text{ for all other $k \in \{1,\dots,10\}$}.
\end{align*}
\end{example}

Now we state the properties of the cocomposition map of $n$-graphs, which will be needed in Sect.\ \ref{sec3} below. In fact, $(G(n),\Delta)$ defines a cooperad, which is a dual notion of an operad \cite{LV12}, but we only focus on its properties here.

\begin{lemma}[\cite{BDHK19}]\label{edgecorrespondence}
For any positive integers $m_1, \dots, m_n$, there is a natural bijection
\begin{align*}
    \Delta\colon E(\Gamma)\rightarrow E(\Delta_0^{m_1 \dots m_n}(\Gamma)) \sqcup E(\Delta_1^{m_1 \dots m_n}(\Gamma)) \sqcup \dots \sqcup E(\Delta_n^{m_1 \dots m_n}(\Gamma)).
\end{align*}
\end{lemma}

\begin{proof}
This is true because an edge in $\Gamma$ is either contained in one of the $n$ subgraphs $\Delta_k^{m_1 \dots m_n}(\Gamma)$ $(1\le k\le n)$,
or it connects two different subgraphs. In the latter case, its image under $\Delta$ is in $\Delta_0^{m_1 \dots m_n}(\Gamma)$.
\end{proof}

\begin{lemma}[\cite{BDHK19}]\label{cyclelemma}
For an oriented cycle $C \in E(\Gamma)$ in an $n$-graph $\Gamma$, either $\Delta(C) \subset E(\Delta_k^{m_1 \dots m_n}(\Gamma))$ for some $1\le k\le n$ or $\Delta(C) \cap E(\Delta_0^{m_1 \dots m_n}(\Gamma))$ is an oriented cycle in $\Delta_0^{m_1 \dots m_n}(\Gamma)$.
\end{lemma}

\begin{proof}
If $C$ is contained in some $\Delta_k^{m_1 \dots m_n}(\Gamma)$, then $\Delta(C) \subset E(\Delta_k^{m_1 \dots m_n}(\Gamma))$. Otherwise, the set of edges from $C$ connecting different subgraphs in $\Gamma$, which is $\Delta(C) \cap E(\Delta_0^{m_1 \dots m_n}(\Gamma))$, gives an oriented cycle in $\Delta_0^{m_1 \dots m_n}(\Gamma)$.
\end{proof}


For positive integers $m_1, \dots, m_n$, let $M_i$ be defined as in \normalfont{(\ref{gencom})}, and $l_1, \dots, l_{M_n}$ be $M_n$ positive integers. For $i \in \{0,\dots, M_n\}$ and $j \in \{1, \dots, n\}$, let 
\begin{align*}
    L_i  = \sum_{n=1}^{i}l_n,\qquad
    K_j  = \sum_{m = M_{j-1}+1}^{M_j}l_m.
\end{align*}

\begin{proposition}[\textbf{Coassociativity} \cite{BDHK19}] \label{graphcoasso}
The cocomposition map $\Delta$, given by \eqref{cocomposition}, satisfies the following coassociativity conditions.
For a graph $\Gamma \in G_0(L_{M_n})$, we have$:$
\begin{enumerate}
     \item $\Delta_0^{m_1 \dots m_n}\bigl(\Delta_0^{l_1 \dots l_{M_n}}(\Gamma)\bigr) = \Delta_0^{K_1 \dots K_n}(\Gamma) \in G(n);$
     \smallskip
     
     \item $\Delta_i^{m_1 \dots m_n}\bigl(\Delta_0^{l_1 \dots l_{M_n}}(\Gamma)\bigr) = \Delta_0^{l_{M_{i-1}+1} \dots l_{M_i}}\bigl(\Delta_i^{K_1 \dots K_n}(\Gamma)\bigr) \in G(m_i)$, for $i = 1, \dots, n;$
     \smallskip
     
     \item $\Delta_{M_{i-1}+j}^{l_1 \dots l_{M_n}}(\Gamma) = \Delta_j^{l_{M_{i-1}+1} \dots l_{M_i}}\bigl(\Delta_i^{K_1 \dots K_n}(\Gamma)\bigr) \in G(l_{M_{i-1}+j})$, for $i = 1, \dots,n$ and $j =1,\dots, m_i$.
\end{enumerate}
\end{proposition}


The symmetric group $S_n$ $(n \geq 1)$ acts on the set of $n$-graphs $G(n)$ as follows. For a permutation $\sigma \in S_n$ and $\Gamma \in G(n)$, the new graph $\sigma\Gamma$ is defined by relabeling each vertex $i$ as vertex $\sigma(i)$ while not changing the edges. This action restricts to an action on $G_0(n)$. 

\begin{example}\label{permutegraph}
Suppose that $\sigma=(12)(354) \in S_5$ and $\Gamma \in G_0(5)$ is the graph below:
\medskip
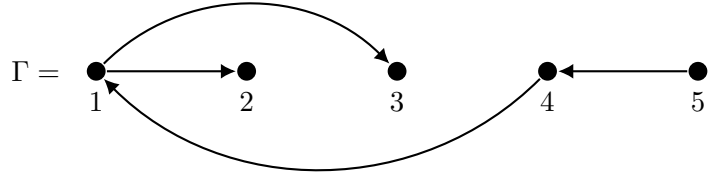
\begin{center}
\begin{tabular}{cc}
$\Gamma = $
\raisebox{-.55\height}{
\begin{tikzpicture}
\tikzset{>={Latex[width=2mm,length=2mm]}}
\tikzstyle{vertex} =[circle, fill = black,inner sep=0pt,minimum size=1mm,scale=2.5]

\node[vertex,label=below:$1$](v1) at (2,0){};
\node[vertex,label=below:$2$](v2) at (4,0){};
\node[vertex,label=below:$3$](v3) at (6,0){};
\node[vertex,label=below:$4$](v4) at (8,0){};
\node[vertex,label=below:$5$](v5) at (10,0){};

\draw [->,thick] (v1) to  (v2);
\draw [->,thick] (v1) to [out=45,in=135] (v3);
\draw [->,thick] (v4) to [out=-135,in=-45] (v1);
\draw [->,thick] (v5) to (v4);
\end{tikzpicture}
}
\end{tabular}
\end{center}
\medskip
Then $\sigma\Gamma$ is the graph: 
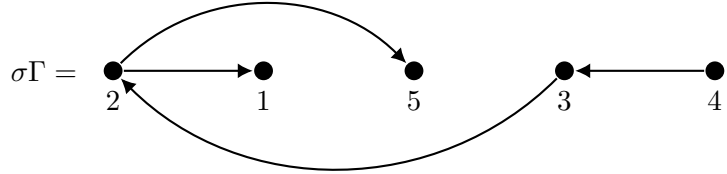
\begin{center}
\begin{tabular}{cc}
$\sigma\Gamma = $
\raisebox{-.55\height}{
\begin{tikzpicture}
\tikzset{>={Latex[width=2mm,length=2mm]}}
\tikzstyle{vertex} =[circle, fill = black,inner sep=0pt,minimum size=1mm,scale=2.5]

\node[vertex,label=below:$2$](v2) at (2,0){};
\node[vertex,label=below:$1$](v1) at (4,0){};
\node[vertex,label=below:$5$](v5) at (6,0){};
\node[vertex,label=below:$3$](v3) at (8,0){};
\node[vertex,label=below:$4$](v4) at (10,0){};

\draw [->,thick] (v2) to  (v1);
\draw [->,thick] (v2) to [out=45,in=135] (v5);
\draw [->,thick] (v3) to [out=-135,in=-45] (v2);
\draw [->,thick] (v4) to (v3);
\end{tikzpicture}
}
\end{tabular}
\end{center}
\medskip
After rearranging the vertices from $1$ to $5$, we can draw the same graph as:
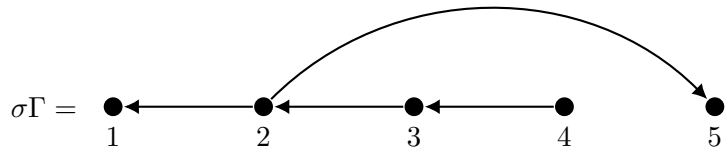
\begin{center}
\begin{tabular}{cc}
$\sigma\Gamma = $
\raisebox{-.22\height}{
\begin{tikzpicture}
\tikzset{>={Latex[width=2mm,length=2mm]}}
\tikzstyle{vertex} =[circle, fill = black,inner sep=0pt,minimum size=1mm,scale=2.5]
\node[vertex,label=below:$1$](v1) at (2,0){};
\node[vertex,label=below:$2$](v2) at (4,0){};
\node[vertex,label=below:$3$](v3) at (6,0){};
\node[vertex,label=below:$4$](v4) at (8,0){};
\node[vertex,label=below:$5$](v5) at (10,0){};

\draw [->,thick] (v2) to  (v1);
\draw [->,thick] (v2) to [out=45,in=135] (v5);
\draw [->,thick] (v3) to (v2);
\draw [->,thick] (v4) to (v3);
\end{tikzpicture}
}
\end{tabular}
\end{center}
\end{example}

The actions of the symmetric groups are compatible with the cocompositions, as expressed in the next statement.

\begin{proposition}[\textbf{Coequivariance} \cite{BDHK19}] \label{graphcoequivariance}
For any positive integers $m_1, \dots, m_n \; (n\ge1)$, permutations $\sigma \in S_n$, $\tau_i \in S_{m_i}$ $(i = 1,\dots, n)$, and a graph $\Gamma \in G_0(\sum_{i=1}^nm_i)$, we have
\begin{align*}
    & \Delta^{m_{\sigma^{-1}(1)}\dots m_{\sigma^{-1}(n)}}\bigl(\sigma(\tau_1,\dots, \tau_n) \Gamma\bigr) \notag \\ 
    = & \bigl( \sigma \Delta_0^{m_1\dots m_n}(\Gamma),\;\tau_{\sigma^{-1}(1)} \Delta_{\sigma^{-1}(1)}^{m_1\dots m_n}(\Gamma),\;\dots,\;\tau_{\sigma^{-1}(n)} \Delta_{\sigma^{-1}(n)}^{m_1\dots m_n}(\Gamma)\bigr),
\end{align*}
where $\sigma(\tau_1,\dots, \tau_n)$ is given by \eqref{permutation comp}.
\end{proposition}


\section{Generalized Classical Operad $\mathcal{P}_{H}^{cl}$}\label{sec3}
\label{firstnewresult}

In this section, we present the main result of the paper, the construction of the generalized classical operad $\mathcal{P}_{H}^{cl}(V)$
for any cocommutative Hopf algebra $H$ and a left $H$-module $V$. 

\subsection{Notation for Hopf algebras}
First, we review several identities in Hopf algebras that will be useful later. 
Given a Hopf algebra $H$, we denote by $\Delta$ its coproduct, by $S$ the antipode, and $\epsilon$ the counit. 
We will extensively use a version of Sweedler's notation (cf.\ \cite{S69,BDK01}): 
\begin{align}
    \Delta(h) & = h_{(1)} \otimes h_{(2)}, \qquad h\in H, \label{eq3.1} \\
    (S \otimes \text{id})\Delta(h) & = h_{(-1)} \otimes h_{(2)}, \label{eq3.2} \\
    (\text{id} \otimes S)\Delta(h) & = h_{(1)} \otimes h_{(-2)}. \label{eq3.3} 
\end{align}
In this notation, the coassociativity of $\Delta$ is written as: 
\begin{align}\label{coas}
    (\Delta \otimes \text{id})\Delta(h) 
    &= (h_{(1)})_{(1)} \otimes (h_{(1)})_{(2)} \otimes h_{(2)} \\
    = (\text{id} \otimes \Delta)\Delta(h) 
    &= h_{(1)} \otimes (h_{(2)})_{(1)} \otimes (h_{(2)})_{(2)} \notag \\
    &= h_{(1)} \otimes h_{(2)} \otimes h_{(3)}, \notag
\end{align}
and the axioms of the antipode and counit as:
\begin{align}
    \epsilon(h) &= h_{(-1)}h_{(2)} = h_{(1)}h_{(-2)} \label{tool0},\\
    h &= \epsilon(h_{(1)})h_{(2)} = h_{(1)}\epsilon(h_{(2)}), \label{antipax}
\end{align}
A useful consequence from them are the identities
\begin{align}\label{tool1}
    h_{(-1)}h_{(2)} \otimes h_{(3)} = 1 \otimes h = h_{(1)}h_{(-2)} \otimes h_{(3)},\\
    h_{(1)} \otimes h_{(-2)}h_{(3)} = h \otimes 1 = h_{(1)} \otimes h_{(2)}h_{(-3)}.\label{identity1}
\end{align}

We define the \emph{iterated coproducts} $\Delta^{(n)} \colon H\to H^{\otimes (n+1)}$ 
inductively by
\begin{align}\label{itcop}
\Delta^{(1)} := \Delta, \quad 
\Delta^{(n)} := (\Delta^{(n-1)} \otimes \text{id})\Delta,
\qquad n\ge2,
\end{align}
and write
\begin{align}\label{itcopn}
\Delta^{(n-1)}(h) = h_{(1)} \otimes h_{(2)} \otimes\cdots\otimes h_{(n)}.
\end{align}
It will be convenient to extend the definition of $\Delta^{(n)}$ to $n=0,-1$ by letting
\begin{align}\label{itcopext}
\Delta^{(0)} := \text{id}, \qquad
\Delta^{(-1)} := \epsilon.
\end{align}
Then \eqref{antipax} implies that \eqref{itcop} holds for $n=0,1$ as well.

From now on, $H$ will be a \emph{cocommutative} Hopf algebra (which is purely even as a superspace), so that 
\begin{align}
h_{(1)} \otimes h_{(2)} = h_{(2)} \otimes h_{(1)}, \qquad h\in H.
\end{align}

\subsection{Definition of $\mathcal{P}_{H}^{cl}$}
\label{defofnewoperad}

Let $V$ be a vector superspace with parity $p$, which is also a left $H$-module. 
When $V$ is fixed, we will write simply $\mathcal{P}^{cl}_H$ instead of $\mathcal{P}_{H}^{cl}(V)$.
For a graph $\Gamma \in G(n)$, we will denote by $s(\Gamma)$ the number of connected components of $\Gamma$.
When $\Gamma$ is fixed, we will often write $s=s(\Gamma)$, and let $\Gamma_k$ be the $k$-th connected component of $\Gamma$,
so that $\Gamma = \Gamma_1 \sqcup \dots \sqcup \Gamma_s$.

An element $Y$ of $\mathcal{P}^{cl}_H(n)$ is defined to be a collection of linear maps 
\begin{align}\label{pclhdef}
    Y^{\Gamma} \colon V^{\otimes n} \rightarrow H^{\otimes s(\Gamma)} \otimes_H V, \qquad \Gamma \in G(n),
\end{align}
satisfying the following three axioms.

\begin{itemize}
    \item \textbf{First cycle condition:} 
$Y^{\Gamma} = 0$ for any graph $\Gamma$ that contains a cycle, i.e., $\Gamma \not\in G_0(n)$.

    \item \textbf{Second cycle condition:} 
        \begin{align}\label{cycle2}
            \sum_{e \in C}Y^{\Gamma \backslash e} =0,
        \end{align}
for any oriented cycle $C \subset E(\Gamma)$ of $\Gamma$.
    \item \textbf{Componentwise $H$-linearity.} 
For any $h\in H$, $v\in V^{\otimes n}$, and 
$k=1,\dots,s(\Gamma)$, we have:
\begin{align}\label{complinear}
    Y^{\Gamma}(h \cdot_{\Gamma_k} v) = (1 \otimes \dots \otimes 1 \otimes \underbrace{h}_k \otimes 1 \otimes \dots \otimes 1) \, Y^{\Gamma}(v) \,,
\end{align}
where in the the right-hand side, $h$ appears in the $k$-th position in the tensor product.
\end{itemize}
In the left-hand side of \eqref{complinear}, $h \cdot_{\Gamma_k} v$ denotes the action of $h$ on $v=v_1 \otimes \dots \otimes v_n$
obtained by acting via the iterated coproduct on the $v_i$'s for which $i$ belongs to the vertex set $V(\Gamma_k)$ of the $k$-th connected component
$\Gamma_k$ of $\Gamma$. Explicitly, if
\begin{align}\label{vgak}
V(\Gamma_k) = \{i_{1k},\dots,i_{n_k k} \} \subset V(\Gamma) = \{1,\dots,n \}, \qquad n_k:=|V(\Gamma_k)| \,,
\end{align}
then
\begin{align*}
    h \cdot_{\Gamma_k} v
    := (1 \otimes \dots \otimes h_{(1)} \otimes \dots \otimes h_{(n_k)} \otimes \dots \otimes 1)(v_1 \otimes \dots \otimes v_n),
\end{align*}
where $h_{(1)},\dots,h_{(n_k)}$ are placed in positions $i_{1k},\dots,i_{n_k k}$, respectively.

Note that, for $n=0$, our convention is that $\mathcal{P}^{cl}_H(0)$ consists of linear maps 
$Y\colon H^{\otimes 0} \to H^{\otimes 0} \otimes_H V$, which can be identified with elements of
\begin{align}\label{pclh0}
H^{\otimes 0} \otimes_H V = \mathbb{F} \otimes_H V \cong V/H_+ V \,,
\end{align}
where $H_+ := \mathrm{Ker}\,\epsilon$ is the augmentation ideal of $H$.

In order to define the structure of an operad on $\mathcal{P}^{cl}_H$, we need to describe the composition, unity, and the action of the symmetric groups $S_n$. 

\begin{itemize}
    \item \textbf{Unity} is the identity map $1 := \text{id}_V$, which is viewed as an element $1\in\mathcal{P}^{cl}_H(1)$ such that
    $1^\bullet\colon V \to H\otimes_H V \cong V$ is the identity on $V$ for the unique graph $\bullet\in G(1)$.
    
    \item \textbf{Permutation actions.} Let $Y\in\mathcal{P}^{cl}_H(n)$ and $\sigma \in S_n$. Recall that, for $\Gamma \in G(n)$, the graph $\sigma\Gamma$ is defined by relabeling the vertices by applying $\sigma$ while not changing the edges (cf.\ Example \ref{permutegraph}). This induces a permutation $\widetilde{\sigma} \in S_s$ of the connected components of $\Gamma$ (explained in more detail below). Then we define $Y^\sigma\in\mathcal{P}^{cl}_H(n)$ by
\begin{align}\label{snaction}
    (Y^\sigma)^\Gamma(v) := (\widetilde{\sigma} \otimes_H 1)\bigl(Y^{\sigma\Gamma}(\sigma v)\bigr),
    \qquad v\in V^{\otimes n}.
\end{align}
\end{itemize}

In the right-hand side of \eqref{snaction}, we are using the action of $S_n$ on $V^{\otimes n}$ defined by (cf.\ \eqref{permuvec}, \eqref{sign1}):
\begin{align}\label{permuvec2}
\sigma(v_1 \otimes \dots \otimes v_n) &:= \epsilon_v(\sigma) v_{\sigma^{-1}(1)} \otimes \dots \otimes v_{\sigma^{-1}(n)}, \\
\epsilon_v(\sigma) &:= \prod_{i<j \,:\, \sigma(i)> \sigma(j)}(-1)^{p(v_i)p(v_j)}. \label{sign3}
\end{align}
This is consistent with assigning $v_i$ to vertex $i$ of $\Gamma$ in $Y^\Gamma(v_1 \otimes \dots \otimes v_n)$, because
the $i$-th vertex of $\sigma\Gamma$ is vertex $\sigma^{-1}(i)$ in $\Gamma$ and it corresponds to the $i$-th factor $v_{\sigma^{-1}(i)}$
of $\sigma(v_1 \otimes \dots \otimes v_n)$.

Similarly, in \eqref{snaction}, we also have the action of $S_s$ on $H^{\otimes s}$ given by:
\begin{align}\label{permuvec3}
\widetilde{\sigma}(h_1 \otimes \dots \otimes h_s) := h_{\widetilde{\sigma}^{-1}(1)} \otimes \dots \otimes h_{\widetilde{\sigma}^{-1}(s)}.
\end{align}
Now let us explain the definition of $\widetilde{\sigma}$. We fix the labeling of the connected components $\Gamma_1,\dots,\Gamma_s$ of $\Gamma$, where $s=s(\Gamma)$,
so that the lowest-labeled vertices are put in increasing order. Explicitly,
for $V(\Gamma_k)$ given by \eqref{vgak} with $i_{1k} < \cdots < i_{n_k k}$, we assume that $i_{11} < i_{12} < \cdots < i_{1s}$.
Following the same convention, we label the connected components of $\sigma\Gamma$ as $(\sigma\Gamma)_1,\dots,(\sigma\Gamma)_s$.
Notice that these are obtained by applying $\sigma$ to the connected components of $\Gamma$, but possibly in different order.
This defines a permutation $\widetilde{\sigma} \in S_s$ so that
\begin{align}\label{tildesigma}
\sigma(\Gamma_{\widetilde{\sigma}(k)}) = (\sigma\Gamma)_{k} \,, \qquad k=1,\dots,s=s(\Gamma) \,.
\end{align}
This is consistent with \eqref{permuvec3} and \eqref{complinear}, as the $k$-th connected component of $\Gamma$ corresponds to the
$k$-th tensor factor of $H^{\otimes s}$ in the image of $Y^\Gamma$.

\begin{example}\label{ex3.2}
Consider the graph $\Gamma \in G(5)$ below:

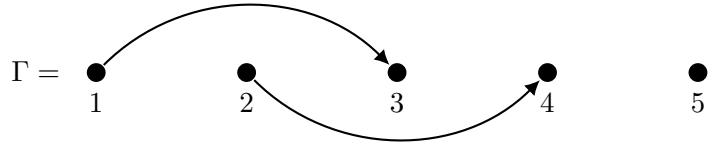
\begin{center}
\begin{tabular}{cc}
$\Gamma = $
\raisebox{-.45\height}{
\begin{tikzpicture}
\tikzset{>={Latex[width=2mm,length=2mm]}}
\tikzstyle{vertex} =[circle, fill = black,inner sep=0pt,minimum size=1mm,scale=2.5]

\node[vertex,label=below:$1$](v1) at (2,0){};
\node[vertex,label=below:$2$](v2) at (4,0){};
\node[vertex,label=below:$3$](v3) at (6,0){};
\node[vertex,label=below:$4$](v4) at (8,0){};
\node[vertex,label=below:$5$](v5) at (10,0){};
\draw [->,thick] (v1) to [out=45,in=135] (v3);
\draw [->,thick] (v2) to [out=-45,in=-135] (v4);
\end{tikzpicture}
}
\end{tabular}
\end{center}
Its connected components are $\Gamma_1,\Gamma_2,\Gamma_3$ with vertex sets
\begin{equation*}
V(\Gamma_1) = \{1,3\} \,,\quad V(\Gamma_2) = \{2,4\} \,,\quad V(\Gamma_3) = \{5\} \,.
\end{equation*}
For the permutation $\sigma = (145)(23) \in S_5$, the graph $\sigma\Gamma$ is
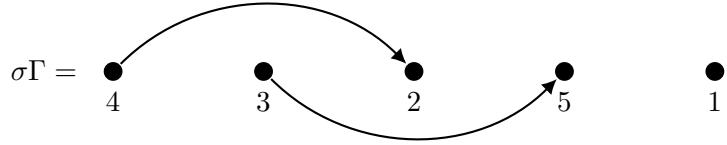
\begin{center}
\begin{tabular}{cc}
$\sigma\Gamma = $
\raisebox{-.45\height}{
\begin{tikzpicture}
\tikzset{>={Latex[width=2mm,length=2mm]}}
\tikzstyle{vertex} =[circle, fill = black,inner sep=0pt,minimum size=1mm,scale=2.5]
\node[vertex,label=below:$4$](v1) at (2,0){};
\node[vertex,label=below:$3$](v2) at (4,0){};
\node[vertex,label=below:$2$](v3) at (6,0){};
\node[vertex,label=below:$5$](v4) at (8,0){};
\node[vertex,label=below:$1$](v5) at (10,0){};
\draw [->,thick] (v1) to [out=45,in=135] (v3);
\draw [->,thick] (v2) to [out=-45,in=-135] (v4);
\end{tikzpicture}
}
\end{tabular}
\end{center}
The connected components of $\sigma\Gamma$ have vertex sets
\begin{equation*}
V((\sigma\Gamma)_1) = \{1\} \,,\quad V((\sigma\Gamma)_2) = \{2,4\} \,,\quad V((\sigma\Gamma)_3) = \{3,5\} \,.
\end{equation*}
On the other hand, applying $\sigma$ to $\Gamma_1,\Gamma_2,\Gamma_3$, we obtain the components of $\sigma\Gamma$ with vertex sets
\begin{equation*}
V(\sigma(\Gamma_1)) = \{2,4\} \,,\quad V(\sigma(\Gamma_2)) = \{3,5\} \,,\quad V(\sigma(\Gamma_3)) = \{1\} \,.
\end{equation*}
Therefore, due to \eqref{tildesigma}, the permutation $\widetilde{\sigma} \in S_3$ is equal to $(132)$.
\end{example}

\begin{lemma}\label{lemact}
For every $Y\in\mathcal{P}^{cl}_H(n)$ and $\sigma\in S_n$, we have $Y^\sigma\in\mathcal{P}^{cl}_H(n)$.
Moreover, $(Y^{\sigma_1})^{\sigma_2} = Y^{\sigma_1\sigma_2}$ for\/ $\sigma_1,\sigma_2\in S_n$.
\end{lemma}
\begin{proof}
It is clear that $Y^\sigma$ satisfies the two cycle conditions, because $Y$ does and $\sigma(C)$ is a cycle in $\sigma\Gamma$ for every cycle $C\subset E(\Gamma)$.
The componentwise $H$-linearity of $Y^\sigma$ follows from that of $Y$ and the discussion above Example \ref{ex3.2}. This proves that $Y^\sigma\in\mathcal{P}^{cl}_H(n)$.

Next, we note that for any fixed $\Gamma$, \eqref{tildesigma} implies that the map $\sigma \mapsto \widetilde{\sigma}$ is a group homomorphism $S_n\to S_s$.
Then $((Y^{\sigma_1})^{\sigma_2})^\Gamma(v) = (Y^{\sigma_1\sigma_2})^\Gamma(v)$ follows from \eqref{snaction}, 
using that $\sigma_1(\sigma_2\Gamma) = (\sigma_1\sigma_2)\Gamma$ and $\sigma_1(\sigma_2v) = (\sigma_1\sigma_2)v$.
\end{proof}

\begin{itemize}
    \item \textbf{Compositions} in the operad $\mathcal{P}^{cl}_H$ will be defined in terms of $\circ$-products (cf.\ Sect.\ \ref{circleprod}). 
\end{itemize}

We start with the $\circ_1$-product. Let $X \in \mathcal{P}^{cl}_H(m)$, $Y \in \mathcal{P}^{cl}_H(n)$, where $m,n \geq 1$. 
Given a graph $\Gamma \in G(m+n-1)$, suppose that $\Delta_1^{m1 \dots 1}(\Gamma)$ has $s$ connected components, and
$\Delta_0^{m1 \dots 1}(\Gamma)$ has $t$ connected components and $n$ vertices. We can write
\begin{align}\label{x1}
X^{\Delta_1^{m1 \dots 1}(\Gamma)}(v) &= \sum_{i} (f_{i1} \otimes \dots \otimes f_{is}) \otimes_H x_i(v), 
\quad v\in V^{\otimes m}, \\
\label{y0}
Y^{\Delta_0^{m1 \dots 1}(\Gamma)}(w) &= \sum_{j} (g_{j1} \otimes \dots \otimes g_{jt}) \otimes_H y_j(w),
\quad w\in V^{\otimes n},
\end{align}
for some linear maps $x_i\colon V^{\otimes m} \rightarrow V$ and $y_j\colon V^{\otimes n} \rightarrow V$,
and elements $f_{ik}, g_{jl} \in H$.
Then for $v\in V^{\otimes m}$, $u\in V^{\otimes (n-1)}$, we define
\begin{align}\label{circle1}
    (Y &\circ_1 X)^{\Gamma}(v \otimes u) \notag \\
    := & \sum_{i,j}\bigl((f_{i1(1)} \otimes \dots \otimes f_{is(1)} \otimes 1 \otimes \dots \otimes 1)(\Delta^{(s-1)}(g_{j1}) \otimes g_{j2} \otimes \dots \otimes g_{jt})\bigr) \notag \\ 
    & \otimes_H  y_j\bigl(x_i(v) \otimes (f_{i1(-2)} \otimes\cdots\otimes f_{is(-2)})\cdot u\bigr) \notag \\ 
    = & \sum_{i,j} (f_{i1(1)} g_{j1(1)} \otimes f_{i2(1)} g_{j1(2)} \otimes \dots \otimes f_{is(1)} g_{j1(s)} \otimes g_{j2} \otimes g_{j3} \otimes \dots \otimes g_{jt}) \notag \\ 
    & \otimes_H  y_j\bigl(x_i(v) \otimes (f_{i1(-2)} \otimes\cdots\otimes f_{is(-2)})\cdot u\bigr).
\end{align}
Let us explain the notation used in \eqref{circle1}, writing explicitly
\begin{equation}\label{vu}
v = v_1 \otimes \dots \otimes v_{m} \,, \quad u = v_{m+1} \otimes \dots \otimes v_{m+n-1} \,.
\end{equation}
Recall that the iterated coproduct $\Delta^{(s-1)}(g_{j1}) = g_{j1(1)} \otimes \dots \otimes g_{j1(s)}$ is given by \eqref{itcop}.
In the right-hand side of \eqref{circle1}, the action $(f_{i1(-2)} \otimes\cdots\otimes f_{is(-2)})\cdot u$ is defined by
letting $f_{ik(-2)}$ act via the iterated coproduct on the vectors $v_j$ with $m+1\le j \le m+n-1$ such that vertex $j$ is 
externally connected to any vertex of the $k$-th connected component of $\Delta_1^{m1 \dots 1}(\Gamma)$ $(1\le k\le s)$.
Finally, recall that $f_{ik(1)} \otimes f_{ik(-2)}$ is given by \eqref{eq3.3}. 
In the special case when no vertex $j\in\{m+1, \dots, m+n-1\}$ is externally connected to any vertex in the $k$-th connected component of $\Delta_1^{m1 \dots 1}(\Gamma)$, we get $f_{ik(1)} \otimes \epsilon(f_{ik(-2)}) = f_{ik} \otimes 1$ (cf.\ \eqref{antipax}).

To illustrate the formula, we present the following example.

\begin{example}
Let $X \in \mathcal{P}^{cl}_H(3)$, $Y \in \mathcal{P}^{cl}_H(4)$, and consider the graph $\Gamma \in G(6)$:
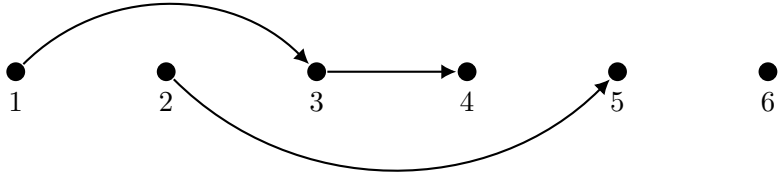
\begin{center}
\begin{tikzpicture}
\tikzset{>={Latex[width=2mm,length=2mm]}}
\tikzstyle{vertex} =[circle, fill = black,inner sep=0pt,minimum size=1mm,scale=2.5]

\node[vertex,label=below:$1$](v1) at (2,0){};
\node[vertex,label=below:$2$](v2) at (4,0){};
\node[vertex,label=below:$3$](v3) at (6,0){};
\node[vertex,label=below:$4$](v4) at (8,0){};
\node[vertex,label=below:$5$](v5) at (10,0){};
\node[vertex,label=below:$6$](v6) at (12,0){};

\draw [->,thick] (v1) to [out=45,in=135] (v3);
\draw [->,thick] (v3) to (v4);
\draw [->,thick] (v2) to [out=-45,in=-135] (v5);
\end{tikzpicture}
\end{center}
As defined in Sect.\ \ref{ngraphs}, we have:
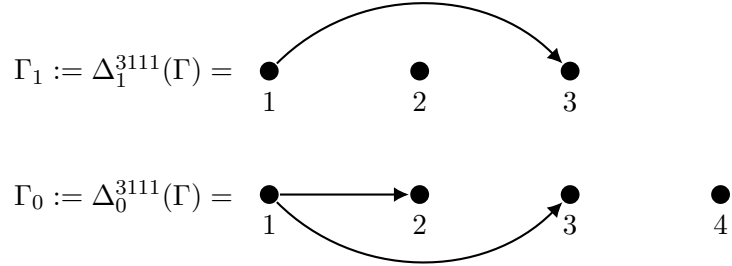
\begin{flushleft}
\begin{tabular}{cc}
\hspace{15mm}$\Gamma_1 := \Delta^{3111}_1(\Gamma) = $
\raisebox{-.3\height}{
\begin{tikzpicture}
\tikzset{>={Latex[width=2mm,length=2mm]}}
\tikzstyle{vertex} =[circle, fill = black,inner sep=0pt,minimum size=1mm,scale=2.5]
\node[vertex,label=below:$1$](v1) at (2,0){};
\node[vertex,label=below:$2$](v2) at (4,0){};
\node[vertex,label=below:$3$](v3) at (6,0){};

\draw [->,thick] (v1) to [out=45,in=135] (v3);
\end{tikzpicture}
}
\end{tabular}
\end{flushleft}

\hfill\\

\begin{flushleft}
\begin{tabular}{cc}
\hspace{15mm}$\Gamma_0 := \Delta^{3111}_0(\Gamma) = $
\raisebox{-.8\height}{
\begin{tikzpicture}
\tikzset{>={Latex[width=2mm,length=2mm]}}
\tikzstyle{vertex} =[circle, fill = black,inner sep=0pt,minimum size=1mm,scale=2.5]

\node[vertex,label=below:$1$](v1) at (2,0){};
\node[vertex,label=below:$2$](v2) at (4,0){};
\node[vertex,label=below:$3$](v3) at (6,0){};
\node[vertex,label=below:$4$](v4) at (8,0){};

\draw [->,thick] (v1) to (v2);
\draw [->,thick] (v1) to [out=-45,in=-135] (v3);
\end{tikzpicture}
}
\end{tabular}
\end{flushleft}
The number of connected components of $\Gamma_1$ and $\Gamma_0$ are both $2$. In addition, the vertex that is externally connected to (any vertex of) the first connected component of $\Gamma_1$ is $v_4$, and $v_5$ for the second. Let us write
\begin{align*}
    X^{\Gamma_1}(v_1 \otimes v_2 \otimes v_3) &= \sum_{i} (f_{i1} \otimes f_{i2}) \otimes_H x_i(v_1 \otimes v_2 \otimes v_3), \\
    Y^{\Gamma_0}(w_1 \otimes \dots \otimes w_4) &= \sum_{j} (g_{j1} \otimes g_{j2}) \otimes_H y_j(w_1 \otimes \dots \otimes w_4).
\end{align*}
Then \eqref{circle1} becomes:
\begin{align*}
     (Y \circ_1 X)^{\Gamma} &(v_1 \otimes \dots \otimes v_6) = \sum_{i,j} (f_{i1(1)} g_{j1(1)} \otimes f_{i2(1)} g_{j1(2)} \otimes g_{j2}) \notag \\
    & \otimes_H y_j\bigl(x_i(v_1 \otimes v_2 \otimes v_3) \otimes f_{i1(-2)}v_4 \otimes f_{i2(-2)}v_5 \otimes v_6\bigr).
\end{align*}
\end{example}

\begin{remark}
Eq.\ \eqref{circle1} can be extended to the case when $X \in \mathcal{P}^{cl}_H(0)$, for which $X$ can be viewed as a vector in $V$ (see \eqref{pclh0}).
If we define $\Delta^{(-1)} = \epsilon$ to be the counit map, then for $X \in \mathcal{P}^{cl}_H(0)$, $Y \in \mathcal{P}^{cl}_H(n)$, and $\Gamma \in G(n-1)$, Eq.\ \eqref{circle1} becomes:
\begin{align}
    (Y \circ_1 X)^{\Gamma}(u)
    = \sum_{j}\epsilon(g_{j1})(g_{j2} \otimes \dots \otimes g_{jt}) \otimes_H y_j(X \otimes u),
\end{align}
for $u \in V^{\otimes (n-1)}$.
\end{remark}

\begin{remark}
    In the case when $H = \mathbb{F}[\partial]$, our $\mathcal{P}_{\mathbb{F}[\partial]}^{cl}$ coincides with the \emph{classical operad} $P^{\mathrm{cl}}$ introduced in \cite[Sect.\ 10.2]{BDHK19}. 
    When $H = \mathbb{F}$, $\mathcal{P}_{\mathbb{F}}^{cl}$ is the \emph{finite} classical operad $P^{\mathrm{fn}}$ from \cite[Sect.\ 10.5]{BDHK19}. 
\end{remark}

\begin{remark}\label{pstarh}
We can realize the operad $\mathcal{P}^*_H$, defined in \eqref{p*hn} and \cite[Sect.\ 3]{BDK01}, as a suboperad of $\mathcal{P}^{cl}_H$ by considering
maps $Y \in \mathcal{P}^{cl}_H(n)$ such that $Y^\Gamma=0$ for any $n$-graph $\Gamma$ with at least one edge.
\end{remark}

To show that (\ref{circle1}) is well defined, we have the following lemmas.

\begin{lemma}\label{numbercc}
        Suppose that $\Gamma \in G_0(m+n-1)$ and  $\Delta_0^{m1 \dots 1}(\Gamma) \in G_0(n)$. If there are  $s$ connected components in $\Delta_1^{m1\dots 1}(\Gamma)$ and $t$ connected components in $\Delta_0^{m1\dots 1}(\Gamma)$, then there are \,$s+t-1$\, connected components in $\Gamma$. 
\end{lemma}
\begin{proof}
Recall that $\Delta_1^{m1\dots 1}(\Gamma)$ is the subgraph of $\Gamma$ consisting of the first $m$ vertices and all edges among them,
and $\widetilde\Gamma:=\Delta_0^{m1\dots 1}(\Gamma)$ is obtained from $\Gamma$ by clasping the first $m$ vertices into one vertex labeled $1$.
Denote the connected components of $\widetilde\Gamma$ as $\widetilde\Gamma_1,\dots,\widetilde\Gamma_t$.
Note that $\widetilde\Gamma_1$ is the image under $\Delta_0^{m1\dots 1}$ of all connected components of $\Gamma$
that are connected to one of the first $m$ vertices of $\Gamma$. 

We claim that $\widetilde\Gamma_1$ is the image of exactly $s$ connected components of $\Gamma$.
Indeed, it is at most $s$, because each vertex of $\widetilde\Gamma_1$ is connected to some connected component of $\Delta_1^{m1\dots 1}(\Gamma)$,
and there are $s$ of them. If it is strictly less than $s$, then 
there will be a path in $\Gamma$ from one of the $s$ connected components of $\Delta_1^{m1\dots 1}(\Gamma)$ to another, and this path will have to first traverse outside $\Delta_1^{m1\dots 1}(\Gamma)$ and then back to another connected component of $\Delta_1^{m1\dots 1}(\Gamma)$. 
This will produce a cycle in $\widetilde\Gamma$, which is a contradiction.

Finally, note that the other $t-1$ connected components of $\widetilde\Gamma$ are not connected to the first $m$ vertices of $\Gamma$; 
hence they are the same in $\Gamma$ and $\widetilde\Gamma$.
Altogether, we obtain that $\Gamma$ has $s+t-1$ connected components.
\end{proof}

\begin{lemma}\label{disjoint}
In \eqref{circle1}, the actions of $f_{ik(-2)}$ and $f_{il(-2)}$ are disjoint for $k \neq l$ when $\Delta_0^{m1 \dots 1}(\Gamma) \in G_0(n)$. 
\end{lemma}
\begin{proof}
Note that $1 \leq k,l \leq s \leq m$, hence $k$ and $l$ are identified as the same vertex in  $\Delta_0^{m1 \dots 1}(\Gamma)$.
If there exits a vertex $j \in \{m+1, \dots, m+n-1\}$ such that both $f_{ik(-2)}$ and $f_{il(-2)}$ act on $v_j$, 
then $j$ is externally connected to both $k$ and $l$. This gives a cycle in  $\Delta_0^{m1 \dots 1}(\Gamma)$, which contradicts with $\Delta_0^{m1 \dots 1}(\Gamma) \in G_0(n)$.
\end{proof}

Lemma \ref{disjoint} indicates that the order of the actions of $f_{ik(-2)}$'s does not matter in (\ref{circle1}). 

\begin{proposition}\label{welldefined1}
Eq.\ \eqref{circle1} defines an element $Y \circ_1 X \in \mathcal{P}^{cl}_H(m+n-1)$.
\end{proposition}
\begin{proof}
We have to check that $Y \circ_1 X$ satisfies the two cycle conditions and the componentwise $H$-linearity.
First, suppose that $\Gamma \in G(m+n-1)$ has a cycle $C$. If $C \subset \Delta_1^{m1 \dots 1}(\Gamma)$, then $X^{\Delta_1^{m1 \dots 1}(\Gamma)} =0$ since $X \in \mathcal{P}^{cl}_H(m)$; hence $(Y \circ_1 X)^{\Gamma} =0$ by \eqref{x1}, \eqref{circle1}. 
Otherwise, $C$ corresponds to a cycle $C' \subset {\Delta_0^{m1 \dots 1}(\Gamma)}$ (see Lemma \normalfont{\ref{cyclelemma}} and ignore the orientation). This makes $Y^{\Delta_0^{m1 \dots 1}(\Gamma)} =0 $ since $Y \in \mathcal{P}^{cl}_H(n)$; thus $(Y \circ_1 X)^{\Gamma} =0$ by \eqref{y0}, \eqref{circle1}. 
This shows that $Y \circ_1 X$ satisfies the first cycle condition.

For the second cycle condition, consider a cycle $C \subset \Delta_1^{m1 \dots 1}(\Gamma)$. For any edge $e \in C$, let $\Gamma \backslash e$ be the graph obtained by removing $e$ from the graph $\Gamma$. We want to evaluate $(Y \circ_1 X)^{\Gamma \backslash e}$.
Since $\Delta_0^{m1 \dots 1}(\Gamma \backslash e)$ is the same for all $e$, we can use the same expression (\ref{y0}) for all $e$. On the other hand, $X$ will be evaluated on graphs $\Delta_1^{m1 \dots 1}(\Gamma \backslash e) = \Delta_1^{m1 \dots 1}(\Gamma) \backslash e$, and removing an edge $e \in C$ from $\Delta_1^{m1 \dots 1}(\Gamma)$ does not change the number of connected components of $\Delta_1^{m1 \dots 1}(\Gamma)$. Hence, in formula \normalfont{(\ref{circle1})}, only the part involving $f_{ik}$ and $x_i$ may vary with different $e$. However, since the transformation $(\text{id} \otimes S) \Delta$ is linear on $f_{ik}$, we can evaluate it on all $X^{\Delta_1^{m1 \dots 1}(\Gamma) \backslash e}$ together before composing with $Y^{\Delta_0^{m1 \dots 1}(\Gamma)}$. Since $\sum_{e \in C}X^{\Delta_1^{m1 \dots 1}(\Gamma) \backslash e} =0$, we get $\sum_{e \in C}(Y \circ_1 X)^{\Gamma \backslash e} =0$.

If the cycle $C$ is not contained in $\Delta_1^{m1 \dots 1}(\Gamma)$, then $\Delta(C)$ is still a cycle in $\Delta_0^{m1 \dots 1}(\Gamma \backslash e)$ for any edge $e \in E(\Delta_1^{m1 \dots 1}(\Gamma))$. For such $e$, $(Y \circ_1 X)^{\Gamma \backslash e} =0$ 
since $Y^{\Delta_0^{m1 \dots 1}(\Gamma \backslash e)} =0$. Thus, 
\begin{equation*}
\sum_{e \in C}(Y \circ_1 X)^{\Gamma \backslash e} = \sum_{e \in C \cap \Delta_0^{m1 \dots 1}(\Gamma)}(Y \circ_1 X)^{\Gamma \backslash e}. 
\end{equation*}
Note that $\Delta_1^{m1 \dots 1}(\Gamma \backslash e)$ is the same for all $e \in C \cap \Delta_0^{m1 \dots 1}(\Gamma)$, which implies that in \normalfont{(\ref{circle1})} only the $g_{jk}$ and $y_j$ parts may vary with different $e \in C \cap \Delta_0^{m1 \dots 1}(\Gamma)$. The vertices externally connected to any connected component of $\Delta_1^{m1 \dots 1}(\Gamma \backslash e)$ 
stay the same for all $e \in C \cap \Delta_0^{m1 \dots 1}(\Gamma)$. As $\Delta^{(s-1)}$ is linear and $\sum_{e \in C \cap \Delta_0^{m1 \dots 1}(\Gamma)}Y^{\Delta_0^{m1 \dots 1}(\Gamma) \backslash e} =0$, we obtain $\sum_{e \in C \cap \Delta_0^{m1 \dots 1}(\Gamma)}(Y \circ_1 X)^{\Gamma \backslash e} =0$. This shows that the second cycle condition is satisfied.

It is left to check the componentwise $H$-linearity of $Y \circ_1 X$. Pick the $k$-th connected component $K$ of $\Gamma \in G(m+n-1)$. If $K$ is contained in the subgraph formed by vertices $\{m+1, \dots, m+n-1\}$, then for any $h\in H$, $v\in V^{\otimes m}$ and $u\in V^{\otimes (n-1)}$,
we have
\begin{align*}
        (Y &\circ_1 X)^{\Gamma}(h \cdot_K (v \otimes u))
        =  (Y \circ_1 X)^{\Gamma}(v \otimes h \cdot_K u).
\end{align*}
Note that $K$ remains a connected component $\widetilde K$ in $\Delta_0^{m1 \dots 1}(\Gamma)$; let it be the $k_0$-th.
Then $k_0>1$ and $k=s-1+k_0$ where $s$ is the number of connected components of $\Delta_1^{m1 \dots 1}(\Gamma)$ (see the proof of Lemma \ref{numbercc}).
From the componentwise $H$-linearity of $Y$, we get for $w_1\in V$:
\begin{align*}
Y&^{\Delta_0^{m1 \dots 1}(\Gamma)}(w_1 \otimes h \cdot_{\widetilde K} u) \notag \\
    &= (1 \otimes \dots \otimes \underbrace{h}_{k_0} \otimes \dots \otimes 1) Y^{\Delta_0^{m1 \dots 1}(\Gamma)}(w_1 \otimes u)\notag \\
    &= \sum_{j} (g_{j1} \otimes \dots \otimes h g_{jk_0} \otimes \dots \otimes g_{jt}) \otimes_H y_j(w_1 \otimes u).
\end{align*}
Note that $h \cdot_K u = h \cdot_{\widetilde K} u$, as both of these correspond to the action of $h$ on the vertices of the same connected component $K=\widetilde K$
of $\Gamma$ or $\Delta_0^{m1 \dots 1}(\Gamma)$, respectively.
Plugging the above in \eqref{circle1}, we obtain
\begin{align*}
(Y &\circ_1 X)^{\Gamma}(h \cdot_K (v \otimes u)) \\
&= \sum_{i,j} (f_{i1(1)} g_{j1(1)} \otimes \dots \otimes f_{is(1)} g_{j1(s)} \otimes g_{j2} \otimes \dots \otimes h g_{jk_0} \otimes \dots \otimes g_{jt}) \\
&\quad \otimes_H y_j\bigl(x_i(v) \otimes (f_{i1(-2)} \otimes \dots \otimes f_{is(-2)})\cdot u\bigr) \\
&= (1 \otimes \dots \otimes \underbrace{h}_{k} \otimes \dots \otimes 1) (Y \circ_1 X)^{\Gamma}(v \otimes u).
\end{align*}
This proves the componentwise $H$-linearity in the case when $K$ is contained in the subgraph with vertices $\{m+1, \dots, m+n-1\}$.

Next, consider the case when $K$ is a connected component of $\Gamma$ that intersects both subgraphs formed by the first $m$ vertices and the last $n-1$ vertices. 
Let us denote these intersections by $K_1$ and $K_2$, respectively. If $K$ is the $k$-th connected component of $\Gamma$, then $K_1$ is the $k$-th connected component of $\Delta_1^{m1 \dots 1}(\Gamma)$ (see the proof of Lemma \ref{numbercc}). 
Then, by the coassociativity of the coproduct, we have for $h\in H$, $v\in V^{\otimes m}$ and $u\in V^{\otimes (n-1)}$:
\begin{align*}
(Y \circ_1 X)^{\Gamma}\bigl(h \cdot_K (v \otimes u)\bigr)
= (Y \circ_1 X)^{\Gamma}\bigl((h_{(1)} \cdot_{K_1} v) \otimes (h_{(2)} \cdot_{K_2} u)\bigr),
\end{align*}
where $\cdot_{K_2}$ denotes the action on the $v_i$'s in $u=v_{m+1} \otimes \dots \otimes v_{m+n-1}$ corresponding to the vertices of $K_2$.
From the componentwise $H$-linearity of $X$, we get
\begin{align*}
X&^{\Delta_1^{m1 \dots 1}(\Gamma)}(h_{(1)} \cdot_{K_1} v) \\ 
&= (1 \otimes \dots \otimes \underbrace{h_{(1)}}_{k} \otimes \dots \otimes 1) X^{\Delta_1^{m1 \dots 1}(\Gamma)}(v) \\
&= \sum_{i} (f_{i1} \otimes \dots \otimes h_{(1)}f_{ik} \otimes \dots \otimes f_{is}) \otimes_H x_i(v),
\end{align*}
where $h_{(1)}$ multiplies only $f_{ik}$. Plugging this in \eqref{circle1}, we obtain:
\begin{align}\label{compHl}
(&Y \circ_1 X)^{\Gamma}\bigl(h \cdot_K (v \otimes u)\bigr) \notag \\
&= (Y \circ_1 X)^{\Gamma}\bigl((h_{(1)} \cdot_{K_1} v) \otimes (h_{(2)} \cdot_{K_2} u)\bigr) \notag \\
&= \sum_{i,j} \bigl(f_{i1(1)} g_{j1(1)} \otimes \dots \otimes (h_{(1)}f_{ik})_{(1)} g_{j1(k)} \otimes \dots \otimes f_{is(1)} g_{j1(s)} \notag \\ 
& \quad \otimes g_{j2} \otimes \dots \otimes g_{jt}\bigr) \otimes_H y_j\bigl(x_i(v) \otimes (f_{i1(-2)} \otimes\dots\otimes (h_{(1)}f_{ik})_{(-2)} \notag \\
& \quad \otimes \dots\otimes f_{is(-2)})\cdot (h_{(2)} \cdot_{K_2} u) \bigr) \notag \\
&= \sum_{i,j} \bigl(f_{i1(1)} g_{j1(1)} \otimes \dots \otimes h_{(1)} f_{ik(1)} g_{j1(k)} \otimes \dots \otimes f_{is(1)} g_{j1(s)} \notag \\ 
& \quad \otimes g_{j2} \otimes \dots \otimes g_{jt}\bigr) \otimes_H y_j\bigl(x_i(v) \otimes (f_{i1(-2)} \otimes\dots\otimes f_{ik(-2)} h_{(-2)} \notag \\
& \quad \otimes \dots\otimes f_{is(-2)})\cdot (h_{(3)} \cdot_{K_2} u) \bigr),
\end{align}
where for the last equality we used the coassociativity \eqref{coas}.
Recall that in the expression
\begin{align*}
(f_{i1(-2)} \otimes\dots\otimes f_{ik(-2)} h_{(-2)} \otimes \dots\otimes f_{is(-2)})\cdot (h_{(3)} \cdot_{K_2} u),
\end{align*}
$f_{ik(-2)} h_{(-2)}$ acts via the iterated coproduct on the vectors $v_j$ in $u$ such that vertex $j$ is 
externally connected to any vertex of the $k$-th connected component $K_1$ of $\Delta_1^{m1 \dots 1}(\Gamma)$.
These are precisely the vertices of $K_2$. Thus, by \eqref{identity1}, we can simplify the right-hand side of (\ref{compHl}) to
\begin{align*}
\sum_{i,j} &\bigl(f_{i1(1)} g_{j1(1)} \otimes \dots \otimes h f_{ik(1)} g_{j1(k)} \otimes \dots \otimes f_{is(1)} g_{j1(s)} \notag \\ 
& \quad \otimes g_{j2} \otimes \dots \otimes g_{jt}\bigr) \otimes_H y_j\bigl(x_i(v) \otimes (f_{i1(-2)} \otimes\dots\otimes f_{is(-2)}) \cdot u \bigr) \notag \\
&= (1 \otimes \dots \otimes \underbrace{h}_{k} \otimes \dots \otimes 1) (Y \circ_1 X)^{\Gamma}(v \otimes u),
\end{align*}
as desired.

Finally, in the special case when $K$ is contained in $\Delta_1^{m1 \dots 1}(\Gamma)$, we can use the same argument as above with an empty subgraph $K_2$. 
This completes the proof.
\end{proof}

Similarly to the $\circ_1$-product, we can define $\circ_k$-products for $2 \leq k \leq n$ as follows. Let $X \in \mathcal{P}^{cl}_H(m)$, $Y \in \mathcal{P}^{cl}_H(n)$, and $\Gamma \in G(m+n-1)$. Now $X$ will be evaluated on $\Delta_k^{1 \dots 1m1 \dots1}(\Gamma)$ where $m$ is at the $k$-th position. 
As before, we write for $v \in V^{\otimes m}$, $w \in V^{\otimes n}$ and some linear maps $x_i$, $y_j$:
\begin{align}\label{xatk}
    X^{\Delta_k^{1 \dots 1m1 \dots1}(\Gamma)}(v) &= \sum_{i} (f_{i1} \otimes \dots \otimes f_{is}) \otimes_H x_i(v), \\
\label{yatk0}
    Y^{\Delta_0^{1 \dots 1m1 \dots1}(\Gamma)}(w) &= \sum_{j} (g_{j1} \otimes \dots \otimes g_{jt}) \otimes_H y_j(w),
\end{align}
where now $s$ is the number of connected components of $\Delta_k^{1 \dots 1m1 \dots1}(\Gamma)$ 
and $t$ is the number of connected components of $\Delta_0^{1 \dots 1m1 \dots1}(\Gamma)$. 
We have the following analogue of Lemma \ref{numbercc}.

\begin{lemma}\label{numbercck}
        Let\/ $\Gamma \in G_0(m+n-1)$ and\/ $\Delta_0^{1 \dots 1m1 \dots 1}(\Gamma) \in G_0(n)$. If there are  $s$ connected components in $\Delta_k^{1 \dots 1m1\dots 1}(\Gamma)$ and $t$ connected components in $\Delta_0^{1 \dots 1m1\dots 1}(\Gamma)$, then there are $s+t-1$ connected components in $\Gamma$. 
\end{lemma}

Notice that when we identify the connected components of $\Delta_0^{1 \dots 1m1 \dots1}(\Gamma)$ and $\Delta_k^{1 \dots 1m1 \dots1}(\Gamma)$ as connected components in $\Gamma$
(similarly to the proof of Lemma \ref{numbercc}), they may come in different order. 
We denote by $\rho_{k}^\Gamma \in S_{s+t-1}$ the permutation identifying them, which is defined explicitly as follows.
Let us label the connected components of $\Gamma$ as $\Gamma_1,\dots,\Gamma_{s+t-1}$, following the convention above \eqref{tildesigma}.
Note that each of the connected components $K_i$ ($1\le i\le s$) of $\Delta_k^{1 \dots 1m1 \dots1}(\Gamma)$ connects to a unique connected component of $\Gamma$,
labeled as $\Gamma_{\rho_{k}^\Gamma(i)}$.
As in the proof of Lemma \ref{numbercc}, let $\widetilde\Gamma_1,\dots,\widetilde\Gamma_t$ be the connected components of $\Delta_0^{1 \dots 1m1 \dots1}(\Gamma)$.
Suppose that the vertex in $\Delta_0^{1 \dots 1m1 \dots1}(\Gamma)$ obtained by clasping the $m$ vertices $\{k, \dots, k+m-1\}$ in $\Gamma$ is contained in the $q$-th connected component $\widetilde\Gamma_q$. Then each $\widetilde\Gamma_j$ with $1\le j\le t$, $j\ne q$ remains a connected component of $\Gamma$, and we let
\begin{equation}\label{rhokga}
\widetilde\Gamma_j = \begin{cases}
\Gamma_{\rho_{k}^\Gamma(s+j)}, & 1\le j\le q-1, \\
\Gamma_{\rho_{k}^\Gamma(s-1+j)}, & q+1\le j\le t.
\end{cases}
\end{equation}
This completes the definition of $\rho_{k}^\Gamma$. It is illustrated in Example \ref{ex3.11} below.

Now the $\circ_k$-product is defined as follows:
\begin{align}\label{circleatk}
(&Y \circ_k X)^{\Gamma}(v_1 \otimes \dots \otimes v_{m+n-1}) \notag \\
&:= (-1)^{p_k^X} \sum_{i,j}\rho_{k}^\Gamma \bigl(g_{j1} \otimes g_{j2} \otimes \dots \otimes g_{j\,q-1} \otimes f_{i1(1)} g_{jq(1)} \otimes f_{i2(1)} g_{jq(2)} \notag \\
& \quad\;\; \otimes \dots \otimes f_{is(1)} g_{jq(s)} \otimes g_{j\,q+1} \otimes \dots \otimes g_{jt} \bigr) \otimes_H y_j\bigl((f_{i1(-2)} \otimes\dots\otimes f_{is(-2)}) \notag \\ 
& \quad\;\; \cdot (v_1 \otimes \dots \otimes v_{k-1} \otimes x_i(v_{k} \otimes \dots \otimes v_{k+m-1}) \otimes v_{k+m} \otimes \dots \otimes v_{m+n-1})\bigr),
\end{align}
where the sign $(-1)^{p_k^X}$ follows from the Koszul--Quillen rule:
\begin{align}\label{signcircle}
    (-1)^{p_k^X} = (-1)^{p(X) (p(v_1) + \dots + p(v_{k-1}))},
\end{align}
and the $f_{il(-2)}$'s act on the vertices $\{1,\dots,k-1\} \cup \{k+m, \dots, m+n-1\}$ in a way similar to the $\circ_1$-product case.
The rest of the notation is as in ($\ref{circle1}$); see the explanations below ($\ref{circle1}$).


\begin{example}\label{ex3.11}
We will compute $(Y \circ_4 X)^{\Gamma}$ for $X \in \mathcal{P}^{cl}_H(3)$ and $Y \in \mathcal{P}^{cl}_H(5)$, with the graph $\Gamma \in G(7)$ as shown below:

\bigskip

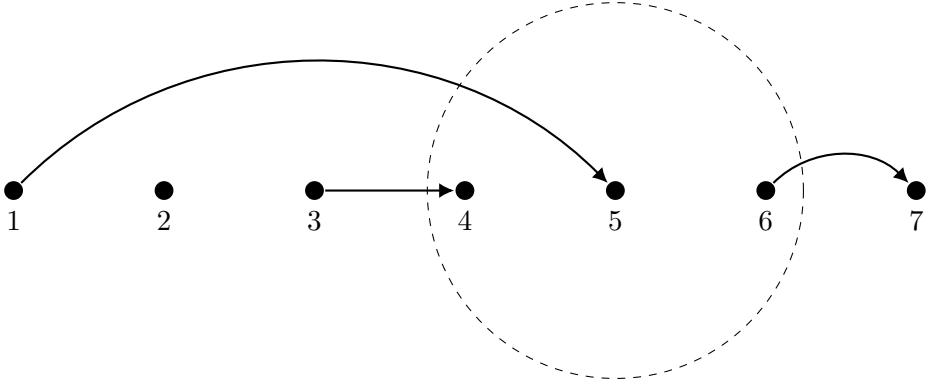
\begin{center}
\begin{tikzpicture}
\tikzset{>={Latex[width=2mm,length=2mm]}}
\tikzstyle{vertex} =[circle, fill = black,inner sep=0pt,minimum size=1mm,scale=2.5]

\node[vertex,label=below:$1$](v1) at (2,0){};
\node[vertex,label=below:$2$](v2) at (4,0){};
\node[vertex,label=below:$3$](v3) at (6,0){};
\node[vertex,label=below:$4$](v4) at (8,0){};
\node[vertex,label=below:$5$](v5) at (10,0){};
\node[vertex,label=below:$6$](v6) at (12,0){};
\node[vertex,label=below:$7$](v7) at (14,0){};

\draw [->,thick] (v1) to [out=45,in=135] (v5);
\draw [->,thick] (v3) to (v4);
\draw [->,thick] (v6) to [out=45,in=135] (v7);

\draw[dashed] (10,0) circle[radius=2.5cm];
\end{tikzpicture}
\end{center}
\bigskip
By definition of the cocomposition map, we have:

\bigskip
\bigskip

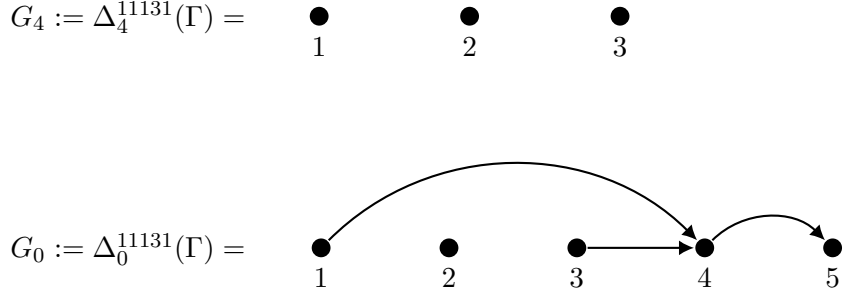
\begin{flushleft}
\begin{tabular}{cc}
\hspace{10mm}$G_4 := \Delta_4^{11131}(\Gamma) = $
\hspace{5mm}\raisebox{-.68\height}{
\begin{tikzpicture}
\tikzset{>={Latex[width=2mm,length=2mm]}}
\tikzstyle{vertex} =[circle, fill = black,inner sep=0pt,minimum size=1mm,scale=2.5]

\node[vertex,label=below:$1$](v1) at (2,0){};
\node[vertex,label=below:$2$](v2) at (4,0){};
\node[vertex,label=below:$3$](v3) at (6,0){};

\end{tikzpicture}
}
\end{tabular}
\end{flushleft}

\bigskip
\bigskip

\begin{flushleft}
\begin{tabular}{cc}
\hspace{10mm}$G_0 := \Delta_0^{11131}(\Gamma) = $ \hspace{4mm}
\raisebox{-.25\height}{
\begin{tikzpicture}
\tikzset{>={Latex[width=2mm,length=2mm]}}
\tikzstyle{vertex} =[circle, fill = black,inner sep=0pt,minimum size=1mm,scale=2.5]
\node[vertex,label=below:$1$](v1) at (1.7,0){};
\node[vertex,label=below:$2$](v2) at (3.4,0){};
\node[vertex,label=below:$3$](v3) at (5.1,0){};
\node[vertex,label=below:$4$](v4) at (6.8,0){};
\node[vertex,label=below:$5$](v5) at (8.5,0){};

\draw [->,thick] (v1) to [out=45,in=135] (v4);
\draw [->,thick] (v3) to (v4);
\draw [->,thick] (v4) to [out=45,in=135] (v5);
\end{tikzpicture}
}
\end{tabular}
\end{flushleft}
\bigskip
which gives that $s=3$, $t=2$, and $q=1$. In order to determine the permutation $\rho^\Gamma_{4} \in S_4$,
let us specify the connected components of the graphs $\Gamma$, $G_4$ and $G_0$ by listing their vertices.
Then the 1-st, 2-nd and 3-rd connected components $K_1=\{1\}$, $K_2=\{2\}$, $K_3=\{3\}$ of $G_4$ connect to the
3-rd, 1-st and 4-th connected components $\Gamma_3=\{3,4\}$, $\Gamma_1=\{1,5\}$, $\Gamma_4=\{6,7\}$ of $\Gamma$, respectively.
This gives that 
\begin{align*}
\rho^\Gamma_{4}(1)=3, \quad \rho^\Gamma_{4}(2)=1, \quad \rho^\Gamma_{4}(3)=4.
\end{align*}
The 1-st connected component $\widetilde\Gamma_1=\{1,3,4,5\}$ of $G_0$ is broken into $s$ connected components of $\Gamma$
so it does not contribute to $\rho^\Gamma_{4}$, while the 2-nd connected component $\widetilde\Gamma_2=\{2\}$ of $G_0$
is equal to the 2-nd connected component $\Gamma_2=\{2\}$ of $\Gamma$. This gives
\begin{align*}
\rho^\Gamma_{4}(s-1+2)=\rho^\Gamma_{4}(4)=2.
\end{align*}
Let us write
\begin{align*}
    X^{G_4}(v)  &= \sum_{i} (f_{i1} \otimes f_{i2} \otimes f_{i3} ) \otimes_H x_i(v),
    & v &\in V^{\otimes 3}, \\
    Y^{G_0}(w) &= \sum_{j} (g_{j1} \otimes g_{j2} ) \otimes_H y_j(w),
    & w &\in V^{\otimes 5}.
\end{align*}
Then, using \eqref{circleatk}, we find:
\begin{align*}
    (&Y \circ_4 X)^{\Gamma} \notag \\
    &= (-1)^{p_4^X} \sum_{i,j} \rho^\Gamma_{4}\bigl(f_{i1(1)} g_{j1(1)} \otimes f_{i2(1)} g_{j1(2)} \otimes f_{i3(1)} g_{j1(3)} \otimes g_{j2}\bigr) \notag \\
    & \otimes_H y_j\bigl((f_{i1(-2)} \otimes f_{i2(-2)} \otimes f_{i3(-2)}) \cdot (v_1 \otimes v_2 \otimes v_3 \otimes x_i(v_4 \otimes v_5 \otimes v_6) \otimes v_7)\bigr) \notag \\~ \notag \\
     &= (-1)^{p_4^X} \sum_{i,j} \bigl(f_{i2(1)} g_{j1(2)} \otimes g_{j2} \otimes f_{i1(1)} g_{j1(1)} \otimes f_{i3(1)} g_{j1(3)} \bigr) \notag \\
     & \otimes_H y_j\bigl(f_{i2(-2)} v_1 \otimes v_2 \otimes f_{i1(-2)}v_3 \otimes x_i(v_4 \otimes v_5 \otimes v_6) \otimes f_{i3(-2)} v_7\bigr).
\end{align*}
\end{example}

One can show that the following results are also true analogously to the proofs of Lemma \ref{disjoint} and Proposition \ref{welldefined1}. The $\circ_k$-product can be thought of as a $\circ_1$-product (possibly up to a sign change) after relabeling the vertices of $\Gamma$.

\begin{lemma}\label{disjointk}
In \eqref{circleatk}, the actions of $f_{ij(-2)}$ and $f_{il(-2)}$ are disjoint if $j \neq l$ and $\Delta_0^{1 \dots 1m1 \dots 1}(\Gamma) \in G_0(n)$. 
\end{lemma}

\begin{proposition}\label{welldefinedk}
Eq.\ \eqref{circleatk} defines an element $Y \circ_k X \in \mathcal{P}^{cl}_H(m+n-1)$.
\end{proposition}

Now we have definitions of $\circ_k$-products for all $k \geq 1$. Then, for any $Y \in \mathcal{P}(n)$ and $X_j \in \mathcal{P}(m_j)$ $(1\le j \le n)$, we define the composition $Y(X_1 \otimes \dots \otimes X_n)$ by (\ref{compcircle}). We will show in the following subsection that this definition makes sense by verifying the associativity axiom \eqref{assoeq}.

\subsection{Proof that $\mathcal{P}^{cl}_H$ is an operad}
In this subsection, we prove the following main theorem of the paper:
\begin{theorem}\label{firstmainthm}
Let $H$ be a cocommutative Hopf algebra, and $V$ be a vector superspace that is also a left $H$-module. Then the vector superspaces $\mathcal{P}^{cl}_H(n)$, $n \geq 0$, defined in \eqref{pclhdef}, with the actions of the symmetric groups $S_n$ given by \eqref{snaction} and the composition maps given by \eqref{compcircle}, \eqref{circleatk}, form an operad
called the \emph{generalized classical operad}.
\end{theorem}

To prove the theorem, we will verify (\ref{assoeq}), (\ref{unityeq}), (\ref{equieq}), which will be done in the following lemmas. The unity axiom (\ref{unityeq}) is obvious, so we focus on the associativity and equivariance axioms. Moreover, since the first and the third equations in (\ref{assoeq}) are equivalent, we only need to prove the first and second. 
Throughout the rest of this subsection, let
\begin{align*}
X \in \mathcal{P}^{cl}_H(m), \quad Y \in \mathcal{P}^{cl}_H(n), \quad Z \in \mathcal{P}^{cl}_H(l).
\end{align*}

\begin{lemma}\label{operadeq1}
We have $(Z \circ_i Y) \circ_j X =(-1)^{p(Y)p(X)}(Z \circ_j X) \circ_{i+m-1} Y$ for $1 \leq j < i$, i.e., the first identity in \eqref{assoeq} holds.
\end{lemma}

\begin{proof}
By Proposition \ref{welldefinedk}, both $(Z \circ_i Y) \circ_j X$ and $(Z \circ_j X) \circ_{i+m-1} Y$ are elements of $P^{cl}_H(m+n+l-2)$. 
Given a graph $\Gamma \in G(m+n+l-2)$, we introduce the following notation:
\begin{align*}
\Gamma_2 &:= \Delta_j^{1 \dots 1m1 \dots 1}(\Gamma) \in G(m), & \overline{\Gamma}_1 &:= \Delta_0^{1 \dots 1m1 \dots 1}(\Gamma) \in G(n+l-1), \\
\Gamma_1 &:= \Delta_i^{1 \dots 1n1 \dots 1}(\overline{\Gamma}_1) \in G(n), & \Gamma_0 &:= \Delta_0^{1 \dots 1n1 \dots 1}(\overline{\Gamma}_1) \in G(l),
\end{align*}
where $m$ appears at the $j$-th position and $n$ appears at the $i$-th position in the superscripts of $\Delta$.

Note that in order to compute $((Z \circ_i Y) \circ_j X)^{\Gamma}$, we need to evaluate $X^{\Gamma_2}$, $Y^{\Gamma_1}$ and $Z^{\Gamma_0}$. Assume that there are $r, t, s$ connected components in $\Gamma_0, \Gamma_1, \Gamma_2$, respectively, and write
\begin{align}
    X^{\Gamma_2}(v) &= \sum_{\alpha} (f_{\alpha1} \otimes \dots \otimes f_{\alpha s}) \otimes_H x_\alpha(v),
    & v &\in V^{\otimes m}, \label{xgamma2}\\
    Y^{\Gamma_1}(w) &= \sum_{\beta} (g_{\beta 1} \otimes \dots \otimes g_{\beta t}) \otimes_H y_\beta(w),
    & w &\in V^{\otimes n}, \label{ygamma1}\\
    Z^{\Gamma_0}(u) &= \sum_{\gamma} (h_{\gamma 1} \otimes \dots \otimes h_{\gamma r}) \otimes_H z_\gamma(u),
    & u &\in V^{\otimes l} \label{zgamma0}.
\end{align}
If $\Gamma_1$ is clasped into a vertex of the $q_1$-th connected component of $\Gamma_0$, then we compute by (\ref{circleatk}):
\begin{align}\label{ziprody}
    (&Z \circ_i Y)^{\overline{\Gamma}_1}(v_1 \otimes \dots \otimes v_{n+l-1}) \notag \\
   &= (-1)^{p_i^Y}\sum_{\beta,\gamma}\rho^{\overline{\Gamma}_1}_{i}\bigl(h_{\gamma1} \otimes \dots \otimes h_{\gamma\, q_1-1} \otimes g_{\beta1(1)} h_{\gamma q_1(1)} \otimes \dots \otimes g_{\beta t(1)} h_{\gamma q_1(t)}  \notag \\
    & \;\;\otimes h_{\gamma\, q_1+1} \dots \otimes h_{\gamma r} \bigr) \otimes_H z_\gamma \bigl((g_{\beta 1(-2)} \otimes\dots\otimes g_{\beta t(-2)}) \cdot(v_1 \otimes \dots \otimes v_{i-1} \notag \\
    & \;\; \otimes y_\beta(v_i \otimes \dots \otimes v_{i+n-1}) \otimes v_{i+n} \otimes \dots \otimes v_{n+l-1})\bigr),
\end{align}
where $p_i^Y$ is given by (\ref{signcircle}). 
To further compute $((Z \circ_i Y) \circ_j X)^{\Gamma}$, we need to consider the relations between $\Gamma_1$ and $\Gamma_2$, hence we split into the following two cases. 

In the first case, assume that $\Gamma_1$ and $\Gamma_2$ are disjoint after they are clasped into two vertices in $\Gamma_0$, that is there exists no path in $\Gamma$ connecting any two connected components of $\Gamma_1$ and $\Gamma_2$ as shown in the graph below. The graph does not show all possible edges; only edges between $\Gamma_2$ and $\Gamma_1$ are not allowed.

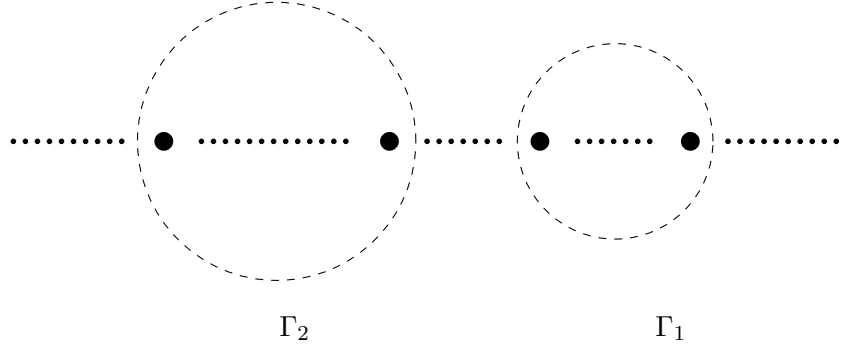
\vspace{4mm}
\begin{center}
\begin{tikzpicture}
\tikzset{>={Latex[width=2mm,length=2mm]}}
\tikzstyle{vertex} =[circle, fill = black,inner sep=0pt,minimum size=1mm,scale=2.5]
\tikzset{decorate sep/.style 2 args=
{decorate,decoration={shape backgrounds,shape=circle,shape size=#1,shape sep=#2}}}

$\Gamma$

\node[vertex](v3) at (3,0){};
\node[circle, fill = white,label=below:$\Gamma_2$](v45) at (4.5,-2){};
\node[vertex](v6) at (6,0){};
\node[vertex](v8) at (8,0){};
\node[circle, fill = white,label=below:$\Gamma_1$](v9) at (9.5,-2){};
\node[vertex](v10) at (10,0){};

\draw[decorate sep={0.6mm}{1.6mm},fill] (1,0) -- (2.5,0); 
\draw[decorate sep={0.6mm}{1.6mm},fill] (3.5,0) -- (5.5,0); 
\draw[decorate sep={0.6mm}{1.6mm},fill] (6.5,0) -- (7.5,0); 
\draw[decorate sep={0.6mm}{1.6mm},fill] (8.5,0) -- (9.5,0); 
\draw[decorate sep={0.6mm}{1.6mm},fill] (10.5,0) -- (12,0); 

\draw[dashed] (4.5,0) circle[radius=1.85cm];
\draw[dashed] (9,0) circle[radius=1.30cm];
\end{tikzpicture}
\end{center}

In this case, if $\Gamma_2$ is clasped into a vertex of the $\overline{q}_2$-th connected component of $\overline{\Gamma}_1$, then $\overline{q}_2$ corresponds to some $q_2$-th connected component in $\Gamma_0$ that is different from the $q_1$-th. We have:
\begin{align}\label{ziyjx}
    & ((Z \circ_i Y) \circ_j X)^{\Gamma}(v_1 \otimes \dots \otimes v_{m+n+l-2}) \notag \\
    &= (-1)^{p_{i+m-1}^Y+p_j^X} \sum_{\alpha,\beta,\gamma} \rho^\Gamma_{j}\bigl((
    \underbrace{1 \otimes \dots \otimes 1}_{\overline{q}_2-1} \otimes f_{\alpha1(1)} \otimes \dots \otimes f_{\alpha s(1)} \otimes 1 \dots \otimes 1)  \notag \\
    & (\Delta^{(s-1)}_{\overline{q}_2} \cdot \rho^{\overline{\Gamma}_1}_{i}(h_{\gamma1} \otimes \dots \otimes h_{\gamma\, q_1-1} \otimes g_{\beta1(1)} h_{\gamma q_1(1)} \otimes \dots \otimes g_{\beta t(1)} h_{\gamma q_1(t)} \otimes h_{\gamma\, q_1+1} \notag \\
    & \otimes \dots \otimes h_{\gamma r}))\bigr) \otimes_H z_\gamma \bigl((f_{\alpha 1(-2)} \otimes\cdots \otimes f_{\alpha s(-2)}) \cdot (g_{\beta1(-2)} 
    \otimes\cdots \otimes g_{\beta t(-2)}) \notag \\
    & \cdot(v_1 \otimes \dots \otimes v_{j-1} \otimes x_\alpha(v_j \otimes \dots \otimes v_{j+m-1}) \otimes v_{j+m} \otimes \dots \otimes v_{i+m-2} \notag \\
    & \otimes y_\beta(v_{i+m-1} \otimes \dots \otimes v_{i+m+n-2}) \otimes v_{i+m+n-1} \otimes \dots \otimes v_{m+n+l-2})\bigr).
\end{align}
Note that 
$\Delta^{(s-1)}_{\overline{q}_2}$ only applies on the $\overline{q}_2$-th connected component after the permutation $\rho^{\overline{\Gamma}_1}_{i}$. Since $\Gamma_1$ and $\Gamma_2$ are disconnected in $\Gamma$, $\overline{q}_2$ is not any connected component corresponding to $h_{\gamma q_1 (p)}$ for any $p$. Similarly, one sees that $f_{\alpha k (1)}$ multiplies with $h_{\gamma a (k)}$ from $\Delta^{(s-1)}(h_{\gamma a})$ (here $h_{\gamma a}$ represents the connected component in $\Gamma_0$ that becomes the $\overline{q}_2$-th connected component in $\overline{\Gamma}_1$ under $\rho_{i}^{\overline{\Gamma}_1}$), while $f_{\alpha k(1)}$'s and $g_{\beta p(1)}$'s do not multiply together. Eventually, $\rho^\Gamma_{j}$ permutes the connected components to their correct positions in $\Gamma$. In $(\ref{ziprody})$, $g_{\beta p(-2)}$'s act only on vectors that are not within $y_\beta$.
They do not act on vectors within $x_\alpha$ in (\ref{ziyjx}) either; otherwise we have a path between $\Gamma_1$ and $\Gamma_2$ in $\Gamma$. Thus, $g_{\beta p(-2)}$'s only act on vectors not in $x_\alpha$ and $y_\beta$, which is also true for $f_{\alpha k(-2)}$'s. The last observation is that the vectors not within $x_\alpha$ and $y_\beta$ can have at most one action: either from $g_{\beta p(-2)}$ or $f_{\alpha k(-2)}$; otherwise we have a path between $\Gamma_1$ and $\Gamma_2$ in $\Gamma$. These observations imply that we do not need to worry about the commutativity of $H$.

Next, we compute the right hand side $((Z \circ_j X) \circ_{i+m-1} Y)^\Gamma$. Now $X$ will always be evaluated on $\Gamma_2$, $Y$ on $\Gamma_1$, and $Z$ on $\Gamma_0$. Let $\widetilde{\Gamma}_1 =\Delta_0^{1 \dots 1n1 \dots 1}(\Gamma)$ where $n$ appears at the $(i+m-1)$-th position. Note that $\Gamma_2$ will be clasped into a vertex in the $q_2$-th connected component of $\Gamma_0$, and we have:
\begin{align}\label{zjprodx}
    & (Z \circ_j X)^{\widetilde{\Gamma}_1}(v_1 \otimes \dots \otimes v_{m+l-1}) \notag \\
    &= (-1)^{p_j^X}\sum_{\beta,\gamma} \rho^{\widetilde{\Gamma}_1}_{j} \bigl(h_{\gamma1} \otimes \dots \otimes h_{\gamma\, q_2-1} \otimes f_{\alpha1(1)} h_{\gamma q_2(1)} \otimes \dots \otimes f_{\alpha s(1)} h_{\gamma q_2(s)} \notag \\
    & \;\otimes h_{\gamma\, q_2+1} \otimes\dots \otimes h_{\gamma r} \bigr) \otimes_H z_\gamma \bigl((f_{\alpha 1(-2)} \otimes\cdots\otimes f_{\alpha s(-2)}) \cdot(v_1 \otimes \dots \otimes v_{j-1} \notag \\
    & \;\otimes x_\alpha(v_j \otimes \dots \otimes v_{j+m-1}) \otimes v_{j+m} \otimes \dots \otimes v_{m+l-1})\bigr).
\end{align}
Suppose that $\Gamma_1$ is clasped into a vertex in the $\widetilde{q}_1$-th connected component of $\widetilde{\Gamma}_1$. This will correspond to the $q_1$-th connected component of $\Gamma_0$ as for the left-hand side. Thus
\begin{align}\label{zjxim1y}
    & ((Z \circ_j X) \circ_{i+m-1} Y)^{\Gamma}(v_1 \otimes \dots \otimes v_{m+n+l-2}) \notag \\
    &= (-1)^{p_{i+m-1}^Y+p_j^X+p(Y)p(X)} \sum_{\alpha,\beta,\gamma} \rho^\Gamma_{i+m-1}\bigl((
    \underbrace{1 \otimes \dots \otimes 1}_{\widetilde{q}_1-1} \otimes g_{\beta1(1)} \otimes \dots\otimes g_{\beta t(1)} \notag \\
    & \;\;\otimes 1 \otimes \dots \otimes 1)(\Delta^{(t-1)}_{\widetilde{q}_1} \cdot \rho^{\widetilde{\Gamma}_1}_{j} 
    (h_{\gamma 1} \otimes \dots \otimes h_{\gamma\, q_2-1} \otimes f_{\alpha1(1)} h_{\gamma q_2(1)} \otimes \cdots \notag \\
    & \;\;\otimes f_{\alpha s(1)} h_{\gamma q_2(s)} \otimes h_{\gamma\, q_2+1} \dots \otimes h_{\gamma r} ))\bigr) \otimes_H  z_\gamma \bigl((g_{\beta1(-2)} \otimes\cdots\otimes g_{\beta t(-2)}) \notag \\
    & \;\;\cdot (f_{\alpha 1(-2)} \otimes\cdots \otimes f_{\alpha s(-2)}) \cdot(v_1 \otimes \dots \otimes v_{j-1} \otimes x_\alpha(v_j \otimes \dots \otimes v_{j+m-1}) \notag \\ 
    & \;\;\otimes v_{j+m} \otimes \dots \otimes v_{i+m-2} \otimes y_\beta( v_{i+m-1} \otimes \dots \otimes v_{i+m+n-2}) \notag \\
    & \;\;\otimes v_{i+m+n-1} \otimes \dots \otimes v_{m+n+l-2})\bigr).
\end{align}
The action of $f_{\alpha k(-2)}$'s and $g_{\beta p(-2)}$'s on the vectors are the same as in (\ref{ziyjx}). Each vector gets at most one action; hence the $z_\gamma$ part 
is the same as in (\ref{ziyjx}). For the coefficients representing each connected component, we have the following observations: 
\begin{enumerate}
\item
$h_{\gamma b}$ (which represents the connected component in $\Gamma_0$ that becomes the $\widetilde{q}_1$-th connected component in $\widetilde{\Gamma}_1$ under $\rho_{j}^{\widetilde{\Gamma}_1}$) is distinct from $h_{\gamma q_2}$ since $\Gamma_1$ and $\Gamma_2$ are disjoint.
\item
$\Delta^{(s-1)}(h_{\gamma q_2})$ multiplies with $f_{\alpha 1(1)} \otimes \dots \otimes f_{\alpha s(1)}$ and then gets permuted twice, by $\rho^{\widetilde{\Gamma}_1}_{j}$ and $\rho^\Gamma_{i+m-1}$, while in (\ref{ziyjx}), $\Delta^{(s-1)}(h_{\gamma a})$ will get permuted by $\rho^\Gamma_{j}$. 
\item
These two parts will be the same, since $h_{\gamma q_2}$ and $h_{\gamma a}$ are essentially the same, 
both representing the connected component in $\Gamma_0$ that contains the vertex clasped from $\Gamma_2$. They are eventually identified as connected components of the same graph $\Gamma$. 
\item
Similar arguments work for $h_{\gamma q_1}$ and $h_{\gamma b}$. The connected components that are disjoint from $\Gamma_1$ and $\Gamma_2$ will be the same in both cases, since they are eventually identified as connected components of $\Gamma$. 
\end{enumerate}

The above arguments prove that $((Z \circ_i Y) \circ_j X)^\Gamma =(-1)^{p(Y)p(X)} ((Z \circ_j X) \circ_{i+m-1} Y)^\Gamma$ when $\Gamma_1$ and $\Gamma_2$ are disjoint.
The other case is when $\Gamma_1$ and $\Gamma_2$ are connected. Note that there can be only one path connecting the $k_1$-th connected component of $\Gamma_1$ and the $k_2$-th connected component of $\Gamma_2$; otherwise we have a cycle in $\Gamma_0$. A typical example looks like the graph below, where there may be some vertices $c$ outside $\Gamma_1$ and $\Gamma_2$ on the path from $a$ to $b$.

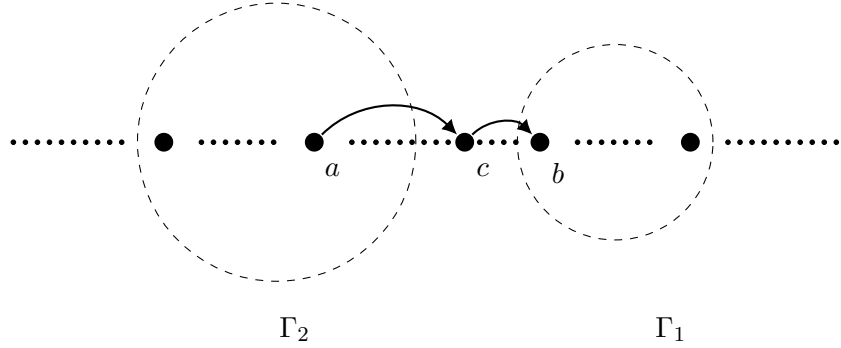
\vspace{4mm}
\begin{center}
\begin{tikzpicture}
\tikzset{>={Latex[width=2mm,length=2mm]}}
\tikzstyle{vertex} =[circle, fill = black,inner sep=0pt,minimum size=1mm,scale=2.5]
\tikzset{decorate sep/.style 2 args=
{decorate,decoration={shape backgrounds,shape=circle,shape size=#1,shape sep=#2}}}

$\Gamma$

\node[vertex](v3) at (3,0){};
\node[circle, fill = white,label=below:$\Gamma_2$](v45) at (4.5,-2){};
\node[vertex,label=below:$a$](v5) at (5,0){};
\node[vertex,label=below:$c$](v7) at (7,0){};
\node[circle, fill = white,label=below:$\Gamma_1$](v9) at (9.5,-2){};
\node[vertex,label=below:$b$](v8) at (8,0){};
\node[vertex](v10) at (10,0){};

\draw[decorate sep={0.6mm}{1.6mm},fill] (1,0) -- (2.5,0); 
\draw[decorate sep={0.6mm}{1.6mm},fill] (3.5,0) -- (4.5,0); 
\draw[decorate sep={0.6mm}{1.6mm},fill] (5.5,0) -- (6.8,0); 
\draw[decorate sep={0.6mm}{1.6mm},fill] (7.2,0) -- (7.8,0); 
\draw[decorate sep={0.6mm}{1.6mm},fill] (8.5,0) -- (9.5,0); 
\draw[decorate sep={0.6mm}{1.6mm},fill] (10.5,0) -- (12,0); 

\draw [->,thick] (v5) to [out=45,in=135] (v7);
\draw [->,thick] (v7) to [out=45,in=135] (v8);

\draw[dashed] (4.5,0) circle[radius=1.85cm];
\draw[dashed] (9,0) circle[radius=1.30cm];
\end{tikzpicture}
\end{center}

To simplify the calculations, we use the observation that for any left $H$-module $W$ and elements $h_1,\dots,h_n\in H$, $w\in W$,
the expression $(h_1 \otimes \dots \otimes h_n) \otimes_H w$ 
can be rewritten in the form $(\widetilde{h}_1 \otimes \dots \otimes \widetilde{h}_{i-1} \otimes 1 \otimes \widetilde{h}_{i+1} \otimes \dots \otimes \widetilde{h}_n) \otimes_H \widetilde{w}$ for any $i \in \{1,\dots, n\}$, where $\widetilde{h}_j \in H$ and $\widetilde{w} \in W$ are unique (see \eqref{tool1}, \eqref{identity1} and \cite{BDK01}). 
Hence, without loss of generality, we can assume that $f_{\alpha k_2} = 1$ in (\ref{xgamma2}) and $g_{\beta k_1} =1$ in (\ref{ygamma1}). 
Then we compute $(Z \circ_i Y)^{\overline{\Gamma}_1}$ from (\ref{ziprody}) by letting $g_{\beta k_1(1)} = g_{\beta k_1(-2)} =1$. 
For $((Z \circ_i Y) \circ_j X)^{\Gamma}$, one can still use (\ref{ziyjx}) with the following changes:
\begin{enumerate}
    \item $f_{\alpha k_2(1)} = f_{\alpha k_2(-2)} =1$. Note that $g_{\beta k_1 (1)} = 1$ and $\Delta^{(s-1)}_{\overline{q}_2}$ applies on the connected component corresponding to $h_{\gamma q_1 (k_1)}$; hence $f_{\alpha k(1)}$'s and $g_{\beta p(1)}$'s do not multiply together.
    
    \item Note that if a vector $v_l$ in $z_\gamma$ gets both actions from $f_{\alpha k(-2)}$ and $g_{\beta p(-2)}$, then it must be in the component that connects $\Gamma_1$ and $\Gamma_2$, and the actions are from $f_{\alpha k_2(-2)}$ and $g_{\beta k_1(-2)}$, which are trivial. Other vectors can only get one action, either from $f_{\alpha k(-2)}$ or $g_{\beta p(-2)}$; otherwise we get a cycle between $\Gamma_1$ and $\Gamma_2$ in $\Gamma_0$.
\end{enumerate}

The right-hand side $((Z \circ_j X) \circ_{i+m-1} Y)^\Gamma$ can be computed in a similar way. First, $(Z \circ_j Y)^{\widetilde{\Gamma}_1}$ can be obtained by letting $f_{\alpha k_2(1)} = f_{\alpha k_2(-2)} =1$ in (\ref{zjprodx}). Then $((Z \circ_j X) \circ_{i+m-1} Y)^{\Gamma}$ is computed from (\ref{zjxim1y}) with the following changes:
\begin{enumerate}
    \item $g_{\beta k_1(1)} = g_{\beta k_1(-2)} =1$. Note that $f_{\alpha k_2 (1)} = 1$ and $\Delta^{(t-1)}_{\widetilde{q}_1}$ acts on the connected component corresponding to $h_{\gamma q_2 (k_2)}$; hence $f_{\alpha k(1)}$'s and $g_{\beta p(1)}$'s do not multiply together.
    
    \item The second change as we compute $((Z \circ_i Y) \circ_j X)^{\Gamma}$ above also applies here.
\end{enumerate}
Since on both sides we identify the connected components of $\Gamma$ by either $\rho^\Gamma_{j}$ or $\rho^\Gamma_{i+m-1}$, we obtain 
$((Z \circ_i Y) \circ_j X)^\Gamma =(-1)^{p(Y)p(X)} ((Z \circ_j X) \circ_{i+m-1} Y)^\Gamma$ when $\Gamma_1$ and $\Gamma_2$ are connected.
This completes the proof of the lemma.
\end{proof}

\begin{remark}
In the above proof, removing the path connecting $\Gamma_1$ and $\Gamma_2$ (there is one and only one such path) will make these two new subgraphs disjoint. Hence, the case when $\Gamma_1$ and $\Gamma_2$ are connected can be thought of as the disjoint case once $f_{\alpha k_2}$ and $g_{\beta k_1}$ are set to $1$.
\end{remark}

\begin{lemma}\label{operadeq2}
We have $(Z \circ_i Y) \circ_j X = Z \circ_i (Y \circ_{j-i+1} X)$ for $i \leq j < i+n$,
i.e., the second identity in \eqref{assoeq} holds.
\end{lemma}

\begin{proof}
Consider a graph $\Gamma \in G(m+n+l-2)$, and introduce the following notation:
\begin{align*}
    \overline{\Gamma}_1 &:= \Delta_i^{1 \dots 1(m+n-1)1 \dots 1}(\Gamma) \in G(m+n-1), &  \Gamma_2 &:= \Delta_{j-i+1}^{1 \dots 1m1 \dots 1}(\overline{\Gamma}_1) \in G(m), \\
    \Gamma_0 &:= \Delta_0^{1 \dots 1(m+n-1)1 \dots 1}(\Gamma) \in G(l), & \Gamma_1 &:= \Delta_0^{1 \dots 1m1 \dots 1}(\overline{\Gamma}_1) \in G(n), \\
    \overline{\Gamma}_0 &:= \Delta_0^{1 \dots 1m1 \dots 1}(\Gamma) \in G(n+l-1), &&
\end{align*}
where $m+n-1$ is at the $i$-th position and $m$ is at the $(j-i+1)$-th position in the superscripts of $\Delta$, 
except that $m$ is at the $j$-th position in the last cocomposition map defining $\overline{\Gamma}_0$.
Note that $\overline{\Gamma}_0$ is obtained by clasping $\Gamma_2$ in $\Gamma$. 
Since $i \leq j < i+n$, the graph $\Gamma$ looks like this:

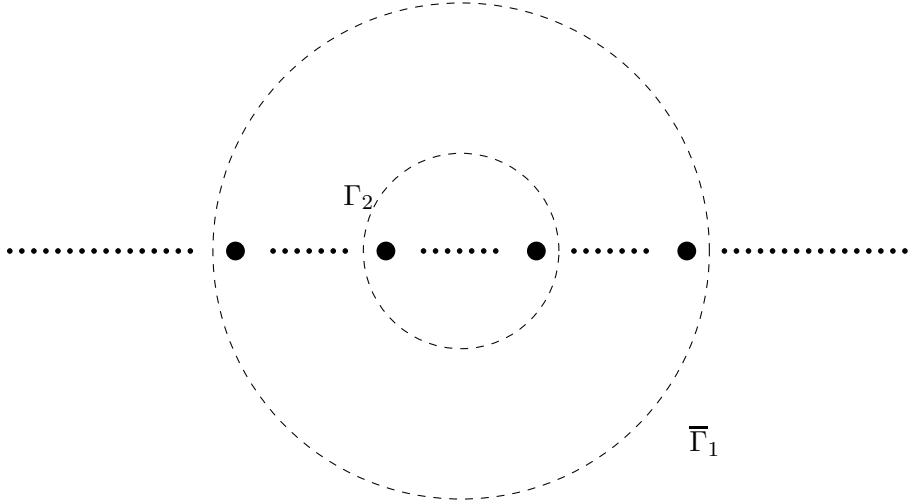
\vspace{2mm}
\begin{center}
\begin{tikzpicture}
\tikzset{>={Latex[width=2mm,length=2mm]}}
\tikzstyle{vertex} =[circle, fill = black,inner sep=0pt,minimum size=1mm,scale=2.5]
\tikzset{decorate sep/.style 2 args=
{decorate,decoration={shape backgrounds,shape=circle,shape size=#1,shape sep=#2}}}

$\Gamma$

\node[vertex](v3) at (4,0){};
\node[circle, fill = white,label=below:$\Gamma_2$](v45) at (5.4,1.2){};
\node[vertex](v6) at (6,0){};
\node[vertex](v8) at (8,0){};
\node[circle, fill = white,label=below:$\overline{\Gamma}_1$](v9) at (10,-2){};
\node[vertex](v10) at (10,0){};

\draw[decorate sep={0.6mm}{1.6mm},fill] (1,0) -- (3.5,0); 
\draw[decorate sep={0.6mm}{1.6mm},fill] (4.5,0) -- (5.5,0); 
\draw[decorate sep={0.6mm}{1.6mm},fill] (6.5,0) -- (7.5,0); 
\draw[decorate sep={0.6mm}{1.6mm},fill] (8.5,0) -- (9.5,0); 
\draw[decorate sep={0.6mm}{1.6mm},fill] (10.5,0) -- (13,0); 

\draw[dashed] (7,0) circle[radius=1.3cm];
\draw[dashed] (7,0) circle[radius=3.30cm];
\end{tikzpicture}
\end{center}

As in the proof of Lemma \ref{operadeq1}, $X, Y, Z$ will be evaluated on $\Gamma_2$, $\Gamma_1$, $\Gamma_0$, respectively. 
Using the same notation for the number of connected components and expressions (\ref{xgamma2}), (\ref{ygamma1}), (\ref{zgamma0}), one computes:
\begin{align}\label{ziyaa}
    & (Z \circ_i Y)^{\overline{\Gamma}_0}(v_1 \otimes \dots \otimes v_{n+l-1}) \notag \\
    &= (-1)^{p_i^Y}\sum_{\beta,\gamma}\rho^{\overline{\Gamma}_0}_{i}\bigl((\underbrace{1 \otimes \dots \otimes 1}_{q_1-1} \otimes g_{\beta1(1)} \otimes \dots \otimes g_{\beta t(1)} \otimes 1 \otimes \dots \otimes 1) \notag \\
    & \;\;(h_{\gamma1} \otimes \cdots \otimes \Delta^{(t-1)}(h_{\gamma q_1}) \otimes \dots \otimes h_{\gamma r})\bigr) \, \otimes_H z_\gamma \bigl((g_{\beta 1(-2)} \otimes\cdots\otimes g_{\beta t(-2)}) \notag \\
    & \;\;\cdot(v_1 \otimes \dots \otimes v_{i-1} \otimes y_\beta(v_i \otimes \dots \otimes v_{i+n-1}) \otimes v_{i+n} \otimes \dots \otimes v_{n+l-1})\bigr) \notag \\
    &= (-1)^{p_i^Y}\sum_{\beta,\gamma}\rho^{\overline{\Gamma}_0}_{i}(h_{\gamma1} \otimes \dots \otimes h_{\gamma\, q_1-1} \otimes g_{\beta1(1)} \otimes \dots \otimes g_{\beta t(1)} \otimes h_{\gamma\, q_1+1} \notag \\ 
    & \;\;\otimes \dots \otimes h_{\gamma r}) \otimes_H z_\gamma \bigl((g_{\beta 1(-2)} \otimes\cdots\otimes g_{\beta t(-2)}) \cdot(v_1 \otimes \cdots \otimes v_{i-1}  \notag \\
    & \;\;\otimes y_\beta(v_i \otimes \dots \otimes v_{i+n-1}) \otimes v_{i+n} \otimes \dots \otimes v_{n+l-1})\bigr), 
\end{align}
where the vertex clasped from $\Gamma_1$ is contained in the $q_1$-th connected component of $\Gamma_0$. The last equality is obtained after setting $h_{\gamma q_1} =1$; we can always do that as explained in the proof of Lemma \ref{operadeq1}.

If $\Gamma_2$ is clasped into a vertex in the $\overline{q}_2$-th connected component of $\overline{\Gamma}_0$, one finds $((Z \circ_i Y) \circ_j X)^{\Gamma}$ as follows:
\begin{align}\label{zyx1}
    & ((Z \circ_i Y) \circ_j X)^{\Gamma}(v_1 \otimes \dots \otimes v_{m+n+l-2}) \notag \\
    &=\! (-1)^{p_{i}^Y\!+p_j^X} \!\!\sum_{\alpha,\beta,\gamma} \rho^\Gamma_{j}\bigl((\underbrace{1 \otimes \!\cdots\! \otimes 1}_{\overline{q}_2-1} \otimes f_{\alpha1(1)} \otimes \dots \otimes f_{\alpha s(1)} \otimes 1 \otimes\!\cdots\!\otimes 1) (\Delta^{(s-1)}_{\overline{q}_2} \notag \\
    &\;\;\cdot (\rho^{\overline{\Gamma}_0}_{i}(h_{\gamma1} \otimes \dots\otimes h_{\gamma\, q_1-1} \otimes g_{\beta1(1)} \otimes \dots \otimes g_{\beta t(1)} \otimes h_{\gamma\, q_1+1} \otimes \dots \otimes h_{\gamma r})))\bigr) \notag \\
    &\; \otimes_H z_\gamma \bigl((f_{\alpha 1(-2)} \otimes \dots \otimes f_{\alpha s(-2)}) \cdot (g_{\beta1(-2)} \otimes\dots\otimes g_{\beta t(-2)}) \cdot(v_1 \otimes \dots \otimes v_{i-1} \notag \\
    & \;\otimes y_\beta(v_i \otimes \cdots 
    \otimes v_{j-1} \otimes x_\alpha(v_j \otimes \dots \otimes v_{j+m-1}) \otimes v_{j+m} \otimes \dots \otimes v_{i+m+n-2}) \notag \\
    & \;\otimes v_{i+m+n-1} \otimes \dots \otimes v_{m+n+l-2})\bigr).
\end{align}
Note that in (\ref{zyx1}), since the vertex clasped from $\Gamma_2$ is in $\Gamma_1$, the $\overline{q}_2$-th connected component of $\overline{\Gamma}_0$ corresponds to $g_{\beta l(1)}$ for some $l$. We assume that $g_{\beta l} =1$; hence $\Delta^{(s-1)}_{\overline{q}_2}(g_{\beta l(1)}) = 1 \otimes \dots \otimes 1$ and $f_{\alpha k(1)}$'s and $g_{\beta p(1)}$'s do not multiply together. In $z_\gamma$, $f_{\alpha k(-2)}$'s act only on vectors outside $x_{\alpha}$ and $g_{\beta p(-2)}$'s act only on vectors outside $y_{\beta}$. All vectors outside $y_{\beta}$ can only have at most a single action, either from $f_{\alpha k(-2)}$ or $g_{\beta p(-2)}$; otherwise we get a cycle in $\Gamma_0$. 

Next, we will compute the right-hand side $(Z \circ_i (Y \circ_{j-i+1} X))^\Gamma$. If the vertex clasped from $\Gamma_2$ is contained in the $\overline{q}_1$-th connected component of ${\Gamma}_1$, then we have:
\begin{align}\label{yx2}
    & (Y \circ_{j-i+1} X)^{\overline{\Gamma}_1}(v_1 \otimes \dots \otimes v_{m+n-1}) \notag \\
    &= (-1)^{p_{j-i+1}^X}\sum_{\beta,\gamma}\rho^{\overline{\Gamma}_1}_{j-i+1}\bigl((\underbrace{1 \otimes \dots \otimes 1}_{\overline{q}_1-1} \otimes f_{\alpha1(1)} \otimes \dots \otimes f_{\alpha s(1)} \otimes 1 \otimes\dots \otimes 1) \notag \\
    & \;\; (g_{\beta 1} \otimes \dots \otimes \Delta^{(s-1)}(g_{\beta \overline{q}_1}) \otimes \dots \otimes g_{\beta t})\bigr) \otimes_H y_\beta \bigl((f_{\alpha 1(-2)} \otimes\dots\otimes f_{\alpha s(-2)}) \notag \\
    &  \cdot (v_1 \otimes \dots \otimes v_{j-i} \otimes x_\alpha(v_{j-i+1} \otimes \dots \otimes v_{j-i+m}) \otimes v_{j-i+m+1} \otimes \dots \otimes v_{m+n-1})\bigr) \notag \\
    &= (-1)^{p_{j-i+1}^X}\sum_{\beta,\gamma}\rho^{\overline{\Gamma}_1}_{j-i+1} \bigl(g_{\beta 1} \otimes \dots\otimes g_{\beta\, \overline{q}_1-1} \otimes f_{\alpha 1(1)} \otimes \dots \otimes f_{\alpha s(1)} \otimes g_{\beta\, \overline{q}_1+1} \notag \\
    & \;\;\otimes \dots \otimes g_{\beta t} \bigr) \otimes_H y_\beta \bigl((f_{\alpha 1(-2)} \otimes\dots\otimes f_{\alpha s(-2)}) \cdot(v_1 \otimes \dots 
    \otimes v_{j-i} \notag \\
    & \;\;\otimes x_\alpha(v_{j-i+1} \otimes \dots \otimes v_{j-i+m}) \otimes v_{j-i+m+1} \otimes \dots \otimes v_{m+n-1})\bigr), 
\end{align}
where the last equality is obtained after we set $g_{\beta \overline{q}_1} = 1$. 
Indeed, here $\overline{q}_1 = l$ as in (\ref{zyx1}), since $g_{\beta \overline{q}_1}$ corresponds to the connected component of $\Gamma_1$ that contains the vertex clasped from $\Gamma_2$. So we are imposing the same condition $g_{\beta l} = 1$ both in (\ref{zyx1}) and \eqref{yx2}.

Finally, we can compute $(Z \circ_i (Y \circ_{j-i+1} X))^\Gamma$. Suppose that $\overline{\Gamma}_1$ is clasped into a vertex contained in the $q_2$-th connected component of $\Gamma_0$. Then
\begin{align}\label{zyx2}
    & (Z \circ_i (Y \circ_{j-i+1} X))^\Gamma (v_1 \otimes \dots \otimes v_{m+n+l-2}) \notag \\
    &= (-1)^{p_{i}^Y+p_j^X} \sum_{\alpha,\beta,\gamma} \rho^\Gamma_{i}\bigl((\underbrace{1 \otimes \dots \otimes 1}_{q_2-1} \otimes \rho^{\overline{\Gamma}_1}_{j-i+1}(g_{\beta 1(1)} \otimes \dots \otimes g_{\beta\, \overline{q}_1-1(1)} \notag \\
    & \;\; \otimes f_{\alpha 1(1)} \otimes \dots \otimes f_{\alpha s(1)} \otimes g_{\beta\, \overline{q}_1+1(1)} \otimes \dots \otimes g_{\beta t(1)}) \otimes 1 \otimes \dots \otimes 1) (h_{\gamma1} \otimes \cdots \notag \\
    & \;\;\otimes h_{\gamma\, q_2-1} \otimes \Delta^{(s+t-2)}(h_{\gamma q_2}) \otimes h_{\gamma\, q_2+1} \otimes \dots \otimes h_{\gamma r})\bigr) \otimes_H z_\gamma \bigl((g_{\beta 1(-2)} \otimes \cdots \notag \\
    & \;\; \otimes g_{\beta\, \overline{q}_1-1(-2)} \otimes f_{\alpha 1(-2)} \otimes \dots \otimes f_{\alpha s(-2)} \otimes g_{\beta\, \overline{q}_1+1(-2)} \otimes \dots \otimes g_{\beta t(-2)}) \notag \\
    & \;\;\cdot(v_1 \otimes \dots \otimes v_{i-1} \otimes y_\beta\bigl((f_{\alpha 1(-3)} \otimes\dots \otimes f_{\alpha s(-3)}) \cdot (v_i \otimes \dots \otimes v_{j-1} \notag \\
    & \;\;\otimes x_\alpha(v_j \otimes \dots \otimes v_{j+m-1}) \otimes v_{j+m} \otimes \dots  \otimes v_{i+m+n-2})\bigr) \notag \\
    & \;\; \otimes v_{i+m+n-1} \otimes \dots \otimes v_{m+n+l-2})\bigr). 
\end{align}
Note that here $q_2 = q_1$ and we are imposing the same condition as in (\ref{ziyaa}): $h_{\gamma q_2} = 1$. Thus, $\rho^{\overline{\Gamma}_1}_{j-i+1}(g_{\beta 1(1)} \otimes \dots \otimes f_{\alpha 1(1)} \otimes \dots \otimes f_{\alpha s(1)} \otimes \dots \otimes g_{\beta t(1)})$ will multiply with $\Delta^{(s+t-2)}(h_{\gamma q_2}) = 1 \otimes \dots \otimes 1$. In $z_\gamma$, we use the identity $f_{\alpha k(1)(1)} \otimes f_{\alpha k(1)(-2)} \otimes f_{\alpha k(-2)} = f_{\alpha k(1)} \otimes f_{\alpha k(-2)} \otimes f_{\alpha k(-3)}$. Hence, $f_{\alpha k (-2)}$'s will only act on the vectors outside $y_\beta$, and $f_{\alpha k (-3)}$'s will only act on the vectors inside $y_\beta$ but out of $x_\alpha$. As in (\ref{zyx1}), $f_{\alpha k (-2)}$'s and $g_{\beta p(-2)}$'s will not act on the same vector outside $y_\beta$; otherwise we get a cycle in $\Gamma_0$.

To see that (\ref{zyx1}) and (\ref{zyx2}) are equal, note that for the coefficient part, each tensor factor can only be $f_{\alpha k(1)}$, $g_{\beta p (1)}$ or $h_{\gamma q}$. No multiplication will appear since we impose the conditions $g_{\beta l} = h_{\gamma q_1} = 1$. Since both (\ref{zyx1}) and (\ref{zyx2}) identify the connected components of $\Gamma$ in the end, the coefficient parts are equal. For the vector part, as $\Delta(f_{\alpha k(-2)}) = f_{\alpha k(-2)} \otimes f_{\alpha k(-3)} = f_{\alpha k(-3)} \otimes f_{\alpha k(-2)}$ (since $H$ is cocommutative), the actions of $f_{\alpha k}$'s are the same; hence the vector parts are the same. This completes the proof of the lemma.
\end{proof}

\begin{lemma}\label{operadeq3}
For any $X \in \mathcal{P}^{cl}_H(m)$, $Y \in \mathcal{P}^{cl}_H(n)$, $\sigma \in S_n$, $\tau \in S_m$, and $1\le i\le n$,
we have $Y^\sigma \circ_i X^\tau = (Y \circ_{\sigma(i)} X)^{\sigma \circ_i \tau}$, i.e., \eqref{equieq} holds where $\sigma \circ_i \tau$ is given by \eqref{sigcircitau}.
\end{lemma}

\begin{proof}
Consider a graph $\Gamma \in G(m+n-1)$, and let 
\begin{align*}
 \Gamma_0 = \sigma(\Delta_0^{1 \dots 1m1 \dots 1}(\Gamma)), \quad \Gamma_i = \tau(\Delta_i^{1 \dots 1m1 \dots 1}(\Gamma)),
\end{align*}
where $m$ appears at the $i$-th position. As before, we write
\begin{align}
X^{\Gamma_i}(v) &= \sum_{\alpha} (f_{\alpha 1} \otimes \dots \otimes f_{\alpha s}) \otimes_H x_\alpha (v),
& v &\in V^{\otimes m}, \label{xgammai}\\
Y^{\Gamma_0}(w) &= \sum_{\beta} (g_{\beta 1} \otimes \dots \otimes g_{\beta t}) \otimes_H y_\beta(w),
& w &\in V^{\otimes n}. \label{ygamma0}
\end{align}
First, using (\ref{snaction})--\eqref{tildesigma}, (\ref{xgammai}) and (\ref{ygamma0}), we find:
\begin{align}
    & (X^\tau)^{\Delta_i^{1 \dots 1m1 \dots 1}(\Gamma)}(v_i \otimes \dots \otimes v_{i+m-1}) \notag \\
    &= (\widetilde{\tau} \otimes_H 1)X^{\Gamma_i} \bigl(\tau(v_i \otimes \dots \otimes v_{i+m-1})\bigr) \notag \\
    &= \epsilon(\tau) \sum_{\alpha} \widetilde{\tau}(f_{\alpha 1} \otimes \dots \otimes f_{\alpha s}) \otimes_H x_\alpha (v_{\tau^{-1}(i)} \otimes \dots \otimes v_{\tau^{-1}(i+m-1)})\label{xtau} \notag \\
    &= \epsilon(\tau) \sum_\alpha (f_{\alpha\,\widetilde{\tau}^{-1}(1)} \otimes \dots \otimes f_{\alpha\,\widetilde{\tau}^{-1}(s)}) \otimes_H x_\alpha (v_{\tau^{-1}(i)} \otimes \dots \otimes v_{\tau^{-1}(i+m-1)}),
\end{align}
and 
\begin{align}
    & (Y^\sigma)^{\Delta_0^{1 \dots 1m1 \dots 1}(\Gamma)}(v_1 \otimes \dots \otimes v_{i-1} \otimes x_\alpha(v_i \otimes \dots \otimes v_{i+m-1}) \notag \\
    & \quad \otimes v_{i+m} \otimes \dots \otimes v_{m+n-1}) \notag \\
    &= (\widetilde{\sigma} \otimes_H 1) Y^{\Gamma_0} \bigl(\sigma(v_1 \otimes \dots \otimes v_{i-1} \otimes x_\alpha(v_i \otimes \dots \otimes v_{i+m-1}) \notag \\
    & \quad \otimes v_{i+m} \otimes \dots \otimes v_{m+n-1})\bigr) \notag \\
    &= \epsilon(\sigma) \sum_{\beta} (g_{\beta\,\widetilde{\sigma}^{-1}(1)} \otimes \dots \otimes g_{\beta\,\widetilde{\sigma}^{-1}(t)}) \otimes_H y_\beta(v_{\sigma^{-1}(1)} \otimes \dots \otimes v_{\sigma^{-1}(\sigma(i)-1)} \notag \\
    & \quad \otimes x_\alpha(v_i \otimes \dots \otimes v_{i+m-1}) \otimes v_{\sigma^{-1}(\sigma(i)+1)} \otimes \dots \otimes v_{\sigma^{-1}(m+n-1)}). \label{ysigma}
\end{align}
Recall that $\epsilon(\sigma)$ and $\epsilon(\tau)$ are given by (\ref{sign3}); the subscripts are suppressed since the tensor products these permutations act on are clear.
Then we evaluate $Y^\sigma \circ_i X^\tau$ on $\Gamma$ as follows:
\begin{align}\label{circlekpermuted}
    & (Y^\sigma \circ_i X^\tau)^{\Gamma}(v_1 \otimes \dots \otimes v_{m+n-1}) \notag \\
    &= (-1)^{\overline{p}_{i}^X} \epsilon(\sigma) \epsilon(\tau) \sum_{\alpha,\beta} \rho^\Gamma_{i}\bigl(( 
    \underbrace{1 \otimes \dots \otimes 1}_{\widetilde{\sigma}(q) -1} \otimes \widetilde{\tau}(f_{\alpha1(1)} \otimes \dots \otimes f_{\alpha s(1)}) \otimes 1 \notag \\ 
    & \;\;\otimes \dots\otimes 1)\,(g_{\beta\,\widetilde{\sigma}^{-1}(1)} \otimes \dots \otimes g_{\beta\,\widetilde{\sigma}^{-1}(\widetilde{\sigma}(q) -1)} 
    \otimes \Delta^{(s-1)} (g_{\beta q})\otimes \dots \otimes g_{\beta\,\widetilde{\sigma}^{-1}(t)})\bigr)  \notag \\
    & \;\;\otimes_H y_\beta\bigl((f_{\alpha 1(-2)} \otimes \dots\otimes f_{\alpha s(-2)}) \cdot (v_{\sigma^{-1}(1)} \otimes \dots \otimes v_{\sigma^{-1}(\sigma(i)-1)}  \notag \\
    &\;\;\otimes x_\alpha(v_{\tau^{-1}(i)} \otimes \dots \otimes v_{\tau^{-1}(i+m-1)}) 
    \otimes \dots \otimes v_{\sigma^{-1}(m+n-1)})\bigr),
\end{align}
  where $g_{\beta q}$ in $\widetilde{\sigma}(g_{\beta 1} \otimes \dots \otimes g_{\beta t})$ corresponds to the connected component of $\Gamma_0$ that contains the vertex clasped from $\Gamma_i$, and
\begin{align*}
    \overline{p}_{i}^X = p(X) (p(v_{\sigma^{-1}(1)}) + \dots + p(v_{\sigma^{-1}(\sigma(i)-1)})).
\end{align*}


For the right-hand side $((Y \circ_{\sigma(i)} X)^{\sigma \circ_i \tau})^\Gamma$,
by Proposition \ref{graphcoequivariance}, we will evaluate $X$ and $Y$ on the graphs 
\begin{align*}
\Delta_{\sigma(i)}^{1 \dots 1m1 \dots 1}((\sigma \circ_i \tau)(\Gamma)) = \Gamma_i \quad\text{and}\quad
\Delta_0^{1 \dots 1m1 \dots 1}((\sigma \circ_i \tau)(\Gamma)) = \Gamma_0,
\end{align*}
respectively, where $m$ appears at the $\sigma(i)$-th position. Meanwhile, recall that \cite[(2.17)]{BDHK19}:
\begin{align}
    (\sigma &\circ_i \tau)(v_1 \otimes \dots \otimes v_{m+n-1}) \notag \\
    &=  \sigma (v_1 \otimes \dots \otimes v_{i-1} \otimes \tau(v_i \otimes \dots \otimes v_{i+m-1}) \otimes v_{i+m} \otimes \dots \otimes v_{m+n-1}) \notag \\
    &=  \epsilon(\sigma) \epsilon(\tau) \, v_{\sigma^{-1}(1)} \otimes \dots \otimes v_{\sigma^{-1}(\sigma(i)-1)} \otimes v_{\tau^{-1}(i)} \otimes \dots \otimes v_{\tau^{-1}(i+m-1)} \notag \\
    & \quad \otimes v_{\sigma^{-1}(\sigma(i)+1)} \otimes \dots \otimes v_{\sigma^{-1}(m+n-1)}.
\end{align}
Hence, we compute the right-hand side:
\begin{align}\label{circlekpermuted2}
    & ((Y \circ_{\sigma(i)} X)^{\sigma \circ_i \tau})^{\Gamma}(v_1 \otimes \dots \otimes v_{m+n-1}) \notag \\
    &= (\widetilde{(\sigma \circ_i \tau)} \otimes_H 1)(Y \circ_{\sigma(i)} X)^{(\sigma \circ_i \tau)\Gamma}((\sigma \circ_i \tau)(v_1 \otimes \dots \otimes v_{m+n-1})) \notag \\
    &= \epsilon(\sigma) \epsilon(\tau) (\widetilde{(\sigma \circ_i \tau)} \otimes_H 1) (Y \circ_{\sigma(i)} X)^{(\sigma \circ_i \tau)\Gamma}(v_{\sigma^{-1}(1)} \otimes \dots \otimes v_{\sigma^{-1}(\sigma(i)-1)} \notag \\
    & \quad\otimes v_{\tau^{-1}(i)} \otimes \dots \otimes v_{\tau^{-1}(i+m-1)} \otimes v_{\sigma^{-1}(\sigma(i)+1)} \otimes\dots \otimes v_{\sigma^{-1}(m+n-1)}) \notag \\
    &= (-1)^{\overline{p}_{i}^X} \epsilon(\sigma) \epsilon(\tau) \sum_{\alpha,\beta}\widetilde{(\sigma \circ_i \tau)} \bigl(\rho^{(\sigma \circ_i \tau)\Gamma}_{\sigma(i)}
    (g_{\beta 1} \otimes \dots \otimes g_{\beta\,q-1} \otimes f_{\alpha1(1)} g_{\beta q(1)} \notag \\
    & \quad\otimes \dots \otimes f_{\alpha s(1)} g_{\beta q(s)} \otimes g_{\beta\,q+1} \otimes \dots \otimes g_{\beta t})\bigr)
    \otimes_H y_\beta\bigl((f_{\alpha 1(-2)} \otimes\dots \otimes f_{\alpha s(-2)}) \notag \\
    & \quad\cdot (v_{\sigma^{-1}(1)} \otimes \dots \otimes v_{\sigma^{-1}(\sigma(i)-1)} \otimes x_\alpha(v_{\tau^{-1}(i)} \otimes \dots \otimes v_{\tau^{-1}(i+m-1)}) \notag \\
    & \quad\otimes v_{\sigma^{-1}(\sigma(i)+1)} \otimes \dots \otimes v_{\sigma^{-1}(m+n-1)})\bigr).
\end{align}

We see immediately that the vector parts are the same in the right-hand sides of \eqref{circlekpermuted} and \eqref{circlekpermuted2}.
For the coefficient parts, without loss of generality, we assume that $g_{\beta q} = 1$. Then it suffices to show that
\begin{align}\label{permutecommute}
    \rho^\Gamma_{i} (\widetilde{\sigma} \circ_{\widetilde{\sigma}(q)} \widetilde{\tau}) = \widetilde{(\sigma \circ_i \tau)}  \rho^{(\sigma \circ_i \tau)\Gamma}_{\sigma(i)} \in S_{s+t-1}
\end{align} 
when acting on any vector $h\in H^{\otimes (s+t-1)}$; in particular, for
\begin{align*}
h= g_{\beta 1} \otimes \dots \otimes g_{\beta\, q-1} \otimes f_{\alpha 1(1)} \otimes \dots \otimes f_{\alpha s(1)} \otimes g_{\beta\, q+1} \otimes \dots \otimes g_{\beta t}.
\end{align*}
In order to prove this, recall from \eqref{tildesigma} that $\widetilde{\sigma}$ identifies the $k$-th connected component of $\sigma\Gamma$ with the $\widetilde{\sigma}(k)$-th connected component of $\Gamma$. Consider the right-hand side of (\ref{permutecommute}), and pick a tensor factor in $h$ that represents a connected component of $\Gamma_i$ or $\Gamma_0$. It is first identified by $\rho_{\sigma(i)}^{(\sigma \circ_i \tau)\Gamma}$ with a connected component of $(\sigma \circ_i \tau)\Gamma$; then identified by $\widetilde{(\sigma \circ_i \tau)}$ with a connected component of $\Gamma$. Now consider the left-hand side of (\ref{permutecommute}), and note that $(\widetilde{\sigma} \circ_{\widetilde{\sigma}(q)} \widetilde{\tau})$ identifies this connected component in $h$ with a connected component in either $\Delta_0^{1 \dots 1m1 \dots 1}(\Gamma)$ or $\Delta_i^{1 \dots 1m1 \dots 1}(\Gamma)$, where $m$ appears at the $i$-th position. If it is a connected component in $\Gamma_0$, it will be identified as a connected component in $\Delta_0^{1 \dots 1m1 \dots 1}(\Gamma)$; while if it is a connected component in $\Gamma_i$, it will be identified as a connected component in $\Delta_i^{1 \dots 1m1 \dots 1}(\Gamma)$. Then $\rho^\Gamma_{i}$ identifies this connected component with a connected component of $\Gamma$. Hence, both sides identify connected components in $h$ with those of $\Gamma$, and they must be equal.
This proves the lemma.
\end{proof}

Combining the results of Lemmas \ref{operadeq1}, \ref{operadeq2} and \ref{operadeq3} concludes the proof of Theorem \ref{firstmainthm}.

\section{Poisson $H$-Pseudoalgebras}\label{sec4}

In this section, as an application of the construction of the generalized classical operad $\mathcal{P}^{cl}_H$,
we introduce the notion of a Poisson $H$-pseudoalgebra, define the variational and classical cohomology of Poisson $H$-pseudoalgebras, and provide examples.
As before, $H$ will be a cocommutative Hopf algebra.

\subsection{Lie pseudoalgebras}
Since every Poisson vertex algebra is in particular a Lie conformal algebra, let us start by recalling the notion of a Lie pseudoalgebra,
which generalizes Lie conformal algebras.

\begin{definition}[\cite{BDK01}]\label{defofLie pseudoalgebra}
A \emph{Lie $H$-pseudoalgebra} is a pair $(L, \beta)$, where $L$ is a superspace with parity $p$, which is a left $H$-module (with $H$ purely even), and $\beta$ is an even element of $\mathrm{Hom}_{H \otimes H}(L \otimes L, (H \otimes H) \otimes_H L)$ satisfying the axioms below. Writing the \emph{pseudobracket} $\beta$ as $\beta(a \otimes b) = [a*b]$,
the axioms are ($a,b,c \in L$, $f,g \in H$, $\sigma = (12) \in S_2$):
\begin{itemize}
    \item \textbf{$H$-bilinearity:} \begin{align}[fa*gb] = ((f \otimes g) \otimes _H 1)[a*b]; \label{hbil} \end{align} 
    \item \textbf{Skewsymmetry:} \begin{align}[b*a] = -(-1)^{p(a)p(b)}(\sigma \otimes _H \text{id})[a*b]; \label{skewsym} \end{align}
    \item \textbf{Jacobi identity:} \begin{align}\label{jacid}
    [a*[b*c]] - (-1)^{p(a)p(b)}((\sigma \otimes \text{id}) \otimes_H \text{id})[b*[a*c]] = [[a*b]*c].\end{align}
\end{itemize}
\end{definition}

\begin{remark}\label{bracketcomp}
In the Jacobi identity \eqref{jacid}, the composition $[[a*b]*c]$ is defined as follows \cite{BDK01}. 
Let us write
\begin{align}\label{bracketcomp3}
[a*b] &= \sum_i (f_i \otimes g_i) \otimes_H e_i, \\
[e_i*c] &= \sum_j (f_{ij} \otimes g_{ij}) \otimes_H e_{ij}, \label{bracketcomp4}
\end{align}
for some $f_i,g_i,f_{ij},g_{ij} \in H$ and $e_i,e_{ij} \in L$. Then
\begin{align}\label{bracketcomp2}
        [[a*b]*c] = \sum_{i,j} (f_if_{ij(1)} \otimes g_if_{ij(2)} \otimes g_{ij}) \otimes_H e_{ij}.
\end{align}
We point out that, while expressions \eqref{bracketcomp3}, \eqref{bracketcomp4} are not unique (as they involve $\otimes_H$),
the right-hand side of \eqref{bracketcomp2} depends only on $[a*b]$ and $c$.
Similarly, if 
\begin{align}\label{bracketcomp5}
[b*c] &= \sum_i (h_i \otimes l_i) \otimes_H d_i, \\
[a*d_i] &= \sum_j (h_{ij} \otimes l_{ij}) \otimes_H d_{ij}, \label{bracketcomp6}
\end{align}
then
\begin{align}\label{bracketcomp1}
    [a*[b*c]] = \sum_{i,j} (h_{ij} \otimes h_il_{ij(1)} \otimes l_il_{ij(2)}) \otimes_H d_{ij}.
\end{align}
\end{remark}

\begin{remark}
If $[a*b]$ is given by \eqref{bracketcomp3}, then the $H$-bilinearity \eqref{hbil} is equivalent to:
\begin{align}\label{hbil2}
[fa*gb] = \sum_i (ff_i \otimes gg_i) \otimes_H e_i, 
\end{align}
while the skewsymmetry \eqref{skewsym} is equivalent to:
\begin{align}\label{skewsym2}
[b*a] = -(-1)^{p(a)p(b)} \sum_i (g_i \otimes f_i) \otimes_H e_i.
\end{align}
Finally, note that $p(e_i)=p(a)+p(b)$ for all $i$, because the pseudobracket is even.
\end{remark}


As an example, we recall the Lie pseudoalgebra $W(\mathfrak{d})$, which is closely related to the Lie--Cartan algebra of vector fields
$W_N = \mathrm{Der}\,\mathbb{F}[[t_1, \dots, t_N]]$ (see \cite{BDK01} for more details).

\begin{example}[\textbf{Lie pseudoalgebra $W(\mathfrak{d})$}]\label{typew}
Let $H = U(\mathfrak{d})$ be the universal enveloping algebra of a finite-dimensional Lie algebra $\mathfrak{d}$. We define $W(\mathfrak{d})$ as the free left $H$-module $H \otimes \mathfrak{d}$ (where $H$ acts by multiplication on the first factor),
with the following pseudobracket:
\begin{align}\label{wbracket}
    [(f \otimes a)&*(g \otimes b)] = (f \otimes g) \otimes_H (1 \otimes [a,b]) \notag \\
    &+(fb \otimes g) \otimes_H (1 \otimes a)-(f \otimes ga) \otimes_H(1 \otimes b),
\end{align}
for $f,g\in H$, $a,b \in\mathfrak{d}$.
\end{example}

We refer to \cite{BDK01} for further examples of Lie pseudoalgebras.
An important special case is when
\begin{align}\label{hpolyn}
H = \mathbb{F}[\partial_1, \dots, \partial_N], \quad\text{with }\;
\Delta(\partial_i) = \partial_i \otimes 1 + 1 \otimes \partial_i,
\end{align}
is the universal enveloping algebra of an $N$-dimensional abelian Lie algebra.
In this case, the pseudobracket can be expressed equivalently as a \emph{$\vec{\lambda}$-bracket}
$L \otimes L \to L[\vec{\lambda}] := L[\lambda_1,\dots,\lambda_N]$, where
$\vec{\lambda} = (\lambda_1, \dots, \lambda_N)$. Explicitly \cite{BDK01}:
\begin{align}\label{twobracketid}
[a*b] =  \sum_i (f_i(\vec{\partial}) \otimes 1) \otimes_H e_i \;\;\Leftrightarrow\;\; [a_{\vec{\lambda}}b] = \sum_i f_i(-\vec{\lambda})e_i,
\end{align}
where $\vec{\partial}=(\partial_1, \dots, \partial_N)$ (recall that we can always arrange that all $g_i=1$ in \eqref{bracketcomp3}).
Then the axioms of a Lie pseudoalgebra coincide with the axioms \eqref{sesq-comp}--\eqref{jacN} of a \emph{Lie conformal algebra} in dimension $N$ \cite{BKV99}.

\subsection{Definition of a Poisson pseudoalgebra}

Motivated by the connection between the classical operad $\mathcal{P}^{cl}$ and the notion of a Poisson vertex algebra (see \cite[Sect.\ 10]{BDHK19}), 
our construction of the generalized classical operad $\mathcal{P}^{cl}_H$ leads to the following definition of a Poisson $H$-pseudoalgebra.

\begin{definition}
A \emph{Poisson $H$-pseudoalgebra} $V$ is a Lie $H$-pseudoalgebra with a pseudobracket $[a*b]$, equipped with a supercommutative associative product $ab \in V$
for $a,b\in V$, satisfying the following axioms:  
\begin{itemize}
\item \textbf{$H$-differential:}
\begin{align}\label{hdiff}
h(ab) = (h_{(1)} a)(h_{(2)} b), \qquad h\in H, \;\; a,b\in V;
\end{align}
\item \textbf{Leibniz rule:}
\begin{align}\label{leibniz}
    [a*bc] = [a*b]c + (-1)^{p(b)p(c)}
    [a*c]b,
\end{align}
\end{itemize}
where $[a*b]c$ is defined by
\begin{align}
    [a*b]c := \sum_i (f_i \otimes g_{i(1)}) \otimes_H 
    e_i(g_{i(-2)}c) \label{interaction}
\end{align}
for $[a*b]$ written in the form \eqref{bracketcomp3}.
\end{definition}

We remark that \eqref{hdiff} means that $V$ is an \emph{$H$-differential algebra}, i.e., the product $V\otimes V\to V$ is an $H$-module homomorphism.

\begin{lemma}\label{lintwd}
The right-hand side of \eqref{interaction} is well defined, i.e., it only depends on $[a*b]$ and $c$ but not on the expression \eqref{bracketcomp3}.
\end{lemma}
\begin{proof}
We need to check that, for any fixed $c\in V$, we have a well-defined linear map
\begin{align}\label{welldefined}
\Phi\colon (H \otimes H) \otimes_H V &\to (H \otimes H) \otimes_H V, \notag \\
(f \otimes g) \otimes_H e &\mapsto (f \otimes g_{(1)}) \otimes_H e(g_{(-2)}c).
\end{align}
Since the right-hand side of \eqref{welldefined} depends linearly on $f$, $g$ and $e$, it defines a linear map
\begin{align*}
\overline\Phi\colon H \otimes H \otimes V \to (H \otimes H) \otimes_H V.
\end{align*}
For $\overline\Phi$ to induce a map $\Phi$ on the quotient $(H \otimes H) \otimes_H V$ of $H \otimes H \otimes V$,
it is necessary and sufficient that 
\begin{align*}
\overline\Phi(f \otimes g \otimes he) = \overline\Phi(f h_{(1)} \otimes g h_{(2)} \otimes e) \qquad\text{for all }\; h\in H.
\end{align*}
To verify this, we compute using \eqref{coas}, \eqref{identity1} and \eqref{hdiff}:
\begin{align*}
\overline\Phi(f h_{(1)} &\otimes g h_{(2)} \otimes e) = (f h_{(1)} \otimes (g h_{(2)})_{(1)}) \otimes_H e( (g h_{(2)})_{(-2)}c) \\
&= (f h_{(1)} \otimes g_{(1)} h_{(2)}) \otimes_H e(h_{(-3)} g_{(-2)}c) \\
&= (f \otimes g_{(1)}) \otimes_H h_{(1)} \bigl( e(h_{(-2)} g_{(-2)}c) \bigr) \\
&= (f \otimes g_{(1)}) \otimes_H (h_{(1)} e) (h_{(2)} h_{(-3)} g_{(-2)}c) \\
&= (f \otimes g_{(1)}) \otimes_H (he)(g_{(-2)}c)
= \overline\Phi(f \otimes g \otimes he).
\end{align*}
This completes the proof.
\end{proof}

Similarly to \eqref{interaction}, we define the product
\begin{align}\label{right1}
    a[b*c] := \sum_i (h_{i(1)} \otimes l_i) \otimes_H (h_{i(-2)}a) d_i,
\end{align} 
for $[b*c]$ written in the form \eqref{bracketcomp5}. As in Lemma \ref{lintwd}, one can show that \eqref{right1} is well defined.
Then from the (left) Leibniz rule \eqref{leibniz} and the skewsymmetry \eqref{skewsym} of the pseudobracket, one can derive the following:

\begin{lemma}[\textbf{Right Leibniz rule}]
In any Poisson pseudoalgebra $V$, we have the right Leibniz rule
\begin{align}\label{leibnizr}
    [ab*c] = a[b*c]+(-1)^{p(a)p(b)}b[a*c].
\end{align}
\end{lemma}
\begin{proof}
Suppose that $[a*b]$ is written again in the form \eqref{bracketcomp3}; then due to skewsymmetry $[b*a]$ is given by \eqref{skewsym2}.
Thus, by \eqref{interaction},
\begin{align*}
    [b*a]c = -(-1)^{p(a)p(b)} \sum_i (g_i \otimes f_{i(1)}) \otimes_H e_i(f_{i(-2)}c).
\end{align*}
Note that $p(e_i)=p(a)+p(b)$ for all $i$, because the pseudobracket is even. Since $V$ is supercommutative,
we have
\begin{align*}
e_i(f_{i(-2)}c) = (-1)^{(p(a)+p(b))p(c)} (f_{i(-2)}c) e_i.
\end{align*}
Using this, we can relate \eqref{interaction} and \eqref{right1}:
\begin{align}\label{mediumstep}
    -(-1)^{p(a)p(b)}(\sigma &\otimes_H 1)([b*a]c) = \sum_i (f_{i(1)} \otimes g_i) \otimes_H  e_i (f_{i(-2)}c)\notag \\
    &= (-1)^{(p(a)+p(b))p(c)}c[a*b].
\end{align}
Now we derive the right Leibniz rule from the left Leibniz rule \eqref{leibniz} and the skewsymmetry \eqref{skewsym}:
\begin{align*}
    [ab*c] &= -(-1)^{(p(a)+p(b))p(c)} (\sigma \otimes_H 1) [c*ab] \notag \\
    &= -(-1)^{(p(a)+p(b))p(c)} (\sigma \otimes_H 1) ([c*a]b) \notag \\
    &\quad\,- (-1)^{(p(a)+p(b))p(c)+p(a)p(b)} (\sigma \otimes_H 1) ([c*b]a) \notag \\
    &= (-1)^{p(a)p(b)}b[a*c] + a[b*c],
\end{align*}
where we use (\ref{mediumstep}) in the last equation. 
\end{proof}

We will provide examples of Poisson pseudoalgebras in Sect.\ \ref{Poisson pseudoalgebra4examples} below.
In the special case when $H$ is the algebra of polynomials in $N$ variables as in \eqref{hpolyn},
the pseudobracket is equivalent to the $\vec{\lambda}$-bracket \eqref{twobracketid}.
Then the Leibniz rule \eqref{leibniz} can be written as:
\begin{align}\label{Poisson pseudoalgebraleibnizrule}
        [a_{\vec\lambda}bc] = [a_{\vec\lambda}b]c+ (-1)^{p(b)p(c)}[a_{\vec\lambda}c]b.
\end{align}
Thus, for $N=1$, $H=\mathbb{F}[\partial]$, our notion of a Poisson pseudoalgebra coincides
with that of a \emph{Poisson vertex algebra} (see \cite[Sect.\ 16]{BDK01} and \cite{BDK09}). 

Note that for Poisson vertex algebras, the right Leibniz rule looks more complicated than the left one
(cf.\ \cite[(1.26)]{BDK09}), while in our approach the left and right versions are symmetric.
We will derive a formula for the pseudobracket of two products, generalizing
the corresponding formula for Poisson vertex algebras \cite[(1.34)]{BDK09}.
In order to state the result, we first need to check that the two products defined by 
\eqref{interaction} and \eqref{right1} satisfy associativity.

\begin{lemma}
Define products
\begin{align*}
V \otimes \bigl( (H \otimes H) \otimes_H V \bigr) &\to (H \otimes H) \otimes_H V, \\
\bigl( (H \otimes H) \otimes_H V \bigr) \otimes V &\to (H \otimes H) \otimes_H V,
\end{align*}
by extending linearly the formulas
\begin{align*}
aB &:= (f_{(1)} \otimes g) \otimes_H (f_{(-2)}a)b, \\
Bc &:= (f \otimes g_{(1)}) \otimes_H b(g_{(-2)}c),
\end{align*}
respectively, where $a,c\in V$ and $B=(f \otimes g) \otimes_H b \in (H \otimes H) \otimes_H V$.
Then $(aB)c=a(Bc)$.
\end{lemma}
\begin{proof}
By a straightforward calculation, we have:
\begin{align*}
(aB)c &= \Bigl( (f_{(1)} \otimes g) \otimes_H (f_{(-2)}a)b \Bigr) c \\
&= (f_{(1)} \otimes g_{(1)}) \otimes_H \bigl((f_{(-2)}a)b \bigr) (g_{(-2)}c),
\end{align*}
and
\begin{align*}
a(Bc) &= a \Bigl( (f \otimes g_{(1)}) \otimes_H b(g_{(-2)}c) \Bigr) \\
&= (f_{(1)} \otimes g_{(1)}) \otimes_H (f_{(-2)}a) \bigl( b(g_{(-2)}c) \bigr).
\end{align*}
These two are equal since the product in $V$ is associative.
\end{proof}

Then we have the following:

\begin{proposition}[\textbf{Iterated Leibniz rule}]
For any Poisson pseudoalgebra $V$ and $a_i,b_j\in V$, we have
\begin{align}\label{iterleib}
[&(a_1 \cdots a_m)*(b_1 \cdots b_n)] 
\notag \\
&= \sum_{i=1}^m \sum_{j=1}^n
\epsilon_{ij} \, (a_1 \cdots a_{i-1} a_{i+1} \cdots a_m) \, [a_i*b_j] \, (b_1 \cdots b_{j-1} b_{j+1} \cdots b_n) \,,
\end{align}
where the sign
\begin{align*}
\epsilon_{ij} := (-1)^{ p(a_i)(p(a_{i+1}) + \dots + p(a_m)) + p(b_j)(p(b_1) + \dots + p(b_{j-1})) } \,.
\end{align*}
\end{proposition}
\begin{proof}
Follows from applying iteratively the left and right Leibniz rules \eqref{leibniz}, \eqref{leibnizr}.
\end{proof}

\subsection{Construction of Poisson $H$-pseudoalgebras from $\mathcal{P}_{H}^{cl}$}\label{Poisson pseudoalgebraConstruction}

In this subsection, we generalize \cite[Theorem 10.7]{BDHK19}, which established a connection between the classical operad and Poisson vertex algebra structures
on an $\mathbb{F}[\partial]$-module $V$.

As before, let $H$ be a cocommutative Hopf algebra, which is purely even, and $V$ be a vector superspace with parity $p$, which is also a left $H$-module. Let $\bar{p} = 1-p$ denote the reverse parity of $p$, and $\Pi V$ denote the same space as $V$ but with parity $\bar{p}$. Note that $\Pi V$ is a left $H$-module too.
Consider the generalized classical operad $\mathcal{P}^{cl}_H(\Pi V)$ corresponding to $\Pi V$
(before we were denoting $\mathcal{P}^{cl}_H(V)$ as $\mathcal{P}^{cl}_H$ for short).
We denote by
\begin{align*}
    W^{cl}_H(\Pi V) = W(\mathcal{P}^{cl}_H(\Pi V)) = \bigoplus_{n \geq -1}W_{H,n}^{cl}(\Pi V) 
\end{align*} 
the universal Lie superalgebra associated to the operad $\mathcal{P}^{cl}_H(\Pi V)$; see Sect.\ \ref{UniversalLie}.
Note that both the operad $\mathcal{P}^{cl}_H(\Pi V)$ and the Lie superalgebra $W^{cl}_H(\Pi V)$
are considered with the parity $\bar{p}$.

\begin{theorem}\label{operadPoisson pseudoalgebra}
 There is a bijective correspondence between odd elements $X\in W_{H,1}^{cl}(\Pi V)$ 
 such that $[X,X]=0$ and Poisson pseudoalgebra structures on $V$, given explicitly by:
\begin{align}\label{ppabrackets}
         ab = (-1)^{p(a)}X\hspace{1mm}^{\begin{tikzpicture}
\tikzset{>={Latex[width=2mm,length=1mm]}}
\tikzstyle{vertex} =[circle, fill = black,inner sep=0pt,minimum size=1mm,scale=1.2]
\node[vertex](v1) at (1,0){};
\node[vertex](v2) at (1.45,0){};
\draw [->,thick] (v1) to (v2);
\end{tikzpicture}}(a \otimes b), \quad [a*b] = (-1)^{p(a)}X\hspace{1mm}^{\begin{tikzpicture}
\tikzset{>={Latex[width=2mm,length=1mm]}}
\tikzstyle{vertex} =[circle, fill = black,inner sep=0pt,minimum size=1mm,scale=1.2]
\node[vertex](v1) at (1,0){};
\node[vertex](v2) at (1.35,0){};
\end{tikzpicture}}(a \otimes b),
\end{align}
     for $a,b \in V$.
\end{theorem}

\begin{proof}
We adapt the proof of Theorem 10.7 from \cite{BDHK19}. Recall that $W_{H,1}^{cl}(\Pi V)$ consists of all permutation invariant elements
$X \in \mathcal{P}^{cl}_H(\Pi V)(2)$, that is $X^\sigma = X$ where $\sigma = (12) \in S_2$. Elements $X \in \mathcal{P}^{cl}_H(\Pi V)(2)$
correspond to collections of linear maps (cf.\ \eqref{pclhdef}):
\begin{align*}
X^{\Gamma} \colon V \otimes V \rightarrow H^{\otimes s(\Gamma)} \otimes_H V, \qquad \Gamma \in G_0(2),
\end{align*}
satisfying the second cycle condition \eqref{cycle2} and the componentwise $H$-linearity \eqref{complinear}.
The second cycle condition gives that
\begin{align*}
    X\hspace{1mm}^{\begin{tikzpicture}
\tikzset{>={Latex[width=2mm,length=1mm]}}
\tikzstyle{vertex} =[circle, fill = black,inner sep=0pt,minimum size=1mm,scale=1.2]
\node[vertex](v1) at (1,0){};
\node[vertex](v2) at (1.45,0){};
\draw [->,thick] (v1) to (v2);
\end{tikzpicture}}
+ X^{\begin{tikzpicture}
\tikzset{>={Latex[width=2mm,length=1mm]}}
\tikzstyle{vertex} =[circle, fill = black,inner sep=0pt,minimum size=1mm,scale=1.2]
\node[vertex](v1) at (1,0){};
\node[vertex](v2) at (1.45,0){};
\draw [->,thick] (v2) to (v1);
\end{tikzpicture}}
= 0;
\end{align*}
therefore, $X$ is uniquely determined by the two maps
\begin{align*}
X\hspace{1mm}^{\begin{tikzpicture}
\tikzset{>={Latex[width=2mm,length=1mm]}}
\tikzstyle{vertex} =[circle, fill = black,inner sep=0pt,minimum size=1mm,scale=1.2]
\node[vertex](v1) at (1,0){};
\node[vertex](v2) at (1.35,0){};
\end{tikzpicture}}
&\colon V \otimes V \rightarrow (H \otimes H) \otimes_H V,
\\
X\hspace{1mm}^{\begin{tikzpicture}
\tikzset{>={Latex[width=2mm,length=1mm]}}
\tikzstyle{vertex} =[circle, fill = black,inner sep=0pt,minimum size=1mm,scale=1.2]
\node[vertex](v1) at (1,0){};
\node[vertex](v2) at (1.45,0){};
\draw [->,thick] (v1) to (v2);
\end{tikzpicture}}
&\colon V \otimes V \rightarrow H \otimes_H V \cong V.
\end{align*}
Then the componentwise $H$-linearity means that the map
$X\hspace{1mm}^{\begin{tikzpicture}
\tikzset{>={Latex[width=2mm,length=1mm]}}
\tikzstyle{vertex} =[circle, fill = black,inner sep=0pt,minimum size=1mm,scale=1.2]
\node[vertex](v1) at (1,0){};
\node[vertex](v2) at (1.35,0){};
\end{tikzpicture}}$
is $H$-bilinear, and
$X\hspace{1mm}^{\begin{tikzpicture}
\tikzset{>={Latex[width=2mm,length=1mm]}}
\tikzstyle{vertex} =[circle, fill = black,inner sep=0pt,minimum size=1mm,scale=1.2]
\node[vertex](v1) at (1,0){};
\node[vertex](v2) at (1.45,0){};
\draw [->,thick] (v1) to (v2);
\end{tikzpicture}}$
is $H$-linear.
Moreover, these two maps are even with respect to the parity $p$, because $\bar{p}(X)=\bar{1}$.

 Next, we claim that the condition $X^{(12)} = X$ is equivalent to the supercommutativity of the product and the skewsymmetry of the pseudobracket for the parity $p$.
 Indeed, if we evaluate $X^{(12)} = X$ on the connected $2$-graph, we get
\begin{align*}
    X\hspace{1mm}^{\begin{tikzpicture}
\tikzset{>={Latex[width=2mm,length=1mm]}}
\tikzstyle{vertex} =[circle, fill = black,inner sep=0pt,minimum size=1mm,scale=1.2]
\node[vertex](v1) at (1,0){};
\node[vertex](v2) at (1.45,0){};
\draw [->,thick] (v1) to (v2);
\end{tikzpicture}}(a \otimes b) &= (X^{(12)})^{\begin{tikzpicture}
\tikzset{>={Latex[width=2mm,length=1mm]}}
\tikzstyle{vertex} =[circle, fill = black,inner sep=0pt,minimum size=1mm,scale=1.2]
\node[vertex](v1) at (1,0){};
\node[vertex](v2) at (1.45,0){};
\draw [->,thick] (v1) to (v2);
\end{tikzpicture}}(a \otimes b) = (-1)^{\bar{p}(a)\bar{p}(b)}X^{\begin{tikzpicture}
\tikzset{>={Latex[width=2mm,length=1mm]}}
\tikzstyle{vertex} =[circle, fill = black,inner sep=0pt,minimum size=1mm,scale=1.2]
\node[vertex](v1) at (1,0){};
\node[vertex](v2) at (1.45,0){};
\draw [->,thick] (v2) to (v1);
\end{tikzpicture}}(b \otimes a) \notag \\
&= (-1)^{p(a)+p(b)+p(a)p(b)}X^{\begin{tikzpicture}
\tikzset{>={Latex[width=2mm,length=1mm]}}
\tikzstyle{vertex} =[circle, fill = black,inner sep=0pt,minimum size=1mm,scale=1.2]
\node[vertex](v1) at (1,0){};
\node[vertex](v2) at (1.45,0){};
\draw [->,thick] (v1) to (v2);
\end{tikzpicture}}(b \otimes a),
\end{align*}
which implies
\begin{align*}
    ab = (-1)^{p(a)}X\hspace{1mm}^{\begin{tikzpicture}
\tikzset{>={Latex[width=2mm,length=1mm]}}
\tikzstyle{vertex} =[circle, fill = black,inner sep=0pt,minimum size=1mm,scale=1.2]
\node[vertex](v1) at (1,0){};
\node[vertex](v2) at (1.45,0){};
\draw [->,thick] (v1) to (v2);
\end{tikzpicture}}(a \otimes b) = (-1)^{p(b)+p(a)p(b)}X^{\begin{tikzpicture}
\tikzset{>={Latex[width=2mm,length=1mm]}}
\tikzstyle{vertex} =[circle, fill = black,inner sep=0pt,minimum size=1mm,scale=1.2]
\node[vertex](v1) at (1,0){};
\node[vertex](v2) at (1.45,0){};
\draw [->,thick] (v1) to (v2);
\end{tikzpicture}}(b \otimes a) = (-1)^{p(a)p(b)}ba.
\end{align*}
If we evaluate $X^{(12)} = X$ on the disconnected $2$-graph, we get
\begin{align*}
    X\hspace{1mm}^{\begin{tikzpicture}
\tikzset{>={Latex[width=2mm,length=1mm]}}
\tikzstyle{vertex} =[circle, fill = black,inner sep=0pt,minimum size=1mm,scale=1.2]
\node[vertex](v1) at (1,0){};
\node[vertex](v2) at (1.35,0){};
\end{tikzpicture}}(a \otimes b) &=  (X^{(12)})\hspace{1mm}^{\begin{tikzpicture}
\tikzset{>={Latex[width=2mm,length=1mm]}}
\tikzstyle{vertex} =[circle, fill = black,inner sep=0pt,minimum size=1mm,scale=1.2]
\node[vertex](v1) at (1,0){};
\node[vertex](v2) at (1.35,0){};
\end{tikzpicture}}(a \otimes b) = (-1)^{\bar{p}(a)\bar{p}(b)}(\sigma \otimes_H 1)X\hspace{1mm}^{\begin{tikzpicture}
\tikzset{>={Latex[width=2mm,length=1mm]}}
\tikzstyle{vertex} =[circle, fill = black,inner sep=0pt,minimum size=1mm,scale=1.2]
\node[vertex](v1) at (1,0){};
\node[vertex](v2) at (1.35,0){};
\end{tikzpicture}}(b \otimes a),
\end{align*}
which implies
\begin{align*}
[a*b]  = (-1)^{p(a)}X\hspace{1mm}^{\begin{tikzpicture}
\tikzset{>={Latex[width=2mm,length=1mm]}}
\tikzstyle{vertex} =[circle, fill = black,inner sep=0pt,minimum size=1mm,scale=1.2]
\node[vertex](v1) at (1,0){};
\node[vertex](v2) at (1.35,0){};
\end{tikzpicture}}(a \otimes b) = -(-1)^{p(a)p(b)}(\sigma \otimes_H 1)[b*a].
\end{align*}

Now we will show that $[X,X]=2X \square X =0$ is equivalent to three properties: associativity of the product $ab$, the Jacobi identity of the pseudobracket $[a*b]$ and the Leibniz rule. We will do so by evaluating $X \square X$ on the $3$-graphs below, because $X \square X$ is invariant under the action of the symmetric group and any acyclic $3$-graph is obtained from one of these under this action (cf.\ Example \ref{exg012}):

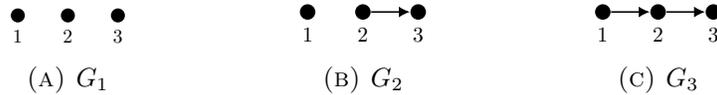
\begin{figure}[h]
     \begin{subfigure}[b]{0.3\textwidth}
          \centering
          \resizebox{1.8cm}{!}{
          \begin{tikzpicture}
          \tikzset{>={Latex[width=2mm,length=2mm]}}
          \tikzstyle{vertex} =[circle, fill = black,inner sep=0pt,minimum size=2mm,scale=1.4]
          \node[vertex,label=below:$1$](v1) at (1,0){};
          \node[vertex,label=below:$2$](v2) at (2,0){};
          \node[vertex,label=below:$3$](v3) at (3,0){};
          \end{tikzpicture}
          }
          \caption{$G_1$}
     \end{subfigure}
     \begin{subfigure}[b]{0.3\textwidth}
          \centering
          \resizebox{1.8cm}{!}{\begin{tikzpicture}
          \tikzset{>={Latex[width=2mm,length=2mm]}}
          \tikzstyle{vertex} =[circle, fill = black,inner sep=0pt,minimum size=2mm,scale=1.4]
          \node[vertex,label=below:$1$](v1) at (1,0){};
          \node[vertex,label=below:$2$](v2) at (2,0){};
          \node[vertex,label=below:$3$](v3) at (3,0){};
          \draw [->,thick] (v2) to (v3);
          \end{tikzpicture}} 
          \caption{$G_2$}
     \end{subfigure}
     \begin{subfigure}[b]{0.3\textwidth}
          \centering
          \resizebox{1.8cm}{!}{\begin{tikzpicture}
          \tikzset{>={Latex[width=2mm,length=2mm]}}
          \tikzstyle{vertex} =[circle, fill = black,inner sep=0pt,minimum size=2mm,scale=1.4]
          \node[vertex,label=below:$1$](v1) at (1,0){};
          \node[vertex,label=below:$2$](v2) at (2,0){};
          \node[vertex,label=below:$3$](v3) at (3,0){};
          \draw [->,thick] (v1) to (v2);
          \draw [->,thick] (v2) to (v3);
          \end{tikzpicture}}  
          \caption{$G_3$}
     \end{subfigure}
 \end{figure}

\noindent
Recall formula (\ref{square1}):
\begin{align*}
    X \square X = X \circ_1 X + X \circ_2 X + (X \circ_2 X)^{(12)}.
\end{align*}
Evaluating each summand above on graph $G_1$, we have:
\begin{align*}
(X \circ_1 X&)^{\begin{tikzpicture}
\tikzset{>={Latex[width=2mm,length=1mm]}}
\tikzstyle{vertex} =[circle, fill = black,inner sep=0pt,minimum size=1mm,scale=1.2]
\node[vertex](v1) at (1,0){};
\node[vertex](v2) at (1.35,0){};
\node[vertex](v3) at (1.7,0){};
\end{tikzpicture}}(a \otimes b \otimes c) 
= (-1)^{p(b)}[[a*b]*c], \\
(X \circ_2 X&)^{\begin{tikzpicture}
\tikzset{>={Latex[width=2mm,length=1mm]}}
\tikzstyle{vertex} =[circle, fill = black,inner sep=0pt,minimum size=1mm,scale=1.2]
\node[vertex](v1) at (1,0){};
\node[vertex](v2) at (1.35,0){};
\node[vertex](v3) at (1.7,0){};
\end{tikzpicture}}(a \otimes b \otimes c)
= (-1)^{1+p(b)}[a*[b*c]],  \\
((X \circ_2 X&)^{(12)})^{\begin{tikzpicture}
\tikzset{>={Latex[width=2mm,length=1mm]}}
\tikzstyle{vertex} =[circle, fill = black,inner sep=0pt,minimum size=1mm,scale=1.2]
\node[vertex](v1) at (1,0){};
\node[vertex](v2) at (1.35,0){};
\node[vertex](v3) at (1.7,0){};
\end{tikzpicture}} (a \otimes b \otimes c) \\
&= (-1)^{\bar{p}(a)\bar{p}(b)}((\sigma \otimes 1) \otimes_H 1)(X \circ_2 X)^{\begin{tikzpicture}
\tikzset{>={Latex[width=2mm,length=1mm]}}
\tikzstyle{vertex} =[circle, fill = black,inner sep=0pt,minimum size=1mm,scale=1.2]
\node[vertex](v1) at (1,0){};
\node[vertex](v2) at (1.35,0){};
\node[vertex](v3) at (1.7,0){};
\end{tikzpicture}} 
(b \otimes a \otimes c) \notag \\
&= (-1)^{p(b)+p(a)p(b)}((\sigma \otimes 1) \otimes_H 1)[b*[a*c]].
\end{align*}
Hence, $(X \square X)^{G_1} =0$ is equivalent to the Jacobi identity \eqref{jacid}.

Next, we evaluate $X \square X$ on graph $G_2$. Similarly to the previous case, we have:
\begin{align*}
    (X \circ_1 X)^{\begin{tikzpicture}
\tikzset{>={Latex[width=2mm,length=1mm]}}
\tikzstyle{vertex} =[circle, fill = black,inner sep=0pt,minimum size=1mm,scale=1.2]
\node[vertex](v1) at (1,0){};
\node[vertex](v2) at (1.4,0){};
\node[vertex](v3) at (1.8,0){};
\draw[->,thick] (v2) to (v3);
\end{tikzpicture}}(a \otimes b \otimes c)
&= (-1)^{p(b)}[a*b]c,  \\
(X \circ_2 X)^{\begin{tikzpicture}
\tikzset{>={Latex[width=2mm,length=1mm]}}
\tikzstyle{vertex} =[circle, fill = black,inner sep=0pt,minimum size=1mm,scale=1.2]
\node[vertex](v1) at (1,0){};
\node[vertex](v2) at (1.4,0){};
\node[vertex](v3) at (1.8,0){};
\draw[->,thick] (v2) to (v3);
\end{tikzpicture}} (a \otimes b \otimes c) &= (-1)^{\bar{p}(a)} X^{\begin{tikzpicture}
\tikzset{>={Latex[width=2mm,length=1mm]}}
\tikzstyle{vertex} =[circle, fill = black,inner sep=0pt,minimum size=1mm,scale=1.2]
\node[vertex](v1) at (1,0){};
\node[vertex](v2) at (1.4,0){};
\end{tikzpicture}} (a \otimes X^{\begin{tikzpicture}
\tikzset{>={Latex[width=2mm,length=1mm]}}
\tikzstyle{vertex} =[circle, fill = black,inner sep=0pt,minimum size=1mm,scale=1.2]
\node[vertex](v1) at (1,0){};
\node[vertex](v2) at (1.35,0){};
\draw[->,thick] (v1) to (v2);
\end{tikzpicture}} (b \otimes c)) \notag \\ 
&= (-1)^{1+p(b)}[a*bc], 
\end{align*}
For $(X \circ_2 X)^{(12)}$, we use the identity $(X \circ_2 X)^{(12)} =(X \circ_1 X)^{(23)}$ (cf.\ (\ref{square1})) to get:
\begin{align*}
    ((X \circ_2 X)^{(12)})^{\begin{tikzpicture}
\tikzset{>={Latex[width=2mm,length=1mm]}}
\tikzstyle{vertex} =[circle, fill = black,inner sep=0pt,minimum size=1mm,scale=1.2]
\node[vertex](v1) at (1,0){};
\node[vertex](v2) at (1.4,0){};
\node[vertex](v3) at (1.8,0){};
\draw[->,thick] (v2) to (v3);
\end{tikzpicture}} (a \otimes b \otimes c) &= ((X \circ_1 X)^{(23)})^{\begin{tikzpicture}
\tikzset{>={Latex[width=2mm,length=1mm]}}
\tikzstyle{vertex} =[circle, fill = black,inner sep=0pt,minimum size=1mm,scale=1.2]
\node[vertex](v1) at (1,0){};
\node[vertex](v2) at (1.4,0){};
\node[vertex](v3) at (1.8,0){};
\draw[->,thick] (v2) to (v3);
\end{tikzpicture}} (a \otimes b \otimes c)  \notag \\
&= (-1)^{\bar p (b) \bar p(c)}(X \circ_1 X)^{\begin{tikzpicture}
\tikzset{>={Latex[width=2mm,length=1mm]}}
\tikzstyle{vertex} =[circle, fill = black,inner sep=0pt,minimum size=1mm,scale=1.2]
\node[vertex](v1) at (1,0){};
\node[vertex](v2) at (1.4,0){};
\node[vertex](v3) at (1.8,0){};
\draw[->,thick] (v3) to (v2);
\end{tikzpicture}} (a \otimes c \otimes b)  \notag \\
&= -(-1)^{\bar p (b) \bar p(c)}(X \circ_1 X)^{\begin{tikzpicture}
\tikzset{>={Latex[width=2mm,length=1mm]}}
\tikzstyle{vertex} =[circle, fill = black,inner sep=0pt,minimum size=1mm,scale=1.2]
\node[vertex](v1) at (1,0){};
\node[vertex](v2) at (1.4,0){};
\node[vertex](v3) at (1.8,0){};
\draw[->,thick] (v2) to (v3);
\end{tikzpicture}} (a \otimes c \otimes b)  \notag \\
&= (-1)^{p(b)+p(b)p(c)}[a*c]b.
\end{align*}
Thus, $(X \square X)^{G_2} =0$ is equivalent to the Leibniz rule \eqref{leibniz}.

Finally, we evaluate $X \square X$ on graph $G_3$:
\begin{align*}
    (X \circ_1 X&)^{\begin{tikzpicture}
\tikzset{>={Latex[width=2mm,length=1mm]}}
\tikzstyle{vertex} =[circle, fill = black,inner sep=0pt,minimum size=1mm,scale=1.2]
\node[vertex](v1) at (1,0){};
\node[vertex](v2) at (1.4,0){};
\node[vertex](v3) at (1.8,0){};
\draw[->,thick] (v1) to (v2);
\draw[->,thick] (v2) to (v3);
\end{tikzpicture}}(a \otimes b \otimes c) = X^{\begin{tikzpicture}
\tikzset{>={Latex[width=2mm,length=1mm]}}
\tikzstyle{vertex} =[circle, fill = black,inner sep=0pt,minimum size=1mm,scale=1.2]
\node[vertex](v1) at (1,0){};
\node[vertex](v2) at (1.4,0){};
\draw[->,thick] (v1) to (v2);
\end{tikzpicture}} (X^{\begin{tikzpicture}
\tikzset{>={Latex[width=2mm,length=1mm]}}
\tikzstyle{vertex} =[circle, fill = black,inner sep=0pt,minimum size=1mm,scale=1.2]
\node[vertex](v1) at (1,0){};
\node[vertex](v2) at (1.35,0){};
\draw[->,thick] (v1) to (v2);
\end{tikzpicture}} (a \otimes b) \otimes c)
= (-1)^{p(b)}(ab)c, \\
(X \circ_2 X&)^{\begin{tikzpicture}
\tikzset{>={Latex[width=2mm,length=1mm]}}
\tikzstyle{vertex} =[circle, fill = black,inner sep=0pt,minimum size=1mm,scale=1.2]
\node[vertex](v1) at (1,0){};
\node[vertex](v2) at (1.4,0){};
\node[vertex](v3) at (1.8,0){};
\draw[->,thick] (v1) to (v2);
\draw[->,thick] (v2) to (v3);
\end{tikzpicture}} (a \otimes b \otimes c) = (-1)^{\bar{p}(a)} X^{\begin{tikzpicture}
\tikzset{>={Latex[width=2mm,length=1mm]}}
\tikzstyle{vertex} =[circle, fill = black,inner sep=0pt,minimum size=1mm,scale=1.2]
\node[vertex](v1) at (1,0){};
\node[vertex](v2) at (1.4,0){};
\draw[->,thick] (v1) to (v2);
\end{tikzpicture}} (a \otimes X^{\begin{tikzpicture}
\tikzset{>={Latex[width=2mm,length=1mm]}}
\tikzstyle{vertex} =[circle, fill = black,inner sep=0pt,minimum size=1mm,scale=1.2]
\node[vertex](v1) at (1,0){};
\node[vertex](v2) at (1.35,0){};
\draw[->,thick] (v1) to (v2);
\end{tikzpicture}} (b \otimes c)) \\
&= (-1)^{1+p(b)}a(bc), \\
((X \circ_2 X&)^{(12)})^{\begin{tikzpicture}
\tikzset{>={Latex[width=2mm,length=1mm]}}
\tikzstyle{vertex} =[circle, fill = black,inner sep=0pt,minimum size=1mm,scale=1.2]
\node[vertex](v1) at (1,0){};
\node[vertex](v2) at (1.4,0){};
\node[vertex](v3) at (1.8,0){};
\draw[->,thick] (v1) to (v2);
\draw[->,thick] (v2) to (v3);
\end{tikzpicture}} (a \otimes b \otimes c) = (-1)^{\bar{p}(a)\bar{p}(b)}(X \circ_2 X)^{\begin{tikzpicture}
\tikzset{>={Latex[width=2mm,length=1mm]}}
\tikzstyle{vertex} =[circle, fill = black,inner sep=0pt,minimum size=1mm,scale=1.2]
\node[vertex](v1) at (1,0){};
\node[vertex](v2) at (1.4,0){};
\node[vertex](v3) at (1.8,0){};
\draw[->,thick] (v1) to [out=45,in=135] (v3);
\draw[->,thick] (v2) to (v1);
\end{tikzpicture}} (b \otimes a \otimes c)  \notag \\ &= (-1)^{\bar{p}(b)+\bar{p}(a)\bar{p}(b)}X^{\begin{tikzpicture}
\tikzset{>={Latex[width=2mm,length=1mm]}}
\tikzstyle{vertex} =[circle, fill = black,inner sep=0pt,minimum size=1mm,scale=1.2]
\node[vertex](v1) at (1,0){};
\node[vertex](v2) at (1.4,0){};
\draw[->,thick] (v1) to [out=45,in=135] (v2);
\draw[->,thick] (v2) to [out=-135,in=-45] (v1);
\end{tikzpicture}} (\cdots)
= 0,
\end{align*}
using the first cycle condition for $X$.
Thus, $(X \square X)^{G_3} =0$ is equivalent to the associativity of the product: $(ab)c = a(bc)$.
\end{proof}

\begin{remark}
As an application of the proof of Theorem \ref{operadPoisson pseudoalgebra}, one can also derive the right Leibniz rule (\ref{leibnizr}) by evaluating $X \square X$ on the graph \begin{tikzpicture}
\tikzset{>={Latex[width=2mm,length=1mm]}}
\tikzstyle{vertex} =[circle, fill = black,inner sep=0pt,minimum size=1mm,scale=1.2]
\node[vertex](v1) at (1,0){};
\node[vertex](v2) at (1.4,0){};
\node[vertex](v3) at (1.8,0){};
\draw[->,thick] (v1) to (v2);
\end{tikzpicture}\,.
\end{remark}

\subsection{Examples of Poisson pseudoalgebras}\label{Poisson pseudoalgebra4examples}

Every Poisson pseudoalgebra is in particular a Lie pseudoalgebra. Conversely, given a Lie pseudoalgebra $L$, we can generate a Poisson pseudoalgebra from it 
by taking the symmetric superalgebra $S(L)$ over $L$. Note that $S(L)$ has an associative supercommutative product. The left action of $H$ on $L$ extends uniquely to $S(L)$,
so that $S(L)$ is an $H$-differential algebra (see \eqref{hdiff}). The pseudobracket on $L$ extends uniquely to $S(L)$ via the iterated Leibniz rule \eqref{iterleib}.

\begin{proposition}
Given a Lie pseudoalgebra $L$, the symmetric superalgebra $S(L)$ has a canonical structure of a Poisson pseudoalgebra described above.
\end{proposition}
\begin{proof}
In the case of Poisson vertex algebras (corresponding to $H=\mathbb{F}[\partial]$), a detailed proof is given in \cite[Theorem 1.15]{BDK09}. As our approach is more symmetric and our formula \eqref{iterleib} is the same as in the well-known case of Lie superalgebras (when $H=\mathbb F$), the proof in our setting is simpler.
\end{proof}

For the rest of this subsection, we assume that $H=U(\mathfrak{d})$ is the universal enveloping algebra of an $N$-dimensional Lie algebra $\mathfrak{d}$. We fix a basis $\{\partial_1,\dots,\partial_N\}$ for $\mathfrak d$, and the Poincar\'e--Birkhoff--Witt basis of $H$ consisting of all ordered monomials
\begin{equation*}
\partial^I := \partial_1^{i_1} \cdots \partial_N^{i_N} \,, \qquad I=(i_1,\dots,i_N) \in \mathbb{Z}_{\ge0}^N \,.
\end{equation*}
The first three examples below are adapted from \cite{BDK20}, and the last two from \cite{BDK01}.

\begin{example}[\textbf{Free superboson}]\label{freeboson}
Let $\mathfrak{g}$ be a finite-dimensional vector superspace with parity $p$ and a supersymmetric nondegenerate bilinear map $\beta\colon \mathfrak{g} \times \mathfrak{g} \rightarrow \mathfrak{d}$. The supersymmetry of $\beta$ means that $\beta(a,b) = (-1)^{p(a)p(b)}\beta(b,a)$ for $a,b \in \mathfrak{g}$ and $\beta(a,b) =0$ whenever $p(a) \neq p(b)$. The \emph{free superboson Lie pseudoalgebra} is the left $H$-module
\begin{align*}
    L^\beta_{\mathfrak{g}} := (H \otimes \mathfrak{g}) \oplus \mathbb{F}K, \quad\text{where}\quad p(K) = \bar{0}, \quad \mathfrak{d} K = 0,
\end{align*}
with the pseudobracket given by:
\begin{align}\label{bosonbracket}
    [a*b] = (\beta(a,b) \otimes 1) \otimes_H K, \quad [a*K]=0, \qquad a,b \in \mathfrak{g},
\end{align}
uniquely extended to $L^\beta_{\mathfrak{g}}$ by $H$-bilinearity \eqref{hbil}. 
When $\mathfrak{g}$ is purely even, we get the \emph{free boson Lie pseudoalgebra}.
The symmetric superalgebra $S(L^\beta_{\mathfrak{g}})$ is a Poisson pseudoalgebra, in which the element $K$ is central, i.e., it has a trivial pseudobracket with
every other element. The quotient Poisson pseudoalgebra
\begin{align*}
    B^\beta_{\mathfrak{g}} := S(L^\beta_{\mathfrak{g}}) / S(L^\beta_{\mathfrak{g}}) (K-1) \cong S(H \otimes \mathfrak{g})
\end{align*}
is called the \emph{free superboson Poisson pseudoalgebra}.
The pseudobracket in $B^\beta_{\mathfrak{g}}$ is given by \eqref{bosonbracket} with $K=1$ on the generators from $\mathfrak{g}$, and then is extended uniquely
by $H$-bilinearity \eqref{hbil} and the iterated Leibniz rule \eqref{iterleib}. In the case when $\mathfrak{g}$ is purely even, $B^\beta_{\mathfrak{g}}$ is the space of polynomials
in $\partial^I u_l$ where $I \in \mathbb{Z}_{\ge0}^N$ and $\{u_l\}$ is a basis for $\mathfrak{g}$.
\end{example}

\begin{example}[\textbf{Free superfermion}]\label{fermionexample}
Let $\mathfrak{g}$ be a finite-dimensional vector superspace with parity $p$ and a super-skewsymmetric nondegenerate bilinear form
$\gamma\colon\mathfrak{g} \times \mathfrak{g} \rightarrow \mathbb{F}$. The super-skewsymmetry of $\gamma$ means that $\gamma(a,b) = -(-1)^{p(a)p(b)}\gamma(b,a)$ for $a,b \in \mathfrak{g}$ and $\gamma(a,b) = 0$ if $p(a) \neq p(b)$. The \emph{free superfermion Lie pseudoalgebra} is the left $H$-module
\begin{align*}
    L^\gamma_{\mathfrak{g}} := (H \otimes \mathfrak{g}) \oplus \mathbb{F}K, \quad\text{where}\quad p(K) = \bar{0}, \quad \mathfrak{d} K = 0,
\end{align*}
with the pseudobracket
\begin{align}\label{fermionbracket}
    [a*b] = (\gamma(a,b) \otimes 1) \otimes_H K, \quad [a*K]=0, \qquad a,b \in \mathfrak{g},
\end{align}
uniquely extended to $L^\gamma_{\mathfrak{g}}$ by $H$-bilinearity \eqref{hbil}. When $\mathfrak{g}$ is purely even, $L^\gamma_{\mathfrak{g}}$ is called the \emph{free fermion Lie pseudoalgebra}. 
The quotient Poisson pseudoalgebra
\begin{align*}
    F^\gamma_{\mathfrak{g}} := S(L^\gamma_{\mathfrak{g}}) / S(L^\gamma_{\mathfrak{g}}) (K-1) \cong S(H \otimes \mathfrak{g})
\end{align*}
is called the \emph{free superfermion Poisson pseudoalgebra}.
Its pseudobracket is given by \eqref{fermionbracket} with $K=1$ on the generators from $\mathfrak{g}$, and then is extended uniquely
by $H$-bilinearity \eqref{hbil} and the iterated Leibniz rule \eqref{iterleib}.
In the case when $\mathfrak{g}$ is purely odd, $F^\gamma_{\mathfrak{g}}$ is the exterior algebra in $\partial^I u_l$ where $I \in \mathbb{Z}_{\ge0}^N$ and $\{u_l\}$ is a basis for $\mathfrak{g}$.
\end{example}

\begin{example}[\textbf{Affine Poisson pseudoalgebra}]
Let $\mathfrak{g}$ be a Lie algebra with a nondegenerate invariant symmetric bilinear map $\beta\colon\mathfrak{g} \times \mathfrak{g} \rightarrow \mathfrak{d}$. 
The invariance of $\beta$ means that $\beta([a,b],c)=\beta(a,[b,c])$ for all $a,b,c\in\mathfrak{g}$.
The \emph{affine Lie pseudoalgebra} is defined as the purely even left $H$-module
\begin{align*}
    L^\beta_{\mathfrak{g}} := (H \otimes \mathfrak{g}) \oplus \mathbb{F}K, \quad\text{where}\quad \mathfrak{d} K = 0,
\end{align*}
with the pseudobracket $(a,b \in \mathfrak{g})$:
\begin{align}\label{affinebracket}
    [a*b] = (1 \otimes 1) \otimes_H [a,b] + (\beta(a,b) \otimes 1) \otimes_H K, \qquad [a*K]=0,
\end{align}
uniquely extended to $L^\beta_{\mathfrak{g}}$ by $H$-bilinearity \eqref{hbil}.
The quotient Poisson pseudoalgebra
\begin{align*}
    A^\beta_{\mathfrak{g}} := S(L^\beta_{\mathfrak{g}}) / S(L^\beta_{\mathfrak{g}}) (K-1) \cong S(H \otimes \mathfrak{g})
\end{align*}
is called the \emph{affine Poisson pseudoalgebra}.
Its pseudobracket is given by \eqref{affinebracket} with $K=1$ on the generators from $\mathfrak{g}$, and then is extended uniquely
by $H$-bilinearity \eqref{hbil} and the iterated Leibniz rule \eqref{iterleib}.
\end{example}

\begin{example}[\textbf{Poisson pseudoalgebra of type $W$}]\label{wmppa}
Let $\mathfrak{d}$ be a Lie algebra with a symmetric bilinear map $\beta\colon\mathfrak{d} \times \mathfrak{d} \rightarrow \mathfrak{d}$
such that $\beta([a,b],c)=\beta(a,[b,c])$ for $a,b,c\in\mathfrak{d}$.
Consider a central extension of the Lie pseudoalgebra $W(\mathfrak{d}) = H \otimes \mathfrak{d}$ from Example \ref{typew} by an even element $C$ with $\mathfrak{d} C = 0$
and the pseudobracket $(a,b \in \mathfrak{d})$:
\begin{align}\label{wmbracket}
    [(1 \otimes a)*(1 \otimes b)] &= (1 \otimes 1) \otimes_H (1 \otimes [a,b]) + (\beta(a,b) \otimes 1) \otimes_H C \notag \\
    &+(b \otimes 1) \otimes_H (1 \otimes a)-(1 \otimes a) \otimes_H(1 \otimes b).
\end{align}
Taking the quotient of its symmetric algebra by the ideal generated by $C-1$, we obtain the \emph{Poisson pseudoalgebra of type $W$}:
\begin{align*}
    P^\beta_W := S(W(\mathfrak{d}) \oplus \mathbb{F}C) / S(W(\mathfrak{d}) \oplus \mathbb{F}C)(C-1) \cong S(H \otimes \mathfrak{d}),
\end{align*}
with the pseudobracket \eqref{wmbracket} with $C=1$ on the generators from $\mathfrak{d}$, extended uniquely by $H$-bilinearity \eqref{hbil} and the iterated Leibniz rule \eqref{iterleib}.
\end{example}

\begin{example}[\textbf{Poisson pseudoalgebra of type $K$}]\label{kppa}
Let $\mathfrak{d}$ be the Heisenberg Lie algebra of dimension $N=2M+1$, with a basis $\partial_0,\partial_1,\dots,\partial_{2M}$ and 
\begin{align*}
[\partial_i,\partial_{M+i}] = - [\partial_{M+i},\partial_i] = \partial_0 \,, \qquad 1\le i\le M;
\end{align*}
all other brackets equal to zero. Then the free $H$-module $He$ has a Lie pseudoalgebra structure, given by \cite[Example 8.4]{BDK01}:
\begin{align*}
[e*e] &= \alpha \otimes_H e \,, \quad\text{where} \\
\alpha &:= 1\otimes \partial_0 - \partial_0 \otimes 1 + \sum_{i=1}^M \bigl( \partial_{i} \otimes \partial_{M+i} - \partial_{M+i} \otimes \partial_{i} \bigr).
\end{align*}
This Lie pseudoalgebra is denoted as $K(\mathfrak{d},\theta)$, where $\theta\in\mathfrak{d}^*$ is defined by $\theta(\partial_i) = \delta_{i,0}$.
There exists a central extension of $K(\mathfrak{d},\theta)$ by an even element $C$ with $\mathfrak{d} C = 0$
and the pseudobracket
\begin{align}\label{kmbracket}
[e*e] = \alpha \otimes_H e + (\partial_0 \otimes 1) \otimes_H C
\end{align}
(see the proof of \cite[Proposition 15.6]{BDK01}).
For any $c\in \mathbb{F}$, we have the \emph{Poisson pseudoalgebra of type $K$}:
\begin{align*}
    P^c_K := S(He \oplus \mathbb{F}C) / S(He \oplus \mathbb{F}C)(C-c) \cong S(He),
\end{align*}
with the pseudobracket \eqref{kmbracket} with $C=c$ on the generator $e$, extended uniquely by $H$-bilinearity \eqref{hbil} and the iterated Leibniz rule \eqref{iterleib}.
\end{example}

\begin{remark}
The central extensions of the Lie pseudoalgebras $W(\mathfrak{d})$ and $K(\mathfrak{d},\theta)$, described in Examples \ref{wmppa} and \ref{kppa}, are trivial
due to \cite[Theorem 15.2]{BDK01}. Nevertheless, it is still interesting to consider the corresponding Poisson pseudoalgebras $P^\beta_W$ and $P^c_K$, 
because any---trivial or not---central extension gives rise to \emph{compatible} Poisson pseudobrackets, which are essential for applications to bi-Hamiltonian systems (see e.g.\ \cite{BDK09}). Recall that, by definition, two Poisson pseudobrackets are called compatible if any linear combination of them is also a Poisson pseudobracket.
\end{remark}

\subsection{Cohomology of Poisson pseudoalgebras}\label{Poisson pseudoalgebra4coho}
In this section, we define two types of cohomology of Poisson pseudoalgebras. The first one is called the classical cohomology and the second is called the variational cohomology. For Poisson vertex algebras, these two types of cohomology have been defined and studied in \cite{DK13,BDHK19}. In particular, it was proved in \cite{BDHKV21} that
they are isomorphic when the Poisson vertex algebra, viewed as a differential algebra, is a finitely-generated algebra of differential polynomials.

Following the notation of Sect.\ \ref{Poisson pseudoalgebraConstruction}, for a left $H$-module $V$, an odd element $X \in W_{H,1}^{cl}(\Pi V)$ such that $[X, X] = 0$ defines on $V$ the structure of a Poisson $H$-pseudoalgebra, as in Theorem \ref{operadPoisson pseudoalgebra}. Since $\mathrm{ad}_X$ is odd, we have 
\begin{equation}\label{adx2}
2\,\mathrm{ad}_X^2 = [\mathrm{ad}_X,\mathrm{ad}_X] = \mathrm{ad}_{[X,X]} = 0.
\end{equation}
Thus, we obtain a cohomology complex 
$(W_H^{cl}(\Pi V), \mathrm{ad}_X)$,
whose cohomology is called the \emph{classical cohomology} of the Poisson pseudoalgebra $V$. 

The variational cohomology of $V$ can be defined following \cite[Sect.\ 11]{BDHK19} in the Poisson vertex algebra case.  
In order to do that, let us consider the operad $\mathcal{P}_H^*(\Pi V)$, which corresponds to the Lie $H$-pseudoalgebra structures on $V$
(cf.\ (\ref{p*hn}) and \cite{BDK01}):
\begin{align*}
\mathcal{P}_H^*(\Pi V)(n) = \mathrm{Hom}_{H^{\otimes n}} ((\Pi V)^{\otimes n}, H^{\otimes n} \otimes_H (\Pi V)) \,.
\end{align*}
This operad can be obtained by restricting $\mathcal{P}_H^{cl}(\Pi V)$ to graphs with no edges (cf.\ Remark \ref{pstarh}). 
Let 
\begin{align*}
    W^*_{H}(\Pi V) = W(\mathcal{P}^*_H(\Pi V)) = \bigoplus_{n \geq -1}W^*_{H,n}(\Pi V)
\end{align*} 
denote the universal Lie superalgebra associated to the operad $\mathcal{P}_H^*(\Pi V)$. 
Similarly to Theorem \ref{operadPoisson pseudoalgebra}, one can prove the following theorem, which is essentially the same as in the $\mathcal{C}hom$ case \cite{DK13, BDHK19}.

\begin{theorem}\label{operadLie pseudoalgebra}
 There is a bijective correspondence between odd elements $X^* \in W_{H,1}^{*}(\Pi V)$ 
 such that $[X^*,X^*]=0$ and Lie pseudoalgebra structures on $V$, given explicitly by:
\begin{align}\label{lpabracket}
[a*b] = (-1)^{p(a)}X^*(a \otimes b), \qquad a,b \in V.
\end{align}
\end{theorem}
Comparing (\ref{lpabracket}) with (\ref{ppabrackets}), we have
$
    X^* = X\hspace{1mm}^{\begin{tikzpicture}
\tikzset{>={Latex[width=2mm,length=1mm]}}
\tikzstyle{vertex} =[circle, fill = black,inner sep=0pt,minimum size=1mm,scale=1.2]
\node[vertex](v1) at (1,0){};
\node[vertex](v2) at (1.35,0){};
\end{tikzpicture}}
$
in the case when $V$ is a Poisson pseudoalgebra.
For a Poisson pseudoalgebra $V$, we define
\begin{align*}
    W^{*,L}_{H}(\Pi V) = \bigoplus_{n \geq -1} W^{*,L}_{H,n}(\Pi V)
\end{align*}
to be the subspace of $W^*_{H}(\Pi V)$ consisting of all elements satisfying the following Leibniz rules:
\begin{align*}
    Y&(a_1 \otimes \dots \otimes b_ia_i \otimes \dots \otimes a_n) \notag \\  
    & = (-1)^{\bar{p}(b_i)(\bar{p}(Y)+\bar{p}(a_1)+ \dots +\bar{p}(a_{i-1}))} \, b_i \cdot_i Y(a_1 \otimes \dots \otimes a_i \otimes \dots \otimes a_n) \notag \\ 
    &+ (-1)^{\bar{p}(a_i)(\bar{p}(b_i)+\bar{p}(Y)+\bar{p}(a_1)+ \dots +\bar{p}(a_{i-1}))} \, a_i \cdot_i Y(a_1 \otimes \dots \otimes b_i \otimes \dots \otimes a_n),
\end{align*} 
for all $i=1, \dots, n$ and $a_1, \dots, a_n, b_i \in \Pi V$. Above we use the notation
\begin{align*}
    b &\cdot_i A = \sum_{j}(g_{j1} \otimes \dots \otimes g_{ji(1)} \otimes \dots \otimes g_{jn}) \otimes_H (g_{ji(-2)}b)v_j, 
\end{align*}
for any $b\in V$ and $A\in H^{\otimes n} \otimes_H V$, written in the form
\begin{align*}
    A = \sum_j (g_{j1} \otimes \dots \otimes g_{jn}) \otimes_H v_j \,, \qquad v_j\in V. 
\end{align*}

\begin{example}
One has:
\begin{align*}
    W^{*,L}_{H,-1}(\Pi V) &= \Pi V/H_+(\Pi V) = W^*_{H,-1}(\Pi V), \\
    W^{*,L}_{H,0}(\Pi V) &= \mathrm{Der}_H(\Pi V) \subset \mathrm{End}_H(\Pi V) = W^*_{H,0}(\Pi V),
\end{align*}
where $\mathrm{Der}_H(\Pi V)$ is the set of all $Y \in\mathrm{End}_H(\Pi V)$ such that
\begin{align*}
    Y(ab) = Y(a) b + (-1)^{\bar{p}(a)\bar{p}(Y)}a \, Y(b) 
\end{align*}
for all $a,b \in \Pi V$.
\end{example}

We will prove below that $W^{*,L}_{H}(\Pi V)$ is a subalgebra of the Lie superalgebra $W^*_{H}(\Pi V)$.
Note that
\begin{align*}
X^* = X\hspace{1mm}^{\begin{tikzpicture}
\tikzset{>={Latex[width=2mm,length=1mm]}}
\tikzstyle{vertex} =[circle, fill = black,inner sep=0pt,minimum size=1mm,scale=1.2]
\node[vertex](v1) at (1,0){};
\node[vertex](v2) at (1.35,0){};
\end{tikzpicture}}
\in W^{*,L}_{H,1}(\Pi V),
\end{align*}
due to the left and right Leibniz rules \eqref{leibniz}, \eqref{leibnizr}.
Since $X^*$ is odd and $[X^*,X^*]=0$, it follows that $\mathrm{ad}_{X^*}^2=0$ (cf.\ \eqref{adx2}).
The cohomology of the complex $(W^{*,L}_H(\Pi V), \mathrm{ad}_{X^*})$ 
is called the \emph{variational cohomology} of the Poisson pseudoalgebra $V$. 

We note that the operad $\mathcal{P}^{cl}_H(\Pi V)$ is graded:
\begin{align}\label{pclh_grading}
    \mathcal{P}^{cl}_H(\Pi V) = \bigoplus_{r\ge0}\mathrm{gr}^r\,\mathcal{P}^{cl}_H(\Pi V),
\end{align}
where $\mathrm{gr}^r \,\mathcal{P}^{cl}_H(\Pi V)$ consists of all elements $Y$ such that
\begin{align*}
    Y^\Gamma = 0 \quad \text{whenever} \quad |E(\Gamma)| \neq r,
\end{align*}
and $|E(\Gamma)|$ denotes the number of edges of the graph $\Gamma$. Then 
\begin{align*}
    \mathrm{gr}^r \,\mathcal{P}^{cl}_H(\Pi V) \circ_k \mathrm{gr}^s \,\mathcal{P}^{cl}_H(\Pi V) \subset \mathrm{gr}^{r+s} \,\mathcal{P}^{cl}_H(\Pi V),
\end{align*}
due to Lemma \ref{edgecorrespondence}. 
The grading \eqref{pclh_grading} induces a Lie superalgebra grading
\begin{align*}
    W^{cl}_H(\Pi V) = \bigoplus_{r\ge0} \mathrm{gr}^r \,W^{cl}_H(\Pi V),
\end{align*}
so that
\begin{align*}
    \bigl[\mathrm{gr}^r \,W^{cl}_H(\Pi V), \mathrm{gr}^s \,W^{cl}_H(\Pi V)\bigr] \subset \mathrm{gr}^{r+s} \,W^{cl}_H(\Pi V).
\end{align*}

Any $f \in W^{cl}_{H,n}(\Pi V)$ can be written as a finite sum
\begin{align*}
    f = \sum_{r = 0}^n f_r \,,
\qquad\text{where}\quad
    f_r^{\Gamma} = \begin{cases} 
      f^{\Gamma} \,, & \text{if } \;  |E(\Gamma)| = r, \\
      0 \,,& \text{otherwise.} 
   \end{cases}
\end{align*}
In particular, for an odd element $X \in W_{H,1}^{cl}(\Pi V)$ such that $[X, X] = 0$, one can write:
\begin{align*}
    X = X_0 + X_1,
\end{align*}
with $X_0$, $X_1$ homogeneous of degree $0$ and $1$, respectively. Moreover,
\begin{align*}
    [X_0, X_0] = [X_1, X_1] = [X_0, X_1] = 0,
\end{align*}
which imply (cf.\ \eqref{adx2}):
\begin{align*}
\mathrm{ad}_{X_0}^2 = \mathrm{ad}_{X_1}^2 = 0, \qquad
\mathrm{ad}_{X_0}\mathrm{ad}_{X_1} = -\mathrm{ad}_{X_1}\mathrm{ad}_{X_0}.
\end{align*}
The following result can be obtained using essentially the same proof as in the Poisson vertex algebra case (see \cite[Lemmas 11.2, 11.3]{BDHK19}).
\begin{lemma}\label{variationaldiff}
    There exists a natural Lie superalgebra isomorphism
        \begin{align}\label{isophi}
            \phi\colon W^*_{H}(\Pi V) \rightarrow \mathrm{gr}^0 W^{cl}_H(\Pi V), \quad \phi(f^*) = f,
        \end{align}   
    where $f$ is defined by
        \begin{align*}
            f^{\Gamma} = \begin{cases} 
                                f^* \,, & \text{if } \;  |E(\Gamma)| = 0, \\
                                0 \,,& \text{if } \;  |E(\Gamma)| > 0. 
                         \end{cases}
        \end{align*}
    Moreover, for $f^* \in W^{*}_{H}(\Pi V)$, $f=\phi(f^*)$, and an odd element $X \in W_{H,1}^{cl}(\Pi V)$ such that $[X, X] = 0$, we have$:$
    \begin{enumerate}
        \item $\mathrm{ad}_X f = 0  \;\;\Leftrightarrow\;\; \mathrm{ad}_{X_0} f = \mathrm{ad}_{X_1} f = 0;$
        \item $\mathrm{ad}_{X_0} f = 0 \;\Leftrightarrow\; \mathrm{ad}_{X^*} f^* = 0;$
        \item $\mathrm{ad}_{X_1} f = 0 \;\Leftrightarrow\; f^* \in W^{*,L}_{H}(\Pi V).$
    \end{enumerate}
    As a consequence,
        \begin{align*}
            f \in \mathrm{Ker} (\mathrm{ad}_X) \;\Leftrightarrow\; f^* \in \mathrm{Ker}\bigl(\mathrm{ad}_{X^*}\big|_{W^{*,L}_H(\Pi V)}\bigr).
        \end{align*}
\end{lemma}

The following theorem follows directly from Lemma \ref{variationaldiff}. Its proof is the same as in the Poisson vertex algebra case in \cite[Sect.\ 11.2]{BDHK19}.
\begin{theorem}
    The isomorphism $\phi$ in \eqref{isophi} restricts to an embedding of Lie superalgebras $W^{*,L}_H(\Pi V) \hookrightarrow W^{cl}_H(\Pi V)$.
    This gives an embedding of the variational complex $(W^{*,L}_H(\Pi V), \mathrm{ad}_{X^*})$ as a subcomplex of the classical complex $(W^{cl}_H(\Pi V), \mathrm{ad}_X)$,
    and of the variational cohomology into the classical cohomology.
\end{theorem}

\end{document}